\documentclass[11pt]{article}

\usepackage{graphicx}
\usepackage[margin=1in]{geometry}
\usepackage{amsmath}
\usepackage{amssymb}
\usepackage{amsthm}
\usepackage{amsfonts}
\usepackage{comment}
\usepackage{esint}
\usepackage[export]{adjustbox}
\usepackage{graphicx}
\usepackage[normalem]{ulem}
\usepackage{enumerate}
\usepackage[scriptsize,hang,raggedright]{subfigure}
\usepackage[cmyk,usenames,dvipsnames]{xcolor}
\usepackage{mathtools}
\usepackage[titletoc]{appendix}

\usepackage{hyperref}
\usepackage{pdfsync}
\usepackage{dsfont}
\usepackage{color}

\usepackage{cite}
\usepackage{bbm}
\usepackage{multirow}
\usepackage[font=footnotesize,labelfont=bf]{caption}
\usepackage{bibentry}
\nobibliography*

\newcommand{\vp}{\varphi}
\newcommand{\rr}{\ensuremath{\mathbb{R}}}
\newcommand{\Rd}{\ensuremath{{\mathbb{R}^d}}}

\newcommand{\ep}{\ensuremath{\varepsilon}}
\newcommand{\bgamma}{\boldsymbol{\gamma}}

\newcommand{\dom}{{\rm dom}}

\newcommand{\norm}[1]{\lVert #1 \rVert}

\newcommand{\RR}{\mathbb{R}}

\newcommand{\PP}{\mathcal P}

\let\div\relax
\DeclareMathOperator{\div}{div}

\DeclareMathOperator{\spt}{\mathop{\rm supp\,}}
\DeclareMathOperator{\tr}{tr}

\newcommand{\loc}{{\rm loc}}

\DeclareSymbolFont{bbold}{U}{bbold}{m}{n}
\DeclareSymbolFontAlphabet{\mathbbold}{bbold}
\newcommand{\id}{{\rm \bf id}}

\def\P{{\mathcal P}}

\newcommand{\R}{{\mathord{\mathbb R}}}

\newcommand{\N}{{\mathord{\mathbb N}}}
\newcommand{\supp}{{\mathop{\rm supp\, }}}

\newcommand{\F}{\mathcal{F}}

\newcommand{\bt}{\mathbf{t}}

\newcommand{\cF}{\mathcal{F}}

\newcommand{\mue}{\mu_{\epsilon}}
\newcommand{\qe}{q_{\epsilon}}
\newcommand{\phie}{\varphi_\ep}

\renewcommand{\:}{\colon}

\def\P{{\mathcal P}}
\def\S{{\mathcal S}}

\def\epsilon{\varepsilon}

\def\F{\mathcal{F}}

\DeclareMathOperator{\interior}{int}

\newcommand{\be}{\begin{equation}}
\newcommand{\ee}{\end{equation}}
\newcommand{\bes}{\begin{equation*}}
\newcommand{\ees}{\end{equation*}}

\newtheorem{thm}{Theorem}[section]

\newtheorem{cor}[thm]{Corollary}
\newtheorem{prop}[thm]{Proposition}
\newtheorem{lem}[thm]{Lemma}

\newtheorem*{as*}{\assumptionnumber}
\providecommand{\assumptionnumber}{}
\makeatletter
\newenvironment{as}[2]
 {%
  \renewcommand{\assumptionnumber}{Assumption #1}%
  \begin{as*}%
  \protected@edef\@currentlabel{#1}({#2})%
 }
 {%
  \end{as*}
 }
\makeatother

\theoremstyle{definition}

\newtheorem{defn}[thm]{Definition}

\theoremstyle{remark}
\newtheorem{rem}[thm]{Remark}

\numberwithin{equation}{section}

\title{Nonlocal Approximation of Slow and Fast Diffusion}

\author{Katy Craig, Matt Jacobs, and Olga Turanova}

\date{\today}                                        
\begin{document}

\maketitle

\begin{abstract}
Motivated by recent work on approximation of diffusion equations by deterministic interacting particle systems, we develop a  nonlocal approximation for a range of linear and nonlinear diffusion equations and prove  convergence of the method in the slow, linear, and fast diffusion regimes. A key ingredient of our approach is a novel technique for using the 2-Wasserstein and dual Sobolev gradient flow structures of the diffusion equations to recover the duality relation characterizing the pressure in the nonlocal-to-local limit. Due to the   general class of internal energy densities that our method is able to handle, a byproduct of our result is a novel particle method for sampling a wide range of probability measures, which extends classical approaches based on the Fokker-Planck equation beyond the log-concave setting.
\end{abstract}

%\tableofcontents

 \section{Introduction}
Linear and nonlinear diffusion equations arise throughout the sciences, including in mathematical biology, fluid mechanics, control theory, and sampling. Key examples are the \emph{fast diffusion equation} ($m<1$), \emph{heat equation} ($m=1$), and \emph{porous medium equation} ($m>1$), 
\begin{align*}
\partial_t \rho = \Delta \rho^m , \quad m > 1-\frac{2}{d+2} \ ,
\end{align*}
as well as general nonlinear diffusion equations, induced by a proper, convex, lower semicontinuous \emph{internal energy density} $f: [0,+\infty) \to \R \cup \{+\infty\}$ and velocity $v \in  L^1_\loc([0,+\infty);W^{1,\infty}(\Rd))$,
\begin{align} \label{generalnonlinear}
\partial_t \rho = \Delta f^*(p) + \nabla \cdot (\rho v) , \quad p \in \partial f(\rho).
\end{align}
 See foundational work by  Vazquez \cite{ VazquezPMEbook, VazquezFDEbook} and Otto \cite{otto_pme_geometry}  for  the classical theory of the porous medium and fast diffusion equations, as well as more recent works  \cite{kwon2021degenerate,jacobs20201,jacobs2021darcy,de2016bv} for general nonlinear diffusion equations, including their connection to \emph{sandpile models} and \emph{starvation driven diffusion} \cite{bantay1992avalanche, cho2013starvation}.
One interesting example is the choice $f = \iota_{[0,1]}$, where $ \iota_{[0,1]}$ denotes the convex characteristic function that is zero on $[0,1]$ and $+\infty$ elsewhere. In this case, equation (\ref{generalnonlinear}) is a \emph{height constrained transport equation} \cite{AKY, craig2018aggregation, MaRoSa10, MaRoSaVe11}.

In recent years, motivated by applications in numerical analysis, sampling, and models of two-layer neural networks, there has been significant interest in developing deterministic interacting particle systems that approximate solutions of the above diffusion equations;  see section \ref{literature} below, for a discussion of the existing literature. At the heart of each of these methods is an approximation of the (local) diffusion equation by nonlocal PDEs that have smooth enough velocity fields to admit a deterministic particle method. While great progress has been made by Carrillo, Esposito, and Wu \cite{carrillo2023nonlocal} in the case of the porous medium equation for $m >1$, as well as for general internal energy densities $f$ that possess similar growth properties, the goal of the present work is to devise a nonlocal approximation of general nonlinear diffusion equations that converges in essentially all cases in which the nonlinear diffusion equation is well-posed.  This includes three very difficult cases that have eluded   previous attempts in the literature:
\begin{itemize}
    \item  the heat equation, $f(s) = s\log s - s$,
    \item  fast diffusion equations, $f(s) =\frac{1}{m-1} s^m$ for $m\in \left(1-\frac{2}{d+2},1\right)$,\footnote{In the fast diffusion case, the restriction $m> 1-2/(d+2)$ is necessary for the well-posedness of the PDE, rather than a limitation of our method; see \cite[Example 9.3.6]{ambrosiogiglisavare}.}
    \item  height constrained transport,   $f(s) = \iota_{[0,1]}(s)$.
\end{itemize}

The ability of our approach to approximate such a general class of nonlinear diffusion equations is particularly significant from the perspective of applications in sampling. As we describe in Remark \ref{samplingremark} below,  nonlinear diffusion equations of the form we consider can be chosen to converge exponentially quickly in time to any probability measure of the form
 \begin{align} \label{barrhoeqn}
 \bar{\rho} =  \max \{ h(Z-V),0 \}  , \  h:\R \to \R \text{ strictly increasing, } V: \Rd \to \R \text{ strongly convex,}
 \end{align}
 as long as $h$ satisfies appropriate regularity and growth hypotheses and its primitive satisfies McCann's convexity condition.
 The normalization constant $Z \in \R$ is chosen so that $ \int \bar{\rho} = 1$.  
  Consequently, particle methods  based on our nonlocal approximation are an important step forward to developing provable sampling methods for $\bar{\rho}$ of the form (\ref{barrhoeqn}). This would be significant, as very little has been rigorously proved for sampling measures $\bar{\rho}$ that are not log-concave.

Our approach is able to cover a general class of diffusion equations by leveraging convex duality. This allows us to consider any internal energy density $f: [0,+\infty) \to \R \cup \{+\infty\}$ that is  convex, lower semicontinuous, and induces a $W_2$-coercive energy functional (see Assumption \ref{internalas}(\ref{moment_control})), with $f(0) = 0$ and ${\rm int}(\dom(f)) \neq \emptyset$. Given  a velocity $v \in  L^1_\loc([0,+\infty);W^{1,\infty}(\Rd))$, we consider solutions $\rho \in AC^2_\loc([0,+\infty);\P_2(\Rd))$ of the nonlinear diffusion equation 
\begin{align} \tag{PDE} \label{PDE}
\begin{cases}
\partial_t \rho - \nabla \cdot ( \nabla f^*(p)  + v \rho) = 0,  \qquad \qquad \text{ in duality with $C^\infty_c( (0,+\infty) \times \Rd),$}  \\
 p \in \partial f(\rho), \text{ for a.e. $(t,x)$} \\
\rho(0, \cdot) = \rho^0 ,
\end{cases}
\end{align}
where $f^*(p) \in L^1_\loc([0,+\infty); W^{1,1}(\Rd))$.  (See Assumption \ref{velocityas} for our precise, more general hypotheses on the velocity.)  Since the Fenchel-Young identity ensures 
\[ p \in \partial f(\rho) \iff \rho p = f(\rho) + f^*(p), \]
 we refer to the  condition $p \in \partial f (\rho)$ in (\ref{PDE}) as the \emph{duality relation} characterizing the pressure $p$.

 In the absence of an external velocity field, i.e. $v=0$, (\ref{PDE}) has a formal Wasserstein gradient flow structure with respect to the  \emph{internal energy}. For any finite Borel measure $\rho \in \mathcal{M}(\Rd)$ with finite second moment that is absolutely continuous with respect to Lebesgue measure, $d\rho(x) = \rho(x) dx$, the internal energy is defined by \begin{align} \label{internalenergydef}
\F(\rho) =  \int f(\rho(x)) dx .
\end{align}
More generally, we extend $\F(\rho)$ to be defined for any $\rho \in \mathcal{M}(\Rd)$ with finite second moment via $W_2$-lower semicontinuity; see \cite[equation (9.3.10)]{ambrosiogiglisavare}.

Inspired by previous works on deterministic particle methods for slow diffusion equations, we aim to approximate (\ref{PDE}) by considering a nonlocal regularization of the associated internal energy $\F$. In the absence of a velocity field, this corresponds to approximating gradient flows of (\ref{internalenergydef}) by gradient flows of the following regularized energy:
\begin{align} \label{eq:eps energy}
\F_\ep(\rho) &= \int f_\ep(\varphi_\ep *\rho)  \quad \text{ for } \quad f_\epsilon(a) = \begin{cases}\frac{\delta(\epsilon)}{2} |a|^2 +{}^{\delta(\epsilon)}f(a) - {}^{\delta(\epsilon)}f(0) & \text{ for }a\geq 0,\\
+\infty &\text{ for } a<0. \end{cases}
\end{align}
In this definition, $\epsilon \mapsto \delta(\epsilon)$ is a non-decreasing function of $\epsilon$ with $\lim_{\ep\rightarrow 0}\delta(\ep)=0$,     $\varphi_\epsilon$ is a mollifier, chosen to smooth out potential singularities in the measure $\rho$ (for instance, when $\rho$ is an empirical measure), and ${}^{\delta(\epsilon)}f$ is the Moreau-Yosida regularization of $f$, chosen to smooth out singularities in the internal energy density. (See Assumption \ref{mollifieras} for the precise hypotheses on the mollifier and Section \ref{moreauenvelopesec} for basic properties of the Moreau-Yosida regularization.)  
  This   leads to the following nonlocal approximation of (\ref{PDE}): given  $v_{\epsilon}\in C^{\infty}_c([0,+\infty)\times \RR^d))$ and $\rho_\ep^0\in \P_2(\R^d)\cap C_c^{\infty}(\Rd)$, that approximate $v$ and $\rho^0$ (see Definition \ref{wellprepareddef}), we consider solutions $\rho_\epsilon \in AC^2_\loc([0,+\infty), \P_2(\Rd))$ satisfying
\begin{align} \tag{${\rm PDE}_{\epsilon}$} \label{epsiloneqn}
\begin{cases} \partial_t \rho_\epsilon - \nabla \cdot \left( \rho_\epsilon  \left( \nabla p_\epsilon  +v_\ep\right)\right) = 0   , \qquad \qquad \text{ in duality with $C^\infty_c( (0,+\infty) \times \Rd)$,}  \\
  p_\epsilon = (\varphi_\ep *f_\ep'(\varphi_\ep*\rho_\ep))  , \\
  \rho_\epsilon(\cdot,0) = \rho^0_{\epsilon}
  \end{cases}   
\end{align}

Our main result, stated in Section \ref{mainresultsection} below, is that weak solutions of (\ref{epsiloneqn}) converge to a   solution of (\ref{PDE}) in the $\epsilon\to 0$ limit.    In Section \ref{strategy}, we describe our strategy of proof, which includes several novel components, the most important of which is that we simultaneously leverage the $W_2$ and $\dot{H}^{-1}$ gradient flow structures of (\ref{PDE}) in order to obtain the duality relation $\rho p = f(\rho) + f^*(p) \iff p \in \partial f(\rho)$ in the nonlocal-to-local limit.   In Section \ref{literature}, we place our result in  context of related work on nonlocal approximations of diffusion and related applications in sampling.

\subsection{Main result}  \label{mainresultsection}
In order to present our main result, we begin by  stating our assumptions on the internal energy density $f$, velocity $v$, mollifier $\varphi_\epsilon$, and initial data $\rho^0$.
\begin{as}{E}{Energy density} \label{internalas} \ \\
Assume the internal energy density $f: \R \to \R \cup \{+\infty\}$ satisfies the following properties:
\begin{enumerate}[(i)]
\item $f$ is  convex and lower semicontinuous,  with   $f(0)=0$ and $f(a) = +\infty$ for   $a <0$. \label{convexlctsas}
\item $\interior(\dom(f)) \neq \emptyset$.  %There exists $a_0>0$ so that $f(a_0)\neq +\infty$.
\label{nontrivial_convex}
\item \label{moment_control} There exists an increasing concave function $H:[0,+\infty)\to[0,+\infty)$ with $H(0)=0$ such that 
\begin{align}
-H\left( \int_{\RR^d} (1+|x|^2)d\rho(x)\right)  \leq  \F(\rho)  \;  \label{seconddersing}
\end{align}
for all finite   Borel measures $\rho \in \mathcal{M}(\Rd)$ with finite second moment, where $\F(\rho)$ is as defined in equation (\ref{internalenergydef}). 
\end{enumerate}
\end{as}

\begin{rem} \label{coercivityremark}
Condition (\ref{moment_control}) ensures that the negative part of $f(\rho)$ is integrable  on $\P_2(\Rd)$; see \cite[Remark 9.3.7]{ambrosiogiglisavare}. This type of hypothesis is particularly natural in the present setting, where it can be thought of as a coercivity requirement for the energy functional $\F$. In particular, when $v \equiv 0$ and (\ref{PDE}) has a formal gradient flow structure with respect to the 2-Wasserstein metric,   it is reasonable to require that the negative part of $\F(\rho)$ doesn't decay too fast with respect to $-M_2(\rho) = -\int|x|^2 d \rho(x)$, or else the discrete-in-time JKO formulation of the dynamics could fail to have a solution, i.e. 
\[
\inf_{\rho\in \PP_2(\RR^d)} \F(\rho)+\frac{1}{2t}W_2^2(\rho, \rho^0)=-\infty \ , \quad \forall t >0 .
\]
\end{rem}

\begin{rem}
    In Lemma \ref{lem:example}, we verify that Assumption \ref{internalas} is satisfied by the internal energy densities corresponding to the  heat equation, fast diffusion equations, and height constrained transport.
\end{rem}

\begin{as}{V}{Velocity field} \label{velocityas}
Assume the velocity field $v:[0,+\infty)\times\RR^d\to\RR^d$  satisfies the following hypotheses:
 \begin{enumerate}[(i)]
 \item  $v \in  L^1_\loc([0,+\infty),W^{1,1}_\loc(\Rd))$.
 \item $ \displaystyle \frac{|v(t,x)|^2+(\nabla \cdot v(t,x))_+}{1+|x|^2}\in L^{1}_\loc([0,+\infty);L^{\infty}(\RR^d))$, where $(\nabla \cdot v)_+=\max(\nabla \cdot v,0)$.
 \end{enumerate}
\end{as}

\begin{as}{M}{Mollifier} \label{mollifieras}
Consider a mollifier $\varphi_\ep(x) = \ep^{-d} \varphi(x/\ep)$, $\ep >0$, where $\varphi:\Rd \to \R$ satisfies the following assumptions:
\begin{enumerate}[(i)]
\item $\varphi \in C^2(\Rd)$, $\varphi \geq 0$, $\int \varphi(x) dx =1$, $\varphi(x) = \varphi(-x)$. \label{firstmolas}
\item There exists $C_\varphi>0$ and  $r > \max\{d,2\}$ so that  $ \vp(x) \leq C_\varphi |x|^{-r}$.\label{secondmolas}  
\item $\varphi , \nabla \varphi, D^2 \varphi \in L^1(\Rd) \cap L^\infty(\Rd)$. \label{molLinftybounds}
\end{enumerate}
\end{as}

\begin{as}{I}{Initial data}
\label{as:id}
Suppose the initial data $\rho^0\in \P_2(\Rd)$ satisfies
\[
\mathcal{F}(\rho^0) + \mathcal{S}(\rho^0)<+\infty ,
\]
for the internal energy $\F$   defined as in equation (\ref{internalenergydef}) and the entropy $\S$ given by
\begin{align} \label{entropydef}
\S(\rho) =  \begin{cases} \int  \rho(x) \log(\rho(x)) dx &\text{ if } d \rho(x) = \rho(x) dx , \\ +\infty &\text{otherwise.} \end{cases}
\end{align}
\end{as}
Finally, we will consider the limiting properties of solutions of (\ref{epsiloneqn}) with \emph{well-prepared} initial data and velocity field, in the following sense.

\begin{defn}[Well-prepared] \label{wellprepareddef}
Given an internal energy density $f$,  a velocity $v$, a mollifier $\varphi_\epsilon$, and initial data $\rho^0$ satisfying     Assumptions \ref{internalas}, \ref{velocityas}, \ref{mollifieras}, and \ref{as:id},  we say that a sequence of initial data $\{\rho_\epsilon^0\}_{\epsilon\in (0,1)} \subseteq \P_2(\Rd)\cap C^{\infty}_c(\RR^d)$ and velocity fields $\{v_{\ep}\}_{\epsilon \in (0,1)} \subseteq   C^{\infty}_c([0,+\infty)\times\RR^d)$ is \emph{well-prepared} if the following conditions are met:
\begin{enumerate}[(i)]
\item For any $T>0$, we have the uniform bounds \begin{equation}\label{eq:data_uniform_bounds}
 \sup_{\epsilon>0} \mathcal{F}_{\ep}(\rho^0_{\ep})+\mathcal{S}(\rho^0_{\ep})+M_2(\rho^0_{\ep})+    \int_0^T \left\| \frac{|v_{\ep}(t,\cdot)|^2}{1+|\cdot|^2}+\frac{(\nabla \cdot v_{\ep}(t,\cdot))_+}{1+|\cdot|^2} \right\|_{L^\infty(\Rd)}  <+\infty,
\end{equation}
where $(\nabla \cdot v_{\ep})_+=\max(\nabla \cdot v_{\ep},0)$.
\item $v_{\epsilon} \to v$ in $L^1_{\loc}([0,+\infty)\times\RR^d)$ and  $\rho_{\ep}^0 \to \rho^0$ in $W_2$.
\end{enumerate}

\end{defn}

It can be shown using standard arguments that, for any velocity and initial data satisfying our assumptions,  well-prepared sequences $\{\rho_\epsilon^0\}_{\epsilon\in (0,1)} $ and $\{v_{\ep}\}_{\ep \in(0,1)}$ exist; see Lemma \ref{wellpreparedlemma}. 
Likewise, in  Lemma \ref{wellposedpdeepslem}, we show that for any $\rho_\ep^0 \in \P_2(\Rd)$ and $v_\ep \in C_b([0,+\infty; W^{1,\infty}(\Rd))$, (\ref{epsiloneqn}) is well-posed.

We now state our main theorem, which proves that, for \emph{any} choice of well-prepared data, solutions of (\ref{epsiloneqn}) converge to a solution of (\ref{PDE}). 

\begin{thm} \label{mainthm}
Consider an internal energy density $f$,  a velocity $v$, a mollifier $\varphi_\epsilon$, and initial data $\rho^0$ satisfying     Assumptions \ref{internalas}, \ref{velocityas}, \ref{mollifieras}, and \ref{as:id}. For
any well-prepared sequences $\{\rho_\ep^0\}_{\ep\in(0,1)}$, $\{v_\ep\}_{\ep\in(0,1)}$, let $\rho_\epsilon \in AC^2_\loc([0,+\infty), \P_2(\Rd))$ be the solution of \eqref{epsiloneqn} with initial data $\rho^0_\epsilon$ and velocity $v_\ep$. Let $r$ be the exponent from Assumption \ref{mollifieras}.  Suppose
 $\delta(\epsilon)$ goes to zero sufficiently slowly such that, for some $\theta \in (0,1)$ 
\begin{align} \label{deltaepshyp}
 \lim_{\epsilon \to 0}  \epsilon^{\frac{\theta(r-d)}{r}}\delta(\epsilon)^{-\frac{r-\theta}{r}} = 0 . 
 \end{align}
 Then there exists $\rho \in AC^2_\loc([0,+\infty); \P_2(\Rd)) $ so that, up to a subsequence, $\rho_\epsilon(t) \to \rho(t)$ in weak $L^1_\loc(\Rd)$ and 1-Wasserstein for all $t \geq 0$, and $\rho$ is a solution of \eqref{PDE} with initial data $\rho^0$.
\end{thm}

\begin{rem}[Decay of $\delta(\epsilon)$]
For example, equation (\ref{deltaepshyp}) holds if $\delta(\epsilon) =\epsilon^{\beta}$ for some $\beta\in (0, \frac{ r-d}{r-1})$.
\end{rem}

\begin{rem}[Passing to a subsequence]
The fact that Theorem \ref{mainthm} only shows convergence of solutions of (\ref{epsiloneqn}) to (\ref{PDE}) up to a subsequence is due to the generality of our hypotheses. In particular, under these hypotheses, it remains open whether a solution of (\ref{PDE}) with initial data $\rho^0$ is unique, and it is in principle possible that different subsequences of solutions of (\ref{epsiloneqn}) could converge to different solutions of (\ref{PDE}).

On the other hand, in the cases where it is known that the solution $\rho$ of (\ref{PDE}) with initial data $\rho^0$ is unique, Theorem \ref{mainthm} shows that every sequence $\rho_\ep$ has a further subsequence that converges to $\rho$, hence $\rho_\ep$ itself must converge to $\rho$.
\end{rem}

%
%
%\begin{rem}[Existence]
%A consequence of Theorem \ref{mainthm} is that, under our assumptions, there exists a solution of (\ref{PDE}). Due to the generality Assumptions \ref{internalas} and \ref{velocityas}, this covers cases for which existence was not previously known. See, for example, \cite{kwon2021degenerate}, where well-posedness for (\ref{PDE}) is established for $f$ is continuous on $[0,+\infty)$, twice differentiable on $(0,+\infty)\setminus\{1\}$, and  satisfies a coercivity condition at 0 that's more restrictive than ours. 
%\end{rem}

Our regularized equation (\ref{epsiloneqn}) gives rise to a deterministic particle method for (\ref{PDE}) in the following manner. First, due to the fact that the velocity field in (\ref{epsiloneqn}), $\nabla p_\epsilon + v_\epsilon$, is globally Lipschitz for all $\epsilon >0$, if the initial conditions of (\ref{epsiloneqn}) are given by an empirical measure,
then the corresponding solution of (\ref{epsiloneqn}) will remain an empirical measure for all time,
\begin{align} \label{rhoepNdef}  {\rho}^0_{\epsilon} := \frac{1}{N} \sum_{i=1}^N \delta_{x^0_{\epsilon,i}}  \implies \rho_\epsilon(t)  = \frac{1}{N} \sum_{i=1}^N \delta_{x_{\epsilon,i}(t)}  , 
\end{align}
where the trajectories $\{x_{\epsilon,i}(t)\}$ are given by solutions of the ordinary differential equation
\begin{align}\label{epNode}
\begin{cases} \dot{x}_{\epsilon,i}(t) &= - \nabla p_\epsilon(x_{\epsilon,i}(t),t) - v_\epsilon(x_{\epsilon,i}(t),t) , \ \forall  t \geq 0 ,  \\
x_{\epsilon,i}(0) &= x_{\epsilon,i}^0 
\end{cases}  \end{align}
with
\begin{align*}
 p_\epsilon(x,t) = \left( \varphi_\ep *f_\ep' \Big( \frac{1}{N} \sum_{j=1}^N \varphi_\ep(x_{\epsilon,i}(t)-\cdot) \Big)\right)(x) .
\end{align*}
%While our main convergence result, Theorem \ref{mainthm}, only considers initial conditions for (\ref{epsiloneqn}) that have bounded entropy, excluding empirical measure initial data, the following Corollary uses stability of (\ref{epsiloneqn}) to obtain convergence as $\epsilon \to 0$ and the particle initial data approximates   well-prepared initial data sufficiently quickly.
Our main convergence result, Theorem \ref{mainthm}, excludes empirical measure initial data: indeed, initial conditions for (\ref{epsiloneqn}) are assumed to have bounded entropy. However, in the following corollary, we use stability of (\ref{epsiloneqn}) to obtain convergence as $\epsilon \to 0$ for particle initial data that approximates   well-prepared initial data sufficiently quickly.

\begin{cor} \label{particlecorollary}
Consider an internal energy density $f$,  a velocity $v$, a mollifier $\varphi_\epsilon$, and initial data $\rho^0$ satisfying     Assumptions \ref{internalas}, \ref{velocityas}, \ref{mollifieras}, and \ref{as:id}. For
any well-prepared sequences  $\{\rho_\ep^0\}_{\ep\in(0,1)}$, $\{v_\ep\}_{\ep\in(0,1)}$ and $T>0$, let $\{\hat{\rho}_{\ep}^0\}_{\ep\in(0,1)}$ be a sequence of empirical measures satisfying $W_2(\hat{\rho}_{\epsilon}^0, \rho_\ep^0) = o(  C_\epsilon^{-1}   e^{-C_\epsilon T}) $, where
\begin{align*}
 &  C_\epsilon :=2\norm{v_{\epsilon}}_{ L^\infty([0,+\infty);W^{1,\infty}(\RR^d))}+\norm{\nabla \varphi_\ep}_{W^{1,1}(\Rd)} \left(\|f'_\epsilon\|_{\rm Lip([0,+\infty))} \norm{\vp_{\epsilon}}_{L^\infty(\Rd)} +f'_{\epsilon}(0)\right) .
 \end{align*}
Let $\hat{\rho}_{\epsilon} \in AC^2_\loc([0,+\infty), \P_2(\Rd))$ be the solution of \eqref{epsiloneqn} with initial data $\hat{\rho}^0_{\epsilon}$ and velocity $v_\epsilon$, as in equations (\ref{rhoepNdef}-\ref{epNode}).
 Suppose
 $\delta(\epsilon)$ goes to zero sufficiently slowly so that (\ref{deltaepshyp}) holds for some $\theta \in (0,1)$.

  Then there exists $\rho \in AC^2_\loc([0,+\infty); \P_2(\Rd)) $ so that, up to a subsequence, $\hat{\rho}_{\epsilon}(t) \to \rho(t)$ in   1-Wasserstein for all $t \in [0,T]$, and $\rho$ is a solution of \eqref{PDE} with initial data $\rho^0$.
\end{cor}

\begin{rem}[Sampling] \label{samplingremark}
In the special case that the internal energy density $f$ is strictly convex, satisfies sufficient regularity and growth hypotheses, and satisfies  McCann's convexity condition that $s \mapsto s^df(s^{-d})$ is convex and nonincreasing on $(0,+\infty)$, then for any velocity of the form  $v(t,x) = \nabla V(x)$, for $V \in W^{1,\infty}(\Rd)$ strongly convex, (\ref{PDE}) is a  2-Wasserstein gradient flow   of  the strongly convex energy $\mathcal{E}(\rho) = \F(\rho)+ \int V \rho$ \cite[Chapter 11]{ambrosiogiglisavare}. Consequently,  we formally expect that, as $t \to +\infty$,  solutions  of (\ref{PDE})  converge exponentially quickly to 
 \begin{align} \label{barrhoeqn2}
 \bar{\rho} =  \max \{ (f^*)'(Z-V),0 \} ,
 \end{align}
 where $Z \in \R$ is a normalization constant chosen so that $\int \bar{\rho} = 1$. Combining this observation with    an argument along the lines of \cite[Corollary 1.3]{craig2022blob}, one expects that discrete particle approximations of (\ref{PDE}) indeed converge   to $\bar{\rho}$ in the long time limit, providing a method for approximating $\bar{\rho}$ by an empirical measure. We leave the question of developing a rigorous, global in time convergence theory, by which our method can be used to sample probability measures $\bar{\rho}$, to future work.
\end{rem}

\subsection{Strategy} \label{strategy}

We now describe the strategy of our proof.
The main challenge   in proving that weak solutions of (\ref{epsiloneqn}) converge to a   solution of (\ref{PDE})   lies in proving the weak convergence of the nonlinear term $\rho_{\epsilon}\nabla p_{\epsilon}$ to $\nabla f^*(p)$ and proving that the nonlocal duality relation $p_{\ep}=\vp_{\ep}*\big(f_{\ep}'(\vp_{\ep}*\rho_{\ep})\big)$ induces the local duality relation $p\in \partial f(\rho)$ in the limit.  This challenge is compounded by the fact that, in contrast to the density and pressure variables for the limiting equation (\ref{PDE}), the nonlocal density and pressure variables $\rho_{\epsilon}, p_{\epsilon}$ are not coupled by a pointwise monotone relation.  As a result, $\rho_{\ep}\nabla p_{\ep}$ is not an exact form, so one cannot move the gradient onto a test function.  More importantly, the breakdown of the monotone relation means that the nonlocal equation (\ref{epsiloneqn}) satisfies far fewer dissipation properties compared to the local equation (\ref{PDE}). Indeed, testing (\ref{PDE}) against any monotone function of $\rho$ or $p$ produces a dissipation relation, whereas for (\ref{epsiloneqn}), to the best of our knowledge, there are only two specific choices that produce useful relations.

First, by testing the equation against $p_{\ep}$, formally one has the standard \emph{energy dissipation relation}
\begin{equation}\label{eq:intro_edr}
\mathcal{F}_{\ep}(\rho_{\ep}(t))+ \int_{[0,t]\times \RR^d} \rho_{\ep}|\nabla p_{\ep}|^2+\rho_{\ep}v_{\ep}\cdot \nabla p_{\ep}=\mathcal{F}_{\ep}(\rho_{\ep}^0).
\end{equation}
 This equation is fundamental to any Wasserstein gradient flow type PDE, as it relates the dissipation of the potential energy $\mathcal{F}_{\ep}(\rho_{\ep})$ to quantities that are approximately measuring the expended kinetic energy.  Crucially this provides some control on the pressure gradients, provided that $\mathcal{F}_{\epsilon}$ can be bounded from below.
 
Second, a less-expected \emph{entropy dissipation relation} was discovered by Lions and Mas-Gallic\cite{lions2001methode} and generalized in the recent paper by Carrillo, Esposito, and Wu \cite{carrillo2023nonlocal}.    Lions and Mas-Gallic  studied nonlocal limits of (\ref{PDE}) in the special case $f(a) = f_m(a):=\frac{1}{m}a^m$ for $m=2$, while Carrillo, Esposito, and Wu studied general convex internal energy densities $f$ that can be bounded above and below by power functions $f_{m_1}, f_{m_2}$ with exponents in the range $m_1, m_2\in [1,\infty)$.    The key insight underlying the entropy dissipation relation is that, as long as the gradient of the pressure $\nabla p_{\ep}$ belongs to the  $W_2$ subdifferential of a functional,   then the divergence of the flux $\nabla \cdot (\rho_{\ep}\nabla p_{\ep})$ will produce a good term when integrated against $\log(\rho_{\ep})$. Note that this requires that the pressure-density coupling  has the structure $p_{\ep}=\vp_{\ep}*f_{\ep}'(\vp_{\ep}*\rho_{\ep})$ rather than, for instance, the choice $p_{\ep}=f'_{\ep}(\vp_{\ep}*\rho_{\ep})$.   Testing equation (\ref{epsiloneqn}) against $\log(\rho_{\ep})$ formally one has 
\begin{equation}\label{eq:intro_entropy_dr}
\mathcal{S}(\rho_{\ep}(t))+\int_{[0,t]\times\RR^d} \nabla \rho_{\ep}\cdot \nabla p_{\ep}-\rho_{\ep}\nabla \cdot v_{\ep}=\mathcal{S}(\rho_{\ep}^0)
\end{equation}
where $\mathcal{S}$ is the entropy functional from equation (\ref{entropydef}).

Since $\rho_{\ep}$ and $p_{\ep}$ have unrelated level sets, it is not immediately clear that $\nabla \rho_{\ep}\cdot \nabla p_{\ep}$ is a good term.  However, if we define
\begin{align} \label{muepdef}
\mu_\ep &:= \varphi_\ep*\rho_\ep  , \\
q_\ep &:= f_\ep'(\mu_\ep),
\label{qepdef}
\end{align}
then we can write $p_{\ep}=\vp_{\ep}*q_{\ep}$ and move the mollifier from $p_{\ep}$ to $\rho_{\ep}$ to see that
\[
\int_{[0,t]\times\RR^d} \nabla \rho_{\ep}\cdot \nabla p_{\ep}=\int_{[0,t]\times\RR^d} \nabla \mu_{\ep}\cdot \nabla q_{\ep}.
\]
Since the level sets of $\mue$ and $\qe$ are monotonically coupled through $f'_{\ep}$, it follows that $\nabla \mue\cdot \nabla \qe\geq 0$ pointwise almost everywhere.  This bound provides some additional spatial regularity for $\mue$ and $\qe$, though note that extracting a specific norm bound for either quantity is not so obvious for general $f$.

Even with the two dissipation relations (\ref{eq:intro_edr}) and (\ref{eq:intro_entropy_dr}) in hand, there is insufficient regularity to deduce the strong convergence of any of the key quantities $\rho_{\ep}, \mue, \qe, p_{\ep}$, let alone strong convergence of their derivatives, at least for general $f$.  Hence, we  need to prove the weak convergence of $\rho_{\ep}\nabla p_{\ep}$ to $\nabla f^*(p)$ and the convergence of the nonlocal density-pressure to the local density-pressure coupling without access to strong convergence. We  accomplish this in two main steps, described in detail in Section \ref{sec:convergence}, which are the heart of our arguments.

First, we  show that the difference $\rho_{\ep}\nabla p_{\ep}-\mue\nabla \qe$ converges weakly to zero.  This allows us to replace the inexact form $\rho_{\ep}\nabla p_{\ep}$ with the exact form $\mue\nabla\qe=\nabla f^*_{\ep}(\qe)$.  To show that the difference vanishes, one must control the error created from attempting to move the mollifier off of $\nabla p_{\ep}=\vp_{\ep}*\nabla \qe$ and onto $\rho_{\ep}$ to produce $\mue=\vp_{\ep}*\rho_{\ep}$.  This has been a key step in to verify convergence of solutions in the previous literature \cite{CarrilloCraigPatacchini, craig2022blob, carrillo2023nonlocal}, and has often been referred to as   the \emph{mollifier exchange} step.    In this paper, we establish the mollifier exchange in a much more general setting, requiring us to develop new arguments. The main difficulty in this step lies in providing uniform-in-$\epsilon$ integrability estimates for the quantity $\mue|\nabla \qe|$.  This is nontrivial, as we only have a priori $L^1$ spacetime control on the quantities $f_{\ep}(\mue), \rho_{\ep}|\nabla p_{\ep}|^2,$ and $\nabla \mue\cdot\nabla \qe$ (from the energy and entropy dissipation inequalities respectively). Note that the control of $\rho_{\ep}|\nabla p_{\ep}|^2$ is not of much
use, since $p_{\ep}$ is smoother than $q_{\ep}$. Instead, we must use the control from $\nabla \mue\cdot\nabla \qe$, and it is here
that our regularization $f_{\ep}$ of the original internal energy $f$ plays a crucial role. 
The key insight is that by replacing $f$ with $\frac{\delta(\epsilon)}{2} |a|^2 +{}^{\delta(\epsilon)}f(a) - {}^{\delta(\epsilon)}f(0)$, we ensure that $f_{\ep}$ is both $W^{2,\infty}$ and strongly convex on $[0,+\infty)$ (or equivalently that $f_{\ep}$ and $f_{\ep}^*$ are strongly convex on the relevant portions of their domains).  This allows us to extract much more useful information from the duality coupling of $\mue$ and $\qe$.  In particular, $\nabla \mue\cdot\nabla \qe$ will now control both $|\nabla \mue|^2$ and $|\nabla \qe|^2$ as 
\[
\nabla \mue\cdot \nabla \qe=f''_{\ep}(\mue)|\nabla \mue|^2=f_{\ep}^{*\prime\prime}(\qe)|\nabla \qe|^2,
\]
where we note that the second derivatives $f''_{\ep}(\mue)$ and $f_{\ep}^{*\prime\prime}(\qe)$ are well defined almost everywhere on the supports of $|\nabla \mue|$ and $|\nabla \qe|$ respectively and uniformly bounded from below.  While for general $f$ this lower bound will degenerate in the $\epsilon\to 0$ limit,  our arguments can handle this degeneration as long as $\delta(\epsilon)$ vanishes sufficiently slowly compared to $\epsilon$ (see Theorem \ref{mainthm} for a precise statement).

Now that we have replaced $\rho_{\ep}\nabla p_{\ep}$ by $\mue\nabla \qe=\nabla f^*_{\ep}(\qe)$, we move on to the second main step, in which we will show that $f^*_{\ep}(\qe)$ weakly converges to $f^*(p)$ for some $p\in \partial f(\rho)$.  Our strategy in this step represents the most novel part of this paper.  Rather than directly prove the weak convergence --- a very challenging task with our limited integrability and regularity information --- we take a detour by reformulating equation (\ref{PDE}) in terms of its $\dot{H}^{-1}$ gradient flow structure instead of the $W_2$ structure. To convert between the two representations of (\ref{PDE}), we leverage the corresponding $\dot{H}^{-1}$ energy density:
\begin{equation}\label{eq:e_intro}
    e(a):=\begin{cases}
    af(a)-2\int_0^a f(\alpha)d\alpha & \textup{if}\; a\in \dom(f)\\
    +\infty &\textup{otherwise.}
\end{cases}
\end{equation}
Properties of the transformation $f\mapsto e$ were previously studied by the second author  \cite{jacobs2021existence}, where the $\dot{H}^{-1}$ structure was used to construct solutions to certain PDE systems containing a $W_2$ gradient flow structure. One can also view this transformation as the reverse direction of Otto's celebrated interpretation of the Porous Media Equation as a $W_2$ gradient flow \cite{otto_pme_geometry}. 

Once we have the transformed energy $e$,  equation (\ref{PDE}) can be rewritten as the following  $\dot{H}^{-1}$ gradient flow-type PDE
\begin{align} \tag{$\dot{H}^{-1}$ PDE} \label{PDEH}
\begin{cases}
\partial_t \rho -\Delta \zeta- \nabla \cdot (\rho v) = 0 ,  \qquad \qquad \text{ in duality with $C^\infty_c( (0,+\infty) \times \Rd),$}  \\
 \zeta \in \partial e(\rho), \text{for a.e. $(t,x)$} \\
\rho(0, \cdot) = \rho^0 .
\end{cases}
\end{align}
In particular, when $v=0$, the above equation is formally the $\dot{H}^{-1}$ gradient flow of the energy $\mathcal{E}(\rho) :=\int e(\rho(x))dx$.  The conversion between these structures comes from the observation that $b\mapsto f^*(b)$ is a monotone map on the set $\{b\in \partial f(a): a\in \dom(\partial f)\}$.  Therefore, the monotone relation between $\rho$ and $p$ (i.e. $p\in \partial f(\rho)$) and the fact that $e$ is defined precisely so that $\partial e(a)=\{f^*(b): b\in \partial f(a)\}$  yield a monotone relation between $\rho$ and $p$ in terms of $e$:   $f^*(p)\in \partial e(\rho)$. 

Due to the nonlocality of  (\ref{epsiloneqn}), it is not possible to rewrite this equation in terms of an $\dot{H}^{-1}$ structure.  Nonetheless, we can still introduce variables that will converge to the correct $\dot{H}^{-1}$ structure in the limit.
In particular,  we consider   the transformed energies $e_{\ep}$, obtained by applying the transformation (\ref{eq:e_intro}) to $f_{\ep}$, and define the approximate $\epsilon$-nonlocal $\dot{H}^{-1}$ dual variables via, 
\begin{equation}
    \label{eq:zeta def}\zeta_{\ep}:=f^*_{\ep}(\qe).    
\end{equation}
We then note that the $W_2$ duality relation $\qe=f'_{\ep}(\mue)$ implies the $\dot{H}^{-1}$ relation $\zeta_{\ep}=e'_{\ep}(\mue)$, where we leverage the fact that the regularized internal energy densities are differentiable.

Now that we have our ducks in a row, we are ready to show how to recover the duality relation in the limit.  Rather than directly prove that $\zeta_{\ep}=f^*_{\ep}(\qe)$ weakly converges to $f^*(p)$ for some $p\in \partial f(\rho)$, we   instead prove that $\zeta_{\ep}$ converges weakly to some $\zeta\in \partial e(\rho)$.  From there, it is not too difficult to show that  $\zeta\in \partial e(\rho)$ implies the existence of some $p\in \partial f(\rho)$ satisfying $\zeta=f^*(p)$.

The key motivation for this $\dot{H}^{-1}$ detour comes from the fact that the $\dot{H}^{-1}$ dual variables  $\zeta_{\ep}, \zeta$ are much better behaved than the $W_2$ dual variables $\qe, p_{\ep}, p$.  Indeed, $\zeta_{\ep}, \zeta$ are always nonnegative, whereas $ q_{\ep}, p_{\ep}, p$ can take on negative values.  In fact, for energies corresponding to fast diffusion equations, $p$ takes on the value $-\infty$ wherever $\rho=0$. (Though note that this doesn't happen at the $\epsilon$ level due to our regularization of $f$).  Furthermore, one typically expects $\dot{H}^{-1}$ dual variables to have better spatial regularity than $W_2$ dual variables.  This is because the energy dissipation relation for $W_2$ cannot control the pressure gradient where the density vanishes --- a difficulty that is not shared by $\dot{H}^{-1}$ flows.  Nonetheless, there are still nontrivial wrinkles that must be ironed out.  Since (\ref{epsiloneqn}) does not have an exact $\dot{H}^{-1}$ gradient flow structure, we do not know that $\int_{\{t\}\times \RR^d} e_{\epsilon}(\mue)$ stays bounded along the flow.  Relatedly, we also do not have enough integrability on any of the variables to know that the product $\mue\zeta_{\ep}$ belongs to $L^1$. This is challenging, as one typically passes to the limit in the term $\zeta_{\ep}=e'_{\ep}(\mue)$ by rewriting it in the equivalent form $\mue\zeta_{\ep}=e_{\ep}(\mue)+e^*_{\ep}(\zeta_{\ep})$.  To get around integrability issues for $\mue\zeta_{\ep}$ and $e_{\ep}(\mue)$, for each $m\in \N$ we introduce the truncated variables $\zeta_{\ep,m}:=\min(\zeta_{\ep}, m)$.  Since $\zeta_{\ep,m}$ is a monotone transformation of $\zeta_{\ep}$, there exists a truncated energy $e_{\epsilon,m}$ such that $\mue\zeta_{\ep,m}=e_{\ep,m}(\mue)+e^*_{\ep,m}(\zeta_{\ep,m})$.  Thanks to the truncation and the nonnegativity of $e_{\ep,m}, e^*_{\ep,m}$, the $L^1$ boundedness of both sides of $\mue\zeta_{\ep,m}=e_{\ep,m}(\mue)+e^*_{\ep,m}(\zeta_{\ep,m})$ becomes trivial.  We are then able to pass to the correct limit (as $\epsilon\to 0)$ in $\mue\zeta_{\ep,m}$ using spacetime compensated compactness and the correct limit in $e_{\ep,m}(\mue)+e^*_{\ep,m}(\zeta_{\ep,m})$ using weak lower semicontinuity and Young's inequality.   Finally, we   send $m\to\infty$ to recover our desired relation $\rho\zeta=e(\rho)+e^*(\zeta)$, completing the argument.

\subsection{Related work} \label{literature}
We now place our result in context with related work. Nonlocal approximation of diffusive PDEs has been an active area of research for more than twenty years. 
As explained above Corollary \ref{particlecorollary} below, this is largely due to the connection between nonlocal approximations and  deterministic particle methods for diffusion. We refer the reader to \cite{CarrilloCraigPatacchini} and \cite{craig2022blob} for comprehensive reviews of the related literature, along with \cite{di2015rigorous, daneri2022deterministic, daneri2023deterministic,di2022optimal} for   related developments in one spatial dimension and { \cite{philipowski2007interacting, figalli2008convergence} for related results in the presence of  Brownian motion. Work by Leclerc, M\'erigot, Santambrogio, and Stra \cite{leclerc2020lagrangian} considered an alternative Lagrangian discretization of nonlinear diffusion equations, which instead applies the Moreau-Yosida regularization directly to the internal energy functional, equation (\ref{internalenergydef}), at the level of the 2-Wasserstein distance. }There are also  several related nonlocal-to-local results for the Cahn-Hilliard equation. In particular, existence and uniqueness of solutions for nonlocal Cahn-Hilliard equations, as well as convergence to the local equation, was established by Davoli, Ranetbauer, Scarpa, and Trussardi  \cite{DRST2020}. In their proof, as in ours, a Moreau-Yosida regularization of the internal energy density is used. Relatedly,   Elbar, Perthame, and Skrzeczkowski establish that the Cahn-Hilliard equation is derived from the Vlasov equation via a nonlocal limit that involves  a double convolution in the diffusion term, similar to (\ref{epsiloneqn}) \cite{elbar2023limit}. Related nonlocal-to-local limits also appear in the study of hydrodynamic limits in models of collective behavior, where the regularization of the velocity is known as the \emph{Favr\'e filtration} \cite{shvydkoy2021dynamics, favre1983turbulence}.

The line of research most closely related to our approach began with the work of    Lions and Mac-Gallic \cite{lions2001methode}, which established that solutions to $\partial_t \rho_\ep-\div(\rho_\ep\nabla (\varphi_\ep*\rho_\ep))=0$ converge to those of the quadratic porous medium equation as $\ep \to 0$, leveraging the entropy dissipation inequality. In \cite{CarrilloCraigPatacchini}, Carillo, Craig, and Pattacchini built on this previous work by introducing a deterministic particle method for PDEs with linear  diffusion $\Delta u$ or diffusion of porous medium type $\Delta(u^m)$ and  establishing convergence for the quadratic porous medium equation ($m=2$). Convergence for a very similar regularization    was established for the inhomogenous quadratic porous medium equation by  Craig, Elamvazhuthi, Haberland, and Turanova \cite{craig2022blob}. More recently, Carillo, Esposito, and Wu succeeded in proving convergence for a general class of porous  medium equations, including $\partial_t \rho = \Delta \rho^m$, $m>1$.  \cite{carrillo2023nonlocal}. Finally, recent work by Burger and Esposito extended these ideas to the setting of cross diffusion equations \cite{BURGER2023113347}.

The connection between nonlocal approximations of diffusion  and sampling (see Remark \ref{samplingremark}) has also inspired several related works. We refer to \cite{craig2022blob} for a detailed review of the literature and the connection to other interacting particle methods for sampling, including Stein Variational Gradient Descent \cite{liu2016stein}. 
Recent works in this direction include work by Maoutsa, Reich, and Opper, who propose a sampling method based on an interacting particle system for linear diffusion  \cite{MaoutsaReichOpper2020}, as well as work by Li, Liu, Korba, Yurochkin, and Solomon \cite{li2022sampling}, who sample via spatially inhomogeneous quadratic porous medium diffusion, similarly to that introduced in  \cite{craig2022blob}. {Chen, Huang, Huang, Reich, and Stuart study the general relationship between gradient flows in the space of probability measures and sampling, as well as the distinguished role of the KL divergence \cite{chen2023sampling}, which corresponds to linear diffusion at the level of (\ref{PDE}).}
Recent work by Lu, Slep\v{c}ev, and Wang \cite{Lu_2023} developes a deterministic sampling method based on a nonlocal approximation of the Fokker-Planck equation with a birth-death term and studies its convergence both on bounded time intervals and asymptotically. 

\subsection{Organization of Paper}
The remainder of our paper is organized by follows. In Section \ref{preliminaries} we collect preliminary information on optimal transport, convex functions, and the Moreau-Yosida regularization. In Section \ref{regintenesec}, we prove several elementary properties of our regularized internal energies $f_\ep$ and their convex conjugates $f_\ep^*$. We also define and establish fundamental properties of the associated $\dot{H}^{-1}$ energy densities $e$ and $e_\ep$,  as well as their truncations $e_m$ and $e_{m,\ep}$. In Section \ref{sec:ep flow}, we develop the theory of the regularized equation (\ref{epsiloneqn}), proving well-posedness, the energy dissipation relation, and the entropy dissipation relation. Finally, in Section \ref{sec:convergence}, we turn to the proof of our main Theorem \ref{mainthm}, that solutions of (\ref{epsiloneqn}) converge to a solution of (\ref{PDE}) as $\epsilon \to 0$.

 \subsection{Acknowledgements}
The work of K. Craig has been supported by NSF DMS grant 2145900. The work of M. Jacobs has been supported by NSF DMS grant 2400641. The work of O. Turanova was supported by NSF DMS grant 2204722. The authors gratefully acknowledge the support from the Simons Center for Theory of Computing, at which part of this work was completed.

\section{Preliminaries} \label{preliminaries}
\subsection{Notation} \label{notationsec}
We denote the  $d$-dimensional Lebesgue measure by $dx$.  Given a Borel probability measure $\mu\in\P(\Rd)$, we write $d\mu(x) \ll dx$ if $\mu$ is absolutely continuous with respect to Lebesgue measure, in which case we will denote both the probability measure $\mu$ and its Lebesgue density by the same symbol, e.g. $d \mu(x) = \mu(x) dx$. For $p \geq 1$, let $M_p(\mu) = \int |x|^p d \mu(x)$ be the $p$-th moment of $\mu$ and let $\P_{p}(\Rd)$ denote the space of probability measures with finite $p$th moment.

We let $L^p(\mu;\Omega)$ denote  the Lebesgue space of functions $f$ on $\Omega$ with $|f|^p$ being $\mu$-integrable, and abbreviate $L^p(\Omega) = L^p(dx;\Omega)$. (We commit a slight abuse of notation by using  the same notation for the Lebesgue spaces of real-valued and $\rr^d$-valued functions.)  
For $p \geq 1$, we abbreviate $\| f \|_p = \| f \|_{L^p(\Rd)} =\left( \int |f|^p dx \right)^{1/p}$ for the $L^p$ 
norm of a function $f$.

For any $\theta \in (0,1)$, we let $\dot{C}^{\theta}(\RR^d)$ denote the space of functions with bounded H\"older seminorm on $\RR^d$, i.e. the space of continuous functions so that the following seminorm is finite:
\begin{align} \label{Cdotdef}
\norm{h}_{\dot{C}^{\theta}(\RR^d)}:=\sup_{x,y\in \R^d} \frac{|h(x)-h(y)|}{|x-y|^{\theta}} .
\end{align}
Likewise, we let $\dot{W}^{-\theta,1}(\RR^d)$ denote the dual space of $\dot{C}^{\theta}(\R^d)$.

\subsection{Optimal transport and the Wasserstein metric} \label{OTsec}
 
 We now describe basic facts about  the Wasserstein metric and Wasserstein gradient flows, which we will use in what follows. For further details, we refer the reader to one of the excellent textbooks on the subject \cite{ambrosiogiglisavare, villani2003topics, santambrogio2015optimal, figalli2021invitation,ambrosio2021lectures}.
 
Given a Borel measurable map $\bt \: \R^n \to \R^m$, we say that \emph{$\bt$ transports $\mu \in \P(\R^n)$ to $\nu \in \P(\R^m)$} if $\nu(A) = \mu(\bt^{-1}(A))$ for all measurable   $A \subseteq \R^m$. We refer to  $\bt$ as a \emph{transport map} and denote $\nu$ by $\bt_\# \mu \in \P(\R^m)$,  called the \emph{push-forward} of $\mu$ through $\bt$.
For $\mu,\nu\in\P(\R^d)$,  the set of \emph{transport plans} from $\mu$ to $\nu$ is given by,
\bes
	\Gamma(\mu,\nu) := \{\bgamma \in\P(\R^d\times\R^d) \mid {\pi^1}_\# \bgamma = \mu,\, {\pi^2}_\# \bgamma = \nu\},
\ees
where $\pi^1,\pi^2\colon \R^d \times \R^d \to \R^d$ are the projections onto the first and second components.
For $p \geq 1$, the \emph{$p$-Wasserstein distance}  between   $\mu,\nu\in\P_p(\R^d)$ is given by,
\be\label{eq:wass-p}
	W_p(\mu,\nu) = \min_{\bgamma \in \Gamma(\mu,\nu)}  \left( \int_{\R^d\times \R^d} |x-y|^p d \bgamma(x,y) \right)^{1/p} .
\ee
A transport plan  $\bgamma$ is \emph{optimal} if it attains the minimum in \eqref{eq:wass-p}, and we denote the set of optimal transport plans by $\Gamma_0(\mu,\nu)$. In the special case $p=1$, the dual formulation of the 1-Wasserstein metric \cite[Theorem 1.1]{villani2003topics} implies the following relationship to the $W^{-1,1}(\Rd)$ norm: for all $\mu, \nu \in \P_1(\Rd)$,
\begin{align} \label{dualsobolevtoW2}
 \norm{\mu-\nu}_{W^{-1,1}(\RR^d)}  = \sup_{\|\phi\|_{W^{1,\infty}} \leq 1} \int \phi ( \mu -  \nu) \leq \sup_{\phi \in W^{1,\infty}(\Rd) : \|\phi\|_{\rm Lip} \leq 1} \int \phi ( \mu -  \nu)  =   W_1(\mu, \nu) . \end{align}

Convergence with respect to the p-Wasserstein metric is stronger than narrow convergence of probability measures \cite[Remark 7.1.11]{ambrosiogiglisavare}. Recall that $\mu_n \to \mu$ narrowly means the probability measures converge in the duality with bounded continuous functions. However, if $ \mu_n \in \P_2(\R^d)$ and $\mu \in \P_2(\R^d)$, then
\begin{align} \label{W2andnarrowconv}
	\mbox{$W_2(\mu_n,\mu) \to 0$ as $n\to +\infty$} \iff \left(\mbox{$\mu_n \to \mu$ narrowly and $M_2(\mu_n) \to M_2(\mu)$ as $n\to +\infty$}\right).
\end{align}

We   require the following notion of regularity in time of curves in the space of probability measures.
\begin{defn}[Absolutely continuous]\label{defi:ac-curve}
We say $\mu:[0,+\infty)\rightarrow \P(\Rd)$ is \emph{locally absolutely continuous} on $[0,T]$, and write $ \mu \in AC^2_\loc([0,+\infty);P_2(\Rd))$, if there exists $f\in L^2_\loc([0,+\infty))$ so that, 
\begin{align} \label{eq:ac-p}
	W_2(\mu(t),\mu(s)) \leq \int_s^t f(r)\,d r \quad \mbox{for all  $0\leq s\leq t<+\infty$.}
\end{align}
\end{defn}

\subsection{Convex functions}
\label{ss:convex fns background}  We begin with 
basic definitions and facts, which may be found in  \cite[Chapters 1, 8, 16]{bauschke2011convex}.

Consider a  function $g:  \R \to \mathbb{R} \cup \{+\infty\}$. The \emph{domain} of $g$ is   $\dom(g) = \{ x \in \R : g(x) < +\infty\}$. We say $g$ is \emph{proper} if $\dom(g)\neq \emptyset$, and   $g$ is \emph{convex} if
 \[
 g(\alpha x+(1-\alpha)y)\leq \alpha g(x) +(1-\alpha)g(y) \text{  for all $x,y\in \dom(g)$ and all $\alpha\in[0,1]$}.
 \]
If $g$ is convex, then so is $\dom(g)$. 

The \emph{subdifferential} of a proper function $g:\R \to \mathbb{R} \cup \{+\infty\}$ is the set-valued operator
 $\partial g:\R \to 2^{\R}$ defined by
 \begin{align} \label{subdiffdef}
 \partial g(x)=\left\{p\in \R \, : \, \text{ for all $y\in \R$, }p(y-x)+g(x)\leq g(y)\right\}.
 \end{align}
 The $\emph{domain}$ of the subdifferential is defined by $\dom (\partial g)  = \{x\in \R\, |\, \partial g\neq \emptyset\}$. If $g$ is proper, convex, and lower semicontinuous, then we have the containment
 \begin{equation}
     \label{eq:dom g dom par g}
{\rm cont}(g)  = \interior \dom(g) \subseteq \dom(\partial g)\subseteq \dom(g) ,
 \end{equation}
 where ${\rm cont}(g)$ denotes the domain of continuity of $g$. Likewise, when  $g:\R\rightarrow\R\cup\{+\infty\}$ is proper, lower semicontinuous, and convex, $\partial g$ is a maximally monotone operator \cite[Definition 20.20, Theorem 20.25]{bauschke2011convex}, which implies that 
 \begin{equation}
     \label{eq:subdiff monotone}
     \text{if $x,y\in \dom(\partial g)$ with $x\leq y$, then $u\leq v$ for all $u\in \partial g(x)$ and all $v\in \partial g(y)$.
}
 \end{equation}
 
Throughout, we will use the following notions of $\lambda$-convexity and concavity for a function $f: \mathbb{R} \to \mathbb{R}$ and $\lambda \in \mathbb{R}$:
 \begin{align*}
& f \text{ is  {$\lambda$-convex/concave}} \iff f(\cdot)- \lambda |\cdot|^2/2 \text{ is convex/concave}.  
 \end{align*}
 Note that, if $f$ is   $\lambda_0$-convex and $\lambda_1$-concave, for $\lambda_0, \lambda_1 \in \mathbb{R}$, then $f \in W^{2, \infty}(\mathbb{R}^d)$ and its distributional second derivative satisfies $\lambda_0 \leq f'' \leq \lambda_1$  \cite[Theorem 18.15]{bauschke2011convex}.

\subsection{Convex conjugate and Moreau-Yosida regularization} \label{moreauenvelopesec}
We continue by collecting some elementary properties of  the convex conjugate and the Moreau-Yosida regularization; see \cite[Chapters 12-14]{bauschke2011convex}.  
Consider a  function $g:  \R \to \mathbb{R} \cup \{+\infty\}$.  
The \emph{convex conjugate} of $g$ is given by
 \[   g^*(b) = \sup_{a \in \mathbb{R}} \left\{ ab - g(a)  \right\}. \]
 If $g$ is proper, convex, and lower semicontinuous, so is the convex conjugate, and $g = (g^*)^*$. 
 Furthermore,  
 \begin{align} \label{convexdualupperlowerhessianbound}
 g \text{ is $\lambda$-convex   if and only if } g^*  \text{ is $\lambda^{-1}$-concave,}
 \end{align}
 and $g$ and $g^*$ satisfy the Fenchel-Young inequality:
 \begin{align} g(a) + g^*(b) \geq ab , \quad  \forall a, b \in \R,  \text{ with equality if and only if } b \in \partial g(a) \iff a \in \partial g^*(b) . \label{caseofequality}
 \end{align}
 
In the following lemma, we recall some standard facts about the convex conjugate. For a proof, see, for example, \cite[Lemma 2.1]{jacobs2021existence}.
 \begin{lem}
 \label{convexconjincreasing}
Suppose $h: \mathbb{R} \to \mathbb{R} \cup \{+\infty\}$ is   convex and lower semicontinuous,  with   $h(0)=0$ and $h(a) = +\infty$ for   $a <0$. Then $h^*$ is nondecreasing, nonnegative, and $\lim_{b \to -\infty} h^*(b) = 0$. 
 \end{lem}

For any proper, convex, lower semicontinuous function $g: \R \to \mathbb{R} \cup \{+\infty\}$ with $\dom(g) \subseteq [0,+\infty)$, the \emph{Moreau-Yosida regularization} of $g$ with parameter $\gamma >0$ is given by
\[ {}^\gamma g(a) = \inf_{b \geq 0} \left\{ g(b) + \frac{1}{2 \gamma} |a-b|^2 \right\}. \]
Let  $J_{\gamma}(a)$ denote the value of $b$ that attains the minimum, which is known as the \emph{proximal map}. 
Then ${}^\gamma g \in W^{2,\infty}(\mathbb{R})$ is a convex function with derivative
\begin{align} \label{MorEnDer}
  ({}^\gamma g)'(a) = \frac{a - J_\gamma(a)}{\gamma} \text{ and } \|({}^\gamma g)' \|_{\rm Lip} \leq \frac{1}{\gamma} \ . 
  \end{align}
Moreover, combining the  fact that ${}^\gamma g \in W^{2, \infty}(\mathbb{R})$ with equation (\ref{MorEnDer}) relating its derivative to $J_\gamma$, we see that $J_\gamma \in W^{1,\infty}(\mathbb{R})$. Likewise, one can  show that $\alpha \mapsto J_\gamma(a)$ is non-decreasing  and $|J_\gamma(a)- J_\gamma(b)| \leq |a-b|$ for all $a, b  \geq 0$ \cite[Proposition 12.27]{bauschke2011convex}. Furthermore, one can show that
\begin{align} \label{slopeboundresolvent}
\gamma |  ({}^\gamma g)'(a) | = |a- J_\gamma(a)| \leq \gamma |\partial g|(a) ,  \ \forall a \geq 0 ;
\end{align}
see \cite[Theorem 3.1.6]{ambrosiogiglisavare} where $|\partial g|(a)=\max\left(|\sup \partial g(a)|, |\inf\partial g(a)|\right)$.
Finally, we also have, 
\begin{align} \label{MEconv} {}^\gamma g(a) \uparrow g(a) \text{ as }\gamma \to 0 , \text{ for all }a \in \R . \end{align}

\section{Wasserstein and Dual Sobolev Internal Energies} \label{regintenesec}

\subsection{Regularized internal energy density}

We now collect some elementary properties of the regularized internal energy density $f_\ep$.

\begin{lem}[Properties of the regularized energy density]
\label{reginternalenergydensityprop}
Suppose the internal energy density $f$ satisfies   Assumption \ref{internalas} (\ref{convexlctsas}) and (\ref{nontrivial_convex}) .
Then the regularized internal energy density $f_\ep$, defined in equation (\ref{eq:eps energy}), satisfies the following properties:
\begin{enumerate}[(i)]

\item $f_\ep$ is proper, lower semicontinuous, and $\delta(\epsilon)$-convex,  with $f_\epsilon(0) = 0$ and $\dom(f_\ep)=[0,+\infty)$. \label{fepstronglyconvex}
\item $f_\ep$ is differentiable on $[0,+\infty)$, where $f_\ep'(0)$ denotes the derivative from the right at 0.
\label{item:fep diff}
\item  $f_\epsilon \in  W^{2,\infty}((0,+\infty))$, with distributional second derivative satisfying $\delta(\epsilon) \leq f''_\ep  \leq \delta(\epsilon) + \delta(\epsilon)^{-1} $ on $(0,+\infty)$, and $f_\ep'$ is Lipschitz on $[0,+\infty)$.
\label{fepw2infty}
\item $f_\epsilon$ converges to $f$ pointwise as $\ep \to 0$; hence  $f_\epsilon$ epi-converges to $f$ and $f_{\epsilon}$ converges uniformly to $f$ on any compact subset of $\interior(\dom(f))$.
\label{fep_to_f}
\end{enumerate}
\end{lem}
\begin{proof}

First, note that standard properties of the Moreau-Yosida regularization,  recalled in Section \ref{moreauenvelopesec}, imply that 
${}^{\delta(\ep)}f\in W^{2,\infty}(\rr)$ with distributional second derivative satisfying $0\leq ({}^{\delta(\ep)}f)''\leq \delta(\ep)^{-1}$. 
Items      (\ref{fepstronglyconvex}), (\ref{item:fep diff}),  and (\ref{fepw2infty})  follow immediately from this and from the definition of $f_\ep$. Next, we show part (\ref{fep_to_f}). By definition of $f_\epsilon$ and the pointwise convergence of the Moreau-Yosida regularization as $\ep \to 0$, as in equation (\ref{MEconv}), we see that $f_\ep(x) \to f(x)$ for all $x \in \mathbb{R}$. Thus, the epi convergence of $f_\epsilon \to f$ and the uniform convergence on compact subsets of  $\interior (\dom(f))$ follows from \cite[Theorem 7.17]{rockafellar2009variational}.

\end{proof}

Now we will  establish a uniform version of the coercivity property, Assumption \ref{internalas}(\ref{moment_control}), for the regularized energies $\F_{\epsilon}$.

\begin{lem}[Lower bound for $\F_\ep$]\label{lem:f_e_lower_bound}
Suppose the internal energy density $f$ and mollifier $\varphi_\epsilon$ satisfy  Assumptions \ref{internalas} and \ref{mollifieras}. 
 {Then there exists an increasing, concave function $\tilde{H}:[0,+\infty)\to[0,+\infty)$ such that $\tilde{H}(0)=0$ and 
\begin{equation}\label{f_e_lower_bound}
 \frac{\delta(\epsilon)}{2}\norm{\vp_{\ep}*\rho}_{L^2(\RR^d)}^2 \leq  \F_{\epsilon}(\rho) + \tilde{H}\left( \int_{\RR^d} (1+|x|^2)\left(\vp_{\ep}*\rho(x)\right)dx\right) \;  
\end{equation}
}
for all $\epsilon >0$ and finite Borel measures $\rho \in \mathcal{M}(\Rd)$ with finite second moment.
\end{lem}
\begin{proof}
If $f$ is everywhere nonnegative, so is the Moreau-Yosida regularization ${}^{\delta(\epsilon)}f$.  Thus, in this case, we can simply choose $\tilde{H}, \tilde{\Psi}=0$ and the conclusion holds.

Otherwise, if $f$ takes negative values, then there must exist some $b_0\in (0,+\infty]$ such that $f$ is strictly decreasing on $[0,b_0)$.  Since $f$ is differentiable at almost every point in its domain, for any $a\in [0, b_0)$ we have 
$f(a)=\int_0^a f'(\theta)\, d\theta.$
Similarly, from the definition of $f_{\epsilon}$, we have 
\[
f_{\epsilon}(a)=\frac{\delta(\epsilon)}{2}a^2+ \int_0^a {}^{\delta(\epsilon)}f'(\theta)\, d\theta.
\]
From inequality (\ref{slopeboundresolvent}), we have $|{}^{\delta(\epsilon)}f'(a)|\leq |f'(a)|$ for any $a$ that is a point of differentiability for $f$. Since 
${}^{\delta(\epsilon)}f'(\theta)\geq \min(0, {}^{\delta(\epsilon)}f'(\theta))$ and $f'(\theta)\leq 0$ for almost every $\theta\in [0, b_0]$ it follows that 
\begin{equation}\label{eq:moreau_deriv_lower_bound}
{}^{\delta(\epsilon)}f'(\theta)\geq f'(\theta)\quad \textup{for almost all} \;  \theta\in [0,b_0].
\end{equation}
Hence, for $a\in [0,b_0]$, we must have
\begin{equation}
    \label{eq:lb int del f'}
\int_0^a {}^{\delta(\epsilon)}f'(\theta)\, d\theta\geq \int_0^a f'(\theta)\, d\theta=f(a).
\end{equation}
Thus,
$f_{\epsilon}(a)\geq \frac{\delta(\epsilon)}{2}a^2+f(a)$
for all $a\in [0,b_0]$.

Now suppose we are given some $\rho \in \mathcal{M}(\Rd)$ and set $\mu=\vp_{\ep}*\rho$. We then split $\mu=\mu_1+\mu_2$ where
\[
\mu_1=\mu1_{\{\mu\leq b_0\}}, \quad \mu_2=\mu(1-1_{\{\mu\leq b_0\}}),
\]
where $1_{\{\mu\leq b_0\}}$ is the characteristic function of the set $\{\mu<b_0\}$.
Using this decomposition, we see that
\begin{equation}\label{eq:f_split_bound}
\mathcal{F}_{\epsilon}(\rho)=\int_{\RR^d} f_{\ep}(\mu_1)+f_{\ep}(\mu_2) \geq \frac{\delta(\epsilon)}{2}\norm{\mu_1}_{L^2(\RR^d)}^2+  \mathcal{F}(\mu_1)+\int_{\RR^d}f_{\ep}(\mu_2)
\end{equation}
where we have used the fact that 
$f_{\epsilon}(a)\geq f(a)$
for all $a\in [0,b_0]$ in the last inequality. 

It remains to handle $\int_{\RR^d} f_{\ep}(\mu_2)$.
Let $a_0\in (0, b_0)$ be a point of differentiability for $f$.  Expanding $\int_{\RR^d} f_{\ep}(\mu_2)$ we may write
\begin{align*}
\int_{\RR^d} f_{\ep}(\mu_2)&=\frac{\delta(\epsilon)}{2}\norm{\mu_2}_{L^2(\RR^d)}^2+ \int_{\RR^d} \int_0^{\mu_2(x)} {}^{\delta(\epsilon)}f'(\theta)\, d\theta \, dx\\
&=\frac{\delta(\epsilon)}{2}\norm{\mu_2}_{L^2(\RR^d)}^2+ \int_{\spt(\mu_2)} \int_0^{a_0} {}^{\delta(\epsilon)}f'(\theta)\, d\theta+\int_{a_0}^{\mu_2(x)} {}^{\delta(\epsilon)}f'(\theta)\, d\theta \, dx,
\end{align*}
where we have used the fact that $\mu_2>b_0\geq a_0$ on $\spt(\mu_2)$. Using (\ref{eq:lb int del f'}), with $a=a_0$, to bound the second term on the right-hand side of the previous line from below yields,
\[
\int_{\RR^d} f_{\ep}(\mu_2)\geq \frac{\delta(\epsilon)}{2}\norm{\mu_2}_{L^2(\RR^d)}^2+ f(a_0)|\spt(\mu_2)|+\int_{\spt(\mu_2)} \int_{a_0}^{\mu_2(x)} {}^{\delta(\epsilon)}f'(\theta)\, d\theta \, dx,
\]
where $|\spt(\mu_2)|$ is the Lebesgue measure of the set $\spt(\mu_2)$.
Exploiting the convexity of ${}^{\delta(\epsilon)}f$, it follows that  
\[
\int_{\RR^d} f_{\ep}(\mu_2)\geq \frac{\delta(\epsilon)}{2}\norm{\mu_2}_{L^2(\RR^d)}^2+ f(a_0)|\spt(\mu_2)|+{}^{\delta(\epsilon)}f'(a_0)\norm{(\mu_2-a_0)_+}_{L^1(\RR^d)}.
\]
Again using the fact that ${}^{\delta(\epsilon)}f'(a)\geq f'(a)$ for all $a\in [0, b_0]$ that are points of differentiability for $f$, we get 
\[
\int_{\RR^d} f_{\ep}(\mu_2)\geq \frac{\delta(\epsilon)}{2}\norm{\mu_2}_{L^2(\RR^d)}^2+ f(a_0)|\spt(\mu_2)|+f'(a_0)\norm{(\mu_2-a_0)_+}_{L^1(\RR^d)}.
\]
Since $\mu_2\geq b_0$ on its support, we have the standard estimate
\[
|\spt(\mu_2)|\leq \frac{1}{b_0}\norm{\mu_2}_{L^1(\RR^d)}.
\]
This finally allows us to conclude that
\begin{equation}\label{eq:F_small_part_bound}
\int_{\RR^d} f_{\ep}(\mu_2)\geq \frac{\delta(\epsilon)}{2}\norm{\mu_2}_{L^2(\RR^d)}^2-\left(\frac{|f(a_0)|}{b_0}+|f'(a_0)|\right)\norm{\mu_2}_{L^1(\RR^d)}.
\end{equation}

Combining our estimates, we see that
\[
\F_{\ep}(\rho)=\int_{\RR^d} f_{\ep}(\mu_1)+f_{\ep}(\mu_2)\geq \frac{\delta(\epsilon)}{2}\norm{\mu_1}_{L^2(\RR^d)}^2+\frac{\delta(\epsilon)}{2}\norm{\mu_2}_{L^2(\RR^d)}^2+\mathcal{F}(\mu_1) -\left(\frac{|f(a_0)|}{b_0}+|f'(a_0)|\right)\norm{\mu_2}_{L^1(\RR^d)}.
\]
This can be simplified to 
\begin{equation}\label{eq:F_1_2_lower_bound}
\F_{\ep}(\rho)\geq \frac{\delta(\epsilon)}{2}\norm{\mu}_{L^2(\RR^d)}^2+\mathcal{F}(\mu_1) -\left(\frac{|f(a_0)|}{b_0}+|f'(a_0)|\right)\norm{\mu_2}_{L^1(\RR^d)},
\end{equation}
where we have exploited the fact that $\mu_1, \mu_2$ have disjoint support to recombine the $L^2$ norms.

Now we have everything we need to conclude. Let $H$ be a choice of function with the properties guaranteed by assumption  \ref{internalas}(\ref{moment_control}) on $\F$. We then define 
\begin{align*}
\tilde{H}(a)=2H \left(\frac{a}{2} \right)+\left(\frac{|f(a_0)|}{b_0}+|f'(a_0)|\right)a,
\end{align*}
from which it is clear that $\tilde{H}$ is nonnegative, nondecreasing,  $\tilde{H}(0)=0,$ and $\tilde{H}$ is concave.
Now we are ready to show that this choice will produce the desired result.  Exploiting the concavity of $\tilde{H}$, we see that
\[
\mathcal{F}_{\epsilon}(\rho)+\tilde{H}\left( \int_{\RR^d} (1+|x|^2)(\vp_{\ep}*\rho)\, dx\right) \geq \mathcal{F}_{\epsilon}(\rho)+\frac{1}{2} \tilde{H}\left(2\int_{\RR^d} (1+|x|^2)\mu_1(x)\, dx\right) +\frac{1}{2}\tilde{H}\left( 2\int_{\RR^d} (1+|x|^2)\mu_2(x)\, dx\right) . 
\]
Using the fact that $\tilde{H}(a)\geq \max(2H(\frac{a}{2}), \left(\frac{|f(a_0)|}{b_0}+|f'(a_0)|\right)a)$, we find that the right-hand side of the previous line is bounded from below by 
\[
\mathcal{F}_{\ep}(\rho)+H\left(\int_{\RR^d} (1+|x|^2)\mu_1(x)\, dx\right) +\int_{\RR^d} \left(\frac{|f(a_0)|}{b_0}+|f'(a_0)|\right)(1+|x|^2)\mu_2(x)\, dx,
\]
which is larger than
\[
\mathcal{F}_{\ep}(\rho)+H\left(\int_{\RR^d} (1+|x|^2)\mu_1(x)\, dx\right) +\left(\frac{|f(a_0)|}{b_0}+|f'(a_0)|\right)\norm{\mu_2}_{L^1(\RR^d)}.
\]
The result now follows from inequality (\ref{eq:F_1_2_lower_bound}) and the guarantee that $\mathcal{F}(\mu_1)+H\left( \int_{\RR^d} (1+|x|^2)d\mu_1(x)\right) \geq 0$ for any nonnegative measure $\mu_1$ with finite second moment.
\end{proof}

\subsection{Duality}

Duality relations and their interplay with convolution will play a fundamental role in our convergence proof.
 In the following lemmas, we collect several elementary properties about the convex conjugate of the regularized energy $f^*_{\epsilon}$, the regularized measure $\mu_\ep = \varphi_\ep * \rho_\ep$, and the potential $q_\ep = f'_\ep(\varphi_\ep *\rho) = f'_\ep(\mu_\ep)$.
 
The following properties of $f^*_{\epsilon}$ are immediate consequences of Lemma \ref{reginternalenergydensityprop} and equation (\ref{convexdualupperlowerhessianbound}).
 \begin{cor}[Properties of the dual of the regularized energy density] \label{conjugatecor}
 Suppose the internal energy density $f$ satisfies Assumption \ref{internalas}. Then the convex conjugate of the regularized internal energy density   enjoys the following properties:
 \begin{enumerate}[(i)]
\item \label{item:fep*} $f_\ep^*$ is proper, lower semicontinuous, and  
convex.
\item $f_\ep^* \in W^{2, \infty}(\mathbb{R})$, with distributional second derivative satisfying $0 \leq (f_{\ep}^*)'' \leq\delta(\epsilon)^{-1} $. 
\label{conjugateW2inf}
\item  $f_\ep^*$ is nondecreasing and nonnegative.
\label{item:fep nondec nonneg}
\end{enumerate}
 \end{cor}

 \begin{proof}
 Recall from Lemma \ref{reginternalenergydensityprop}, part (\ref{fepstronglyconvex}), that $f_\epsilon$ is proper, $\delta(\epsilon)$-convex, and lower semicontinuous.
  Then, item (\ref{item:fep*}) immediately follows from the definition of convex conjugate. For item  (\ref{conjugateW2inf}), since $f_\epsilon$ is $\delta(\epsilon)$-convex, equation (\ref{convexdualupperlowerhessianbound}) ensures that $f_\epsilon^*$ is $\delta(\epsilon)^{-1}$-concave. Combining this with the fact that $f_\epsilon^*$ is convex ensures $f_\epsilon^* \in W^{2, \infty}(\R)$ and provides the lower and upper bounds on the distributional second derivative (see Section \ref{ss:convex fns background}).  
 Finally, item (\ref{item:fep nondec nonneg}) follows from Lemma \ref{convexconjincreasing}.   
  \end{proof}

 Now we establish properties of the coupled variables $\mue$ and $\qe$.

\begin{lem}[Properties of $\mu_\ep$ and $q_\ep$] \label{muepqeplem}
Suppose the internal energy density $f$ satisfies Assumption \ref{internalas} and the mollifier $\varphi_\ep$ satisfies Assumption \ref{mollifieras}. Fix $\rho_\ep \in \P_2(\Rd)$. Define $\mu_\ep = \varphi_\ep*\rho_\ep  $ and $q_\ep = f_\ep'(\mu_\ep)$. 
Then the following hold, for all $0 < \epsilon < 1$:
\begin{enumerate}[(i)]
\item  \label{FYidentity}
$\mu_\ep q_\ep = f_\ep(\mu_\ep) +f^*_\ep(q_\ep)$.
\item
$\mu_\ep  = (f_\ep^*)'(q_\ep)$.  \label{qtomu}
\item \label{muepdif} $\mu_\ep \in C^1(\Rd)$ with $\|\mu_\ep\|_{L^\infty(\R^d)}<+\infty$, and  $\nabla \mu_\ep = (\nabla \varphi_\ep)*\rho_\ep$.
\item\label{item:qepbdd} $\|q_\ep\|_{L^\infty(\R^d)}<+\infty$.
\item\label{item:pepbdd} Denoting $p_\ep = \varphi_\ep*q_\ep$, we have $\|D^2 p_\ep\|_{L^\infty(\R^d)}+\|\nabla p_\ep\|_{L^\infty(\R^d)}<+\infty$.
\item \label{qeplip} $q_\epsilon$ is globally Lipschitz; $f''_\ep(\mu_\ep(x))$ is well defined almost everywhere on $\{ x : \nabla \mu_\ep(x) \neq 0 \}$;    for a.e. $x \in \mathbb{R}^d$ we have,
\begin{align} \label{nablaqep}
\nabla q_\ep(x) = \begin{cases} f''_\ep(\mu_\ep(x)) \nabla \mu_\ep(x) &\text{ for $\nabla \mu_\ep(x) \neq 0$,} \\
0 &\text{ otherwise;}
\end{cases}
\end{align}
 and $f^*_{\ep}(\qe) \in W^{1,\infty}(\Rd)$ with the equality
$\nabla f^*_{\ep}(\qe)=\mue\nabla \qe$ holding a.e. on $\Rd$. 
\item $(f_\ep^*)''(q_\ep(x))$ is well defined almost everywhere on $\{x: \nabla q_\ep(x) \neq 0 \}$ and for a.e. $x \in \mathbb{R}^d$, \label{nablamuep}
\begin{align}
\nabla \mu_\ep(x) = \begin{cases} (f^*_\ep)''(q_\ep(x)) \nabla q_\ep(x) &\text{ for $\nabla q_\ep(x) \neq 0$,} \\
0 & \text{ otherwise.}
\end{cases}
\end{align}
\item $|\nabla \qe(x)|\leq (\delta(\ep)+ \delta(\ep)^{-1})|\nabla \mue(x)|$ and $|\nabla \mue (x)|\leq \delta(\ep)^{-1} |\nabla \qe(x)|$ for a.e. $x \in \mathbb{R}^d$; \label{nabla_q_nabla_mu}
\item \label{muepM2} $M_2(\mu_\ep) \leq 2 M_2 (\rho_\ep)+ 2    \left( 1+\frac{d}{r-d-2} \omega_d C_\varphi \right)$.
\end{enumerate}
\end{lem}

\begin{proof}
First, we note that our assumptions on the mollifier $\varphi_\ep$ imply $\mu_\ep\in C^1(\rr^d)$ and $\mu_\ep$ non-negative on $\rr^d$. Since, by Lemma \ref{reginternalenergydensityprop} item (\ref{fepw2infty}),  $f_\ep'$ is Lipschitz on $[0,+\infty)$, and the composition of Lipschitz functions is Lipschitz, we find that $q_\ep$ is Lipschitz on $\rr^d$.

Item (\ref{FYidentity}) follows from the case of equality (\ref{caseofequality}) in the Fenchel-Young inequality.  
Item (\ref{qtomu}) follows since, by Corollary \ref{conjugatecor} (\ref{conjugateW2inf}), $f_{\ep}^*$ is continuously differentiable and, by the case of inequality (\ref{caseofequality}) in the Fenchel-Young inequality, $(f_\ep^*)' = (f_\ep')^{-1}$.

To see (\ref{muepdif}), we first use Young's inequality for convolution and  Assumption \ref{mollifieras} (\ref{molLinftybounds}) to obtain  
 $\|\varphi_\ep *\rho_\ep\|_{L^\infty(\rr^d)}\leq \|\varphi_\ep \|_{L^\infty(\rr^d)}\|\rho_\ep\|_{L^1(\rr^d)}<+\infty$. Next,
note that
\begin{align*}
  \frac{ |\mu_\ep(y) - \mu_\ep(x) - ((\nabla \varphi_\ep)*\rho_\ep)(x) \cdot(y-x)|}{|x-y|}  &=   \frac{ \left| \int \left( \varphi_\ep(z-y) - \varphi_\ep(z-x) -  \nabla \varphi_\ep(z-x) \cdot (y-x) \right) d \rho_\ep(z)  \right|}{|x-y|} \\
& \leq \|D^2 \varphi_\ep \|_\infty |x-y| ,
\end{align*}
where the right hand side goes to zero as $y \to x$.

Now, we use $\|\mu_\ep\|_{L^\infty(\rr^d)}<+\infty$ and the fact that $f'_\ep$ is Lipschitz (Lemma \ref{reginternalenergydensityprop} (\ref{fepw2infty})) to deduce item (\ref{item:qepbdd}). For the first estimate of item (\ref{item:pepbdd}), we use the definition of $p_\ep$, item (\ref{item:qepbdd}), and our assumptions on $\varphi$ to obtain
\[
\|D^2 p_\epsilon\|_{L^\infty(\rr^d)} \leq  \|D^2\varphi_\ep\|_{L^1(\rr^d)} \|f_\ep'(\varphi_\ep*\rho_\ep))\|_{L^\infty(\rr^d)}<+\infty.
\]
The other estimate of item (\ref{item:pepbdd}) is obtained similarly.

Next, we show items (\ref{qeplip}-\ref{nablamuep}). Lemma \ref{reginternalenergydensityprop}(\ref{fepw2infty}) ensures $f_\ep'$ is Lipschitz on $[0,+\infty)$ and part (\ref{muepdif}) ensures $\mu_\ep \in C^1(\mathbb{R}^d) \subseteq W^{1,1}_\loc(\mathbb{R}^d)$; therefore the composition $q_\ep=f_\ep'(\mu_\ep)$ is Lipschitz on $\Rd$.  
The expression for the gradient of $q_\ep$ follows from unpublished work of Serrin \cite{Serrin}; see also \cite{leoni2007necessary}.  Next, Corollary \ref{conjugatecor}(\ref{conjugateW2inf}) ensures $(f_\ep^*)'$ is Lipschitz and item (\ref{qeplip}) ensures $q_\ep \in W^{1,\infty}(\mathbb{R}^d) \subseteq  W^{1,1}_\loc(\mathbb{R}^d)$; therefore, the composition $f_\ep^*(q_\ep)$ is also Lipschitz on $R^d$. The  equality $\nabla f_\ep^*(\qe)=\mue\nabla \qe$ follows from the chain rule for Lipschitz functions.  Finally, item (\ref{nablamuep}) follows from the same logic as the previous item.
 Now, we turn to item (\ref{nabla_q_nabla_mu}). This is an immediate consequence of items (\ref{qeplip}-\ref{nablamuep}), combined with  the bound $\|f_\ep'\|_{\rm Lip ((0,\infty))} \leq \delta(\epsilon) + \delta(\epsilon)^{-1}$ from  Lemma \ref{reginternalenergydensityprop}(\ref{fepw2infty}) as well as the bound $\|(f_\ep^{*})'\|_{\rm Lip} \leq \delta(\epsilon)^{-1}$ from Corollary \ref{conjugatecor} (\ref{conjugateW2inf}).

To see (\ref{muepM2}), note that the evenness of $\vp$ implies that
\begin{align*}
 \int |x|^2 d \mu_\ep(x) = \int |x|^2 \varphi_\ep(z-x) d \rho_\ep(z) dx = \int |z-y|^2 \varphi_\ep(y) d \rho_\ep(z) dy =  M_2(\rho_\ep) +  M_2(\varphi_\ep).
 \end{align*}
 By Assumption \ref{mollifieras} (\ref{secondmolas}),
 \begin{align*} M_2(\varphi_\ep) &\leq \int_{|x|\leq 1} |x|^2 \varphi_\ep(x) dx +\int_{|x|\geq 1} |x|^2 \varphi_\ep(x) dx \leq 1 + \epsilon^{r-d} C_\varphi \int_{|x| \geq 1} |x|^2 |x|^{-r} dx  \\
 &= 1 + d \ep^{r-d}    \omega_d  C_\varphi \int_1^\infty r^{2-r+d-1} dr \leq 1+ d \ep^{r-d}    \omega_d C_\varphi (r-d-2)^{-1} ,
 \end{align*}
where $\omega_d$ is the volume of the $d$-dimensional unit ball.
\end{proof}

\subsection{Definition and properties of the $\dot{H}^{-1}$ energy density}
\label{sss:e}

Given an internal energy density $f$, we  now define the $\dot{H}^{-1}$ energy density $e$, as in previous work by the second author \cite[Lemma 2.3]{jacobs2021existence}.   These transformed energy densities,  their truncations  (see Definition \ref{eq:def em} below), and the limiting properties thereof, are all important ingredients in our approach; specifically, they are used extensively in the proofs of Proposition \ref{identifylimitpoint} and Lemma \ref{existenceofpwow}.

This self-contained section is devoted to collecting   and establishing properties of the $\dot{H}^{-1}$ energy densities. The proofs use standard properties of convex functions; we include all the details for the convenience of the reader. 
 
\begin{defn}[$\dot{H}^{-1}$ energy density] \label{hm1energydendef}
    For any  energy density $f$ satisfying Assumption \ref{internalas}(i) and (ii),  
we define the transformed energy density $e$ via
\begin{align} 
\label{eq:e def}
e(a) := \begin{cases} a f(a) - 2 \int_0^a f(s) ds &\text{ if } a \in \dom (f) , \\
    +\infty &\text{ otherwise},  \end{cases} 
    \end{align}
\end{defn}
We now recall several basic properties of this energy, including that it enjoys the properties in Assumption \ref{internalas} (\ref{convexlctsas}), and its subdifferential has an explicit characterization in terms of the subdifferential of $f$.

\begin{lem}[Basic properties of the $\dot{H}^{-1}$ energy density] \label{eepsproperties} 
Consider $f$ satisfying items (\ref{convexlctsas}) and (\ref{nontrivial_convex}) of Assumption \ref{internalas}, and let $e$ be defined by (\ref{eq:e def}). Then, 
\begin{enumerate}[(i)]
\item The function $e: \R \to \R \cup \{+\infty\}$ is convex and lower semicontinuous, with $e(0)=0$ and $e(a)=+\infty$ for $a<0$. Moreover, ${\rm int}(\dom(e)) \neq \emptyset$ and $e$ is non-decreasing on  $[0,+\infty)$.\label{eepsconvex} 
\item For any $a \in \dom (\partial f)$  with $a>0$,    $\{f^*(b): b\in \partial f(a)\}= \partial e(a)$. \label{subdiffe}
\item $\partial e(0)=(-\infty, 0]$.\label{item:e_at_zero}
\item  There exists $b\in \RR$ such that $f(a)=ba$ for all $a\in e^{-1}(0)$.\label{f_e_linear}
\end{enumerate}
\end{lem}
\begin{proof}
Parts (\ref{eepsconvex}) and (\ref{subdiffe}) follow from  \cite[Lemma 2.3]{jacobs2021existence}. (Note that the proof of \cite[Lemma 2.3]{jacobs2021existence} applies to $f$  satisfying our hypotheses and does not require $\sup \partial f(0) < f'(a)$ for some $a>0$.)

For item (\ref{item:e_at_zero}), 
 we first check that for any $a>0$
\[
\sup \partial e(0)\leq \frac{e(a)-e(0)}{a}=f(a)-2\frac{1}{a}\int_0^a f(s)\, ds\leq f(a)-2f(a/2),
\]
where the second inequality follows from Jensen's inequality.
Sending $a\to 0$, we see that $\partial e(0)\cap (0,+\infty)=\emptyset$.  It follows that $\partial e(0)=(-\infty,0]$ since $e(0)=0$ and $e(a)=+\infty$ for $a<0$.

The final item is trivial if $e^{-1}(0)=\{0\}$.  Otherwise, $e^{-1}(0)=[0,a_1]$ for some $a_1>0$.  Since $e$ is differentiable on $(0,a_1)$, it follows that $f$ is differentiable on $(0,a_1)$.  We then have the relation
\[
0=e'(a)=af'(a)-f(a)
\]
for all $a\in (0,a_1)$. Since $f'$ is nondecreasing and $f(a)=\int_0^a f'(\theta)d\theta$, we can only have $f(a)=af'(a)$ if $f'(\theta)=f'(a)$ for all $0\theta\leq a<a_1$.  This implies that $f(a)=ba$ on $e^{-1}(0)$ for some $b\in \RR$.

\end{proof}

Next, we consider truncations in the subdifferential of the $\dot{H}^{-1}$ energy density, and the convergence of the truncations to $e$ itself.  These truncations, which extend $e$ linearly whenever the subdifferential of $e$ exceeds a given value $m$, will play a critical role when we identify the limiting duality coupling between the density and the pressure. 
\begin{defn}[Truncated $\dot{H}^{-1}$ energy densities] \label{trunk}
Given a $W_2$ energy density $f$ satisfying items (\ref{convexlctsas}) and (\ref{nontrivial_convex}) of Assumption \ref{internalas}, and its associated $\dot{H}^{-1}$ density $e$, as in Definition \ref{hm1energydendef}, we first  define, for each $m\in \N$,
\begin{align*}
a_{m} := \sup \{ a \in \dom(\partial e) : \sup \{ \partial e(a) \}  \leq m   \}.  \  \end{align*}
Note that, by  Lemma \ref{eepsproperties} (\ref{item:e_at_zero}), we have $a_m\geq 0$ for all $m\in \N$. We then use $a_m$ to define the truncated $\dot{H}^{-1}$ energy densities,
\begin{align}
\label{eq:def em}
e_{m}(a) := \begin{cases} e(a) &\text{ for } a< a_{m} , \\ e(a_{m}) + m(a-a_{m}) & \text{ for } a \geq a_{m} .\end{cases}  
\end{align}
\end{defn}

 \begin{lem}[Truncation of the $\dot{H}^{-1}$ energy] 
\label{lem:truncation} Let $f$ satisfy items (\ref{convexlctsas}) and (\ref{nontrivial_convex}) of Assumption \ref{internalas}, and let $e$ be as in Definition \ref{hm1energydendef}.  Then, the truncated energies $e_m$ from Definition \ref{trunk} satisfy the following properties:
\begin{enumerate}[(i)]
\item \label{item:e_m_convex} For all $m\in \N$,  $e_m$ is proper, lower semicontinuous, and convex, with $e_m(0)=0$.
\item For all $m\in \N$,    $e_m$ is non-decreasing on $[0,+\infty)$ and $\dom(e_{m}) = [0,+\infty)$. \label{domainitem}
\item \label{item:m_nondecreasing}
The sequence $\{a_m\}_{m\in \N}$ is nondecreasing, and for each $a\in [0,+\infty)$, the sequence $\{e_m(a)\}_{m\in \N}$ is nondecreasing.  Furthermore $e_m(a)\leq e(a)$ for all $m\in \N$ and $a\in \RR$.
\label{item:size e em} 
\item   For all $a \in \R$,  $\lim_{m \to +\infty} e_{m} (a) = e(a)$. \label{emtoe}    \label{emtoepointwise}
\item  
\label{partialemepschar}
For  all  $m\in \N$ and for all $a \in \dom ( \partial e)$,  $\partial e_{m}(a) = \left\{\min \{ p, m \}\, :\, p\in \partial e(a) \right\}.$ 

\end{enumerate}
\end{lem}

\begin{proof}
Throughout the proof, we will use the fact that $e$ enjoys the properties established in Lemma \ref{eepsproperties}(\ref{eepsconvex}) and therefore the properties outlined in Section \ref{ss:convex fns background}. Part (\ref{item:e_m_convex}) follows  from the definition.

%We begin by showing (\ref{item:e_m_convex}). The definition of $e_m$ and the fact that $e(0)=0$ (Lemma \ref{eepsproperties}(\ref{eepsconvex})) imply $e_m(0)=0$. Next, we see that $e_m$ is proper, lower semicontinuous, and convex on $(-\infty, a_m]$ and $[a_m,\infty)$ when considered on each domain separately.   This immediately implies that $e_m$ is proper and lower semicontinuous on all of $\RR$.  If $a_m=0$, then convexity on all of $\RR$ is also immediate since $e(a)=+\infty$ on $(-\infty, 0)$.  Otherwise, if $a_m>0$, to obtain convexity on all of $\RR$
% it suffices to show that $e_m'$ is monotone almost everywhere.  Since $e$ is differentiable almost everywhere on $\dom(e)$ and $e_m'(a)=m$ for all $a>a_m$, this amounts to showing that $e'(a)\leq m$ whenever $a\in [0,a_m)$ is a point of differentiability for $e$.  Luckily, this follows immediately from the definition of $a_m$.  

 Now we consider item (\ref{domainitem}). Note that if $a_m=+\infty$, then there exist $\tilde{a}_i\in \dom(\partial e)\subset \dom(e)$ with $\tilde{a}_i\rightarrow+\infty$. Since $\dom(e)$ is convex and $0\in \dom(e)$, we have $\dom(e)=[0,+\infty)$ as desired.  It also follows that $e_m$ is nondecreasing on $[0,+\infty)$ since we know that $e$ is nondecreasing on $[0,+\infty)$.
 On the other hand, if $a_m<+\infty$, then for any $a\leq a_m$ we have $e_m(a)=e(a)=\int_0^{a} e'(\alpha)\, d\alpha\leq ma$ by the convexity of $e$ and the definition of $a_m$.  Combining this with the definition of $e_m$ for $a\geq a_m$, it follows that $0\leq e_m(a)\leq ma$ for all $a\geq 0$, so $\dom(e_m)=[0,+\infty)$.  In this case, $e_m$ being nondecreasing on $[0,+\infty)$, follows from the fact that $e$ and $e(a_m)+m(a-a_m)$ are nondecreasing and agree at $a_m$.

For item (\ref{item:m_nondecreasing}), it is clear that $a_m$ is nondecreasing since $\sup \partial e(a)$ is nondecreasing. 
Next, we want to compare $e_m(a)$ and $e_{m+1}(a)$ for some $a\in \RR$.  If $a_m=a_{m+1}$, then it is clear that $e_{m}\leq e_{m+1}$ holds on $\RR$.  Otherwise, if $a_m<a_{m+1}$, then it follows that $[a_m, a_{m+1})\subset \dom(e)$ so $e$ is differentiable almost everywhere on $(0,a_{m+1})$. 
 Using the fact that $e_{m}$ and $e_{m+1}$ agree up to $a_m$, we can write
\[
e_{m+1}(a)-e_m(a)=\int_{a_m}^{\max(a,a_m)} e_{m+1}'(\alpha)-m\, d\alpha=\max(a-a_{m+1},0)+\int_{a_m}^{\max(\min(a_{m+1},a),a_m)} e'(\alpha)-m\, d\alpha.
\]
This is nonnegative from the definition of $a_m$ and the convexity of $e$. It remains to show that $e_m(a)\leq e(a)$ for all $a\in \RR$ and $m\in \N$.  This is trivial when $e(a)=+\infty$ or when $a\leq a_m$, so we only need to consider $a\in \dom(e)\cap (a_m,\infty)$. Fixing some $a\in \dom(e)\cap (a_m,\infty)$, we can write
\[
e(a)=\int_0^a e'(\alpha)d\alpha=e(a_m)+\int_{a_m}^a e'(\alpha)d\alpha\geq e(a_m)+m(a-a_m) = e_m(a) 
\]
where we have used the definition of $a_m$ in the inequality. % Thus, $e_m(a)\leq e(a)$ for all $a\in \RR$ and $m\in \N$.

Now we show item (\ref{emtoe}).  Since $a_m$ is increasing, it follows that $\lim_{m\to\infty} a_m=a^*$ for some  $a^*\in \RR\cup\{+\infty\}$.   Clearly, $\lim_{m\to\infty} e_m(a)=e(a)$ whenever $a< a^*$.  If $a^*\neq +\infty$, then for any $a>a^*$ we have 
\[
\lim_{m\to\infty} e_m(a)=\lim_{m\to\infty} e(a_m)+m(a-a_m)\geq \lim_{m\to\infty} m(a-a_m)=+\infty,
\]
where we have used the fact that $e\geq 0$ everywhere.  Since we know that $e\geq e_m$ for all $m$, we must also have $e(a)=+\infty$.  Thus, the only remaining case is the convergence when $a=a^*\neq +\infty$. Note that we must have $a^*>0$ since $\dom(e)\neq \{0\}$. By the lower semicontinuity of $e$,
\[
e(a^*)\leq\lim_{h\to 0^+} e(a^*-h)=\lim_{h\to 0^+}\lim_{m\to\infty} e_m(a^*-h)\leq \lim_{m\to\infty} e_m(a^*)
\]
where we have used the fact that $e_m$ is increasing on $[0,+\infty)$ in the final inequality.  Since $e(a^*)\geq e_m(a^*)$ for all $m$ we are done.

Finally, we show item (\ref{partialemepschar}). Note that we only need to worry about values of $a\in \dom(\partial e)\cap [a_m,\infty)$.
From the definition of $a_m$, we have $\sup \partial e(a)>m$ for all $a\in \dom(\partial e)\cap (a_m,\infty)$.  Using the general fact about convex functions that $\sup \partial e(a)\leq \inf \partial e(a')$ whenever $a<a'$, we have $\inf \partial e(a)>m$ for all $m\in \dom(\partial e)\cap (a_m,\infty)$. 
Therefore, 
\[
 \left\{\min \{ p, m \}\, :\, p\in \partial e(a) \right\}  =\{ m \}= \{e_m'(a)\},
\]
for all $a\in \dom(\partial e)\cap (a_m,\infty)$.  Finally, we shall show that the desired equality holds at $a_m$. Let $p\in \partial e(a_m)$ be such that $p\leq m$.  The definition of $e_m$ implies that the tangent line to $e$ at $a_m$ with slope $p$ is also a tangent line to $e_m$ at $a_m$, so we find $p\in\partial e_m(a_m)$.  Therefore $\{\min(p,m): p\in \partial e(a_m)\}\subset \partial e_m(a_m)$.  To show the reverse containement, let us suppose we have $b\in \partial e_m(a_m)$ but $b\notin \partial e(a_m)$. Then there exits $a>a_m$ with $e(a_m)+b(a-a_m)>e(a)$; note that this implies $a\in \dom(e)$. We can rewrite this condition as $b> \frac{1}{a-a_m}\int_{a_m}^a e'(\alpha)d\alpha$, which from the definition of $a_m$ implies that $b>m$.  However, this is a contradiction, as it is clear that $e_m$ cannot have a tangent line with slope greater than $m$ at $a_m$.  Therefore, $\{\min(p,m): p\in \partial e(a_m)\}= \partial e_m(a_m)$ and we are done.

\end{proof}

The definition of $f_\ep$, along with 
Lemma  \ref{reginternalenergydensityprop} (\ref{fepstronglyconvex}), imply that, for all $\ep>0$, the energy $f_\ep$ satisfies items (\ref{convexlctsas}) and (\ref{nontrivial_convex}) of Assumption \ref{internalas}. Therefore,  the energies $e_\ep$ and $e_{m,\ep}$ can be defined analogously to $e$ and $e_m$ and enjoy the properties established in the previous two lemmas. Recall, however, that $\dom(f_\ep)=[0,+\infty)$ and $f_\ep$ is differentiable on its domain. This simplifies several aspects of the definitions of $e_\ep$ and $e_{m,\ep}$. We summarize the definitions as:
\begin{equation}
\label{eq:def eepm}
\begin{split}
& e_\epsilon(a) := \begin{cases} a f_\epsilon(a) - 2 \int_0^a f_\epsilon(s) ds &\text{ if } a \in [0,+\infty), \\
    +\infty &\text{ otherwise},  \end{cases}\\
&a_{m, \epsilon} := \sup \{ a \in [0,+\infty):  e_\epsilon'(a)  \leq m   \},\\
&
  e_{m,\epsilon}(a) := \begin{cases} e_\epsilon(a) &\text{ for } a< a_{m, \epsilon}, \\ e_\epsilon(a_{m,\epsilon}) + m(a-a_{m, \epsilon}) & \text{ for } a \geq a_{m, \epsilon} . \end{cases} 
\end{split}
\end{equation}
In addition, since $f_\ep$ is differentiable on $[0,+\infty)$,   so is $e_\ep$. Therefore, Lemma \ref{eepsproperties}(\ref{subdiffe}) and Lemma \ref{lem:truncation} (\ref{partialemepschar}) applied to $e_\ep$ imply,
\begin{equation}
    \label{eq:emep'}
    e_{m,\ep}'(a) = \min\{e_\ep'(a) , m\} =\min\{ f_\epsilon^*(f_{\epsilon}'(a)),m\}  \text{ for all $a\in (0,+\infty)$}.
\end{equation}

Next, we study the limits in $\ep$ of $e_\ep$ and $e_{m,\ep}$.

\begin{lem}[Limits in $\ep$ of the $\dot{H}^{-1}$ energy densities and their truncations] \label{ephm1limits}
Let $f$ satisfy items (\ref{convexlctsas}) and (\ref{nontrivial_convex}) of Assumption \ref{internalas} and let $f_\ep$ be given by equation (\ref{eq:eps energy}). Let $e$, $e_m$, $e_\ep$, and $e_{m,\ep}$ be defined via Definition \ref{hm1energydendef}, Definition \ref{trunk}, and (\ref{eq:def eepm}).  Then we have, 
    \begin{enumerate}[(i)]
    \item If $a\in [0,+\infty)\setminus \dom(e)$ then $\lim_{\epsilon\to 0} e_{\epsilon}'(a)=+\infty$.
    \item    $e_\epsilon$ converges pointwise to $e$.  \label{item:eep to e}
   \item \label{emepstoem} 
For all $m\in \N$, $e_{m,\epsilon}$ converges pointwise to $e_m$ as $\epsilon\to 0$.

\end{enumerate}
\end{lem}

\begin{proof}
Choose some $a\in [0,+\infty)\setminus \dom(e)$.
 The differentiability of $f_{\epsilon}$ on $[0,+\infty)$ implies that $e_{\epsilon}$ is differentiable on $(0,+\infty)$ and we may write \[
e'_{\epsilon}(a)=af_{\epsilon}'(a)-f_{\epsilon}(a)=\frac{\delta(\epsilon)}{2}a^2+a({}^{\delta(\epsilon)}f'(a))+{}^{\delta(\epsilon)}f(0)-{}^{\delta(\epsilon)}f(a).
\]
Therefore,
\[
\liminf_{\epsilon\to 0} e_{\epsilon}'(a)= \liminf_{\epsilon\to 0} a({}^{\delta(\epsilon)}f'(a))-{}^{\delta(\epsilon)}f(a)=\liminf_{\epsilon\to 0} ({}^{\delta(\epsilon)}f)^*\left({}^{\delta(\epsilon)}f'(a)\right).
\]
Using the fact that there exists $a_0\in (0,\infty)\cap \dom(f)$, it follows that
\[
\liminf_{\epsilon\to 0} ({}^{\delta(\epsilon)}f)^*\left({}^{\delta(\epsilon)}f'(a)\right)\geq \liminf_{\epsilon\to 0} a_0{}^{\delta(\epsilon)}f'(a)-{}^{\delta(\epsilon)}f(a_0).
\]
Since $\dom(e)$ is closed and $J_{\delta(\epsilon)}(a)\in \dom(e)$ for all $a\in [0,+\infty), \epsilon>0$, it follows  that $a-J_{\delta(\epsilon)}(a)>0$ independently of $\epsilon$ when $a\in [0,+\infty)\setminus \dom(e)$.  Therefore, \[
\liminf_{\epsilon\to 0} a_0{}^{\delta(\epsilon)}f'(a)=\liminf_{\epsilon\to 0}a_0\frac{a-J_{\delta(\epsilon)}(a)}{\delta(\epsilon)}=+\infty,
\]
which from our work above implies that 
$\liminf_{\epsilon\to 0} e_{\epsilon}'(a)=+\infty$ for all 
$a\in [0,+\infty)\setminus \dom(e)$.

Now we consider the pointwise convergence of $e_{\ep}$ to $e$.  By item (\ref{fep_to_f}) in Lemma \ref{reginternalenergydensityprop}, we know that $f_{\ep}$ converges pointwise everywhere to $f$, and uniformly on compact subsets of ${\rm int}(\dom(f))$.  From the definition of $e$, it is then clear that $e_{\ep}(a)$ converges to $e(a)$ for all $a\in \dom(e)$.  If $a<0$, then $e_{\epsilon}(a)=+\infty=e(a)$.  Thus, it remains to establish the convergence for $a\in [0,+\infty)\setminus \dom(e)$.  Again using the fact that $\dom(e)$ is closed, given $a\in [0,+\infty)\setminus \dom(e)$ we can find $a'\in [0,+\infty)\setminus \dom(e)$ such that $a'<a$.  We can then write
\[
e_{\epsilon}(a)=e_{\epsilon}(a')+\int_{a'}^a e_{\epsilon}'(\alpha)d\alpha\geq \int_{a'}^a e_{\epsilon}'(\alpha)d\alpha.
\]
Fatou's lemma then implies that
\[
\liminf_{\epsilon\to 0} e_{\epsilon}(a)\geq \int_{a'}^a \liminf_{\epsilon\to 0} e_{\epsilon}'(\alpha)d\alpha=+\infty.
\]
Thus, we have $\lim_{\epsilon\to 0} e_{\epsilon}(a)=+\infty=e(a)$ for all $a\in [0,+\infty)\setminus \dom(e)$ as desired.

Turning to the proof of the final item (\ref{emepstoem}),  we note that the pointwise convergence of $e_{\epsilon}$  to $e$ combined with the differentiability of $e_{\epsilon}$ on $[0,+\infty)$ implies that $e_{\epsilon}'(a)$ converges to $e'(a)$ whenever $a$ is a point of differentiability for $e$ (see for instance \cite[Lemma A.1]{jacobs2021existence}).  Furthermore, almost every point in $\dom(e)$ is a point of differentiability for $e$ and  $a\in \dom(e)$ implies that $[0,a]\subset \dom(e)$.  Therefore, given $a\in \dom(e)$, it follows that
\[
|e_{m,\epsilon}(a)-e_m(a)|\leq \int_0^a|e'_{m,\epsilon}(\alpha)-e_m'(\alpha)|d\alpha \leq \int_0^a \max(2m\, ,\, |e'_{\epsilon}(\alpha)-e'(\alpha)|)d\alpha,
\]
where we have used the fact that $b\mapsto\min(b,m)$ is a contraction on $\RR$.  Hence, the dominated convergence theorem implies that $\lim_{\epsilon\to 0} |e_{m,\epsilon}(a)-e_m(a)|=0$ for all $a\in\dom(e)$. 

 If $a\in [0,+\infty)\setminus \dom(e)$, then we saw from above that $\lim_{\epsilon\to 0} e_{\epsilon}'(a)=+\infty$ so $\lim_{\epsilon\to 0} e_{m,\epsilon}'(a)=m$.  It is also clear from the definition of $a_m$ that $e_m'(a)=m$ for $a\in [0,+\infty)\setminus \dom(e)$.  Thus, if $a^*:=\sup\dom(e)\neq +\infty$, then for all $a> a^*$, we can calculate
\[
|e_{m,\epsilon}(a)-e_m(a)|\leq \int_0^{a^*}|e'_{m,\epsilon}(\alpha)-e_m'(\alpha)|d\alpha+\int_{a^*}^{a} |e_{m,\epsilon}'(\alpha)-m|d\alpha,
\]
and we can use the same argument as above ($b\mapsto \min(b,m)$ is a contraction mapping along with the dominated convergence theorem) to conclude that 
$\lim_{\epsilon\to 0} |e_{m,\epsilon}(a)-e_m(a)|=0$.  Finally if $a<0$ then $e_{m,\epsilon}(a)=+\infty=e_m(a)$ so there is nothing to prove.  Thus, we are done.

\end{proof}

Finally, we record a few properties of the convex conjugates of the $\dot{H}^{-1}$ energy densities. We denote $e_m^*=(e_m)^*$ and $e_{m,\ep}^*=(e_{m,\ep})^*$.

\begin{lem}[Convex conjugates of $\dot{H}^{-1}$ energy densities]
    \label{lem:conjugate em}
Let $f$ satisfy items (\ref{convexlctsas}) and (\ref{nontrivial_convex}) of Assumption \ref{internalas} and let $f_\ep$ be given by equation (\ref{eq:eps energy}). Let $e$, $e_m$, $e_\ep$, and $e_{m,\ep}$ be defined via Definition \ref{hm1energydendef}, Definition \ref{trunk}, and (\ref{eq:def eepm}). We have,
\begin{enumerate}[(i)]
    \item \label{item:e* proper} For all $m\in \N$ and all $\ep>0$,  that $e^*$, $e_m^*$, and $e_{m,\ep}^*$ are proper, lower semicontinuous, and convex. Furthermore, $\interior \dom (e_m^*)\neq \emptyset$.
    \item \label{item:e*(0)}For all $m\in \N$ and all $\ep>0$,   $e^*(0)=0$, and $e_{m,\ep}^*(0)=0$.
    
    \item \label{item:em* decreasing} For all $a\in \R$ and for all $m\in \N$,  $e_m^*(a)\geq e^*(a)$.
\end{enumerate}
\end{lem}
\begin{proof}
    The first claim of item (\ref{item:e* proper}) follows from Lemma \ref{eepsproperties}(\ref{eepsconvex}), Lemma \ref{lem:truncation}(\ref{item:e_m_convex}), and the definition of convex conjugate. Next, we recall from Lemma \ref{lem:truncation}(\ref{domainitem}) that $e_m$ is nondecreasing on its domain. Using this in the definition of $e_m^*$ implies $e^*_m(b) \leq 0$ for all $b \leq 0$, which implies  ${\rm int}(\dom(e_m^*)) \neq \emptyset$. For item (\ref{item:e*(0)}), we recall that, according to Lemma \ref{eepsproperties}(\ref{eepsconvex}), $e(0)=0$ and $e$ is non-decreasing. 
    Thus, $e^*(0) =-\inf_a e(a)=-e(0)=0$. The same argument  shows $e_m^*(0)=0$.  Finally, 
    we recall from Lemma \ref{lem:truncation}(\ref{item:m_nondecreasing})  that $e_m(a)\leq e(a)$ for all $a\in \R$,  which implies, from the definition of convex conjugate, that (\ref{item:em* decreasing}) holds.
 
\end{proof}

\section{Properties of the $\epsilon$ flow}
\label{sec:ep flow}

In this section, we will establish properties of the  $\epsilon$-approximate equation \eqref{epsiloneqn}. {We begin by showing existence, uniqueness, and stability, which hold under relatively weak hypotheses on the initial data and velocity field. This stability result is a crucial ingredient of Corollary \ref{particlecorollary}, in which we extend our main convergence result, that solutions of (\ref{epsiloneqn}) converge to (\ref{PDE}), to the case of particle initial data.

Next, we  prove several dissipation properties of solutions, which for simplicity, we show under the stronger hypotheses that  $\rho_{\ep}^0\in C^{\infty}_c(\RR^d)\cap \PP_2(\R^d)$  and $v_{\epsilon}\in C^{\infty}_c([0,+\infty)\times \RR^d))$. These dissipation properties are then used in Section \ref{sec:convergence} to obtain our main convergence result, under the hypotheses that the velocity field and initial data are well-prepared; see Definition \ref{wellprepareddef}.}

\subsection{Existence, uniqueness, and stability}

We begin by showing that \eqref{epsiloneqn} is well-posed for fixed $\epsilon >0$.

\begin{lem}[Well-posedness of (\ref{epsiloneqn})] \label{wellposedpdeepslem}
Fix $\ep>0$. Suppose $f$ and $\varphi_\ep$ satisfy Assumptions \ref{internalas} and \ref{mollifieras}, and suppose we have initial data 
{$\rho_\epsilon^0\in \P_2(\Rd)$} and velocity field { $v_{\ep}\in C_b([0,+\infty);W^{1,\infty}(\Rd))$}. 
 Then there exists a unique $\rho_\epsilon \in AC^2_{\loc}([0,+\infty), \P_2(\Rd))$ satisfying (\ref{epsiloneqn}) with initial data $\rho^0_\epsilon$ and a unique flow map  $X_\ep \in W^{1,\infty}_{\loc}([0,+\infty);W^{1,\infty}_{\loc}(\RR^d))$
\begin{equation}
    \label{eq:Xep}
X_{\epsilon}(t,x)=x-\left( \int_0^t \nabla p_{\epsilon}(s,X_{\epsilon}(s,x)) + v_{\epsilon}(s,X_{\epsilon}(s,x))\right) \, ds,
\end{equation}
such that $\rho_\ep(t) = X_{\ep}(t)_{ \#}\rho_{\ep}^0$. Furthermore, for any two choices of initial data $\rho_\ep^{0,1}, \rho_\ep^{0,2} \in \P_2(\Rd)$, we have
\begin{align} \label{epsiloneqnstability} W_2(\rho_\ep^1(t), \rho_\ep^2(t)) \leq W_2(\rho^{0,1},\rho^{0,2}) \left( 1+C_\epsilon t e^{C_\epsilon t} \right),\end{align}
for  $C_\epsilon =2\norm{v_{\epsilon}}_{{L^\infty([0,+\infty);W^{1,\infty}}(\RR^d))}+\norm{\nabla \varphi_\ep}_{W^{1,1}(\Rd)} \left(\|f'_\epsilon\|_{\rm Lip([0,+\infty))} \norm{\vp_{\epsilon}}_{L^\infty(\Rd)} +f'_{\epsilon}(0)\right)$. 

Finally, if $\rho^0_\epsilon \in C^\infty_c(\Rd)$, we have $\partial_t X_{\epsilon}\in L^{\infty}_{\loc}([0,+\infty);W^{1,\infty}(\RR^d))$, $\rho_{\ep}\in L^{\infty}_{\loc}([0,+\infty);L^{\infty}(\RR^d))$, and, for every time $t\geq0$, $\rho_{\ep}(t)$ is compactly supported in space. 
\end{lem}

\begin{proof}
We use Lemma \ref{lem:wp continuity eqn} to establish well-posedness.  Here, for any $\rho\in \PP_2(\RR^d)$, our map $w$ is given by
\[
w(\rho) =\nabla (\vp_{\epsilon}*f'_{\epsilon}(\vp_{\epsilon}*\rho))+v_{\epsilon}.
\]
Since $v_{\ep}\in C_b([0,+\infty);W^{1,\infty}(\Rd))$, we also have $w(\rho) \in C_b([0,+\infty);W^{1,\infty}(\Rd))$, for all $\rho \in \P_2(\Rd)$.

To see that inequality (\ref{continuityinspace}) holds, note that Lemma \ref{reginternalenergydensityprop} (\ref{fepw2infty}) ensures there exists $C_\epsilon '>0$ so that, for all $\mu \in \P_2(\Rd)$ and $t \geq 0$,
\begin{align*}
 \norm{w(\mu)(t,\cdot)}_{{L^\infty([0,+\infty);W^{1,\infty}(\RR^d))}}&\leq \norm{v_{\epsilon}(t, \cdot)}_{{W^{1,\infty}}(\RR^d)}+  \norm{\nabla \varphi_\ep}_{W^{1,1}(\Rd)} \norm{f'_{\epsilon}(\vp_{\epsilon}*\mu)}_{L^{\infty}(\RR^d)}\\&\leq 
\norm{v_{\epsilon}}_{{L^\infty([0,+\infty);W^{1,\infty}(\RR^d))}}+\norm{\nabla \varphi_\ep}_{W^{1,1}(\Rd)} \left(\|f'_\epsilon\|_{\rm Lip([0,+\infty))} \norm{\vp_{\epsilon}*\mu}_{L^{\infty}(\RR^d)}+f'_{\epsilon}(0)\right) 
\\
&\leq 
\norm{v_{\epsilon}}_{{L^\infty([0,+\infty);W^{1,\infty}(\RR^d))}}+\norm{\nabla \varphi_\ep}_{W^{1,1}(\Rd)} \left(\|f'_\epsilon\|_{\rm Lip([0,+\infty))} \norm{\vp_{\epsilon}}_{L^\infty(\Rd)} +f'_{\epsilon}(0)\right)  \\
&\leq C_\epsilon ,
\end{align*}
where we use  $\int_\Rd \mu = 1$.

To see inequality (\ref{continuityofw}), note that, for any $\mu_1, \mu_2\in \P_2(\Rd)$,
\begin{align*}
\norm{w(\mu_1)-w(\mu_2)}_{L^{\infty}([0,+\infty) \times \RR^d)}&=\norm{\nabla \vp_{\epsilon} *\left(f'_{\epsilon}(\vp_{\epsilon}*\mu_1)-f'_{\epsilon}(\vp_{\epsilon}*\mu_2)\right) }_{L^{\infty}([0,+\infty) \times \RR^d)}\\
&\leq \norm{\nabla \varphi_\ep}_{L^1(\Rd)} \norm{f'_\epsilon}_{\rm Lip([0,+\infty))}  \norm{(\vp_{\epsilon}*\mu_1)-(\vp_{\epsilon}*\mu_2) }_{L^{\infty}([0,+\infty) \times \RR^d)}\\
&\leq \norm{\nabla \varphi_\ep}_{L^1(\Rd)} \norm{f'_\epsilon}_{\rm Lip([0,+\infty))}  \norm{\vp_{\epsilon}}_{\rm Lip(\Rd)} W_1(\mu_1,\mu_2)\\
&= C_{\epsilon}''W_2(\mu_1,\mu_2).
\end{align*}
Thus, $w$ satisfies the assumptions from Lemma \ref{lem:wp continuity eqn}, so the equation (\ref{epsiloneqn}) is well-posed and $\rho_\ep(t)=X_\ep(t)\#\rho^0_\ep$ holds. Likewise, we obtain the stability estimate (\ref{epsiloneqnstability}).

The regularity on $X_{\ep}$ follows from equation \eqref{eq:Xep} and standard Lagrangian flow theory; see for instance Lemma \ref{lem:DXep bds}. When $\rho_\ep \in C^\infty_c(\Rd)$, the regularity on $\rho_{\ep}$ follows from the pushforward formula and our assumptions $\rho^0_{\ep}$.  The pushforward formula also implies that $\rho_{\ep}(t)$ has compact support for any fixed time $t\geq 0$,  since $\rho^0_{\ep}$ has compact support and $\partial_t X\in L^{\infty}_{\loc}([0,+\infty);L^{\infty}(\RR^d))$.

\end{proof}

\subsection{Dissipation properties}

\subsubsection{Energy dissipation and second moment bound}

First, we recall the energy dissipation inequality for solutions of (\ref{epsiloneqn}). The proof is standard, and we defer it to the appendix.

\begin{lem}[Energy dissipation relation] \label{energydissipationrelationlem} Fix $\ep>0$. Suppose $f$ and $\varphi_\ep$ satisfy Assumptions \ref{internalas} and \ref{mollifieras}, and suppose we have initial data 
$\rho_\epsilon^0\in \P_2(\Rd)\cap C^{\infty}_c(\RR^d)$   and velocity field {$v_{\ep}\in C_b([0,+\infty);W^{1,\infty}(\Rd))$}. 
If $\rho_\epsilon$ is a solution of (\ref{epsiloneqn}) with initial data $\rho_{\ep}^0\in \PP_2(\RR^d)$, then for any $t\geq 0$ ,
\begin{equation}\label{eq:edr}
\mathcal{F}_{\ep}(\rho_\ep(t,\cdot))+\int_0^{t}\int_{\RR^d} \rho_{\epsilon}(|\nabla p_{\epsilon}|^2+v_{\ep}\cdot\nabla p_{\ep})= \mathcal{F}_{\ep}(\rho^0_{\ep}).
\end{equation}

\end{lem}

Now we use the energy dissipation relation to bound several important quantities. 

\begin{lem}[Energy and moment bounds] \label{energydissipationlem}
Suppose $f$ and $\varphi_\ep$ satisfy Assumptions \ref{internalas} and \ref{mollifieras}, and we have well-prepared initial data $\{\rho_\epsilon^0\}_{\epsilon\in (0,1)} \subseteq \P_2(\Rd)\cap C^{\infty}_c(\RR^d)$ and velocity fields $\{v_{\ep}\}_{\ep>0}\subseteq   C_b([0,+\infty);W^{1,\infty}(\Rd))$. 
For any  $T>0$  there exists a positive constant $C_{T,d}$ independent of $\epsilon$ so that, if $\rho_\epsilon$ is the solution of (\ref{epsiloneqn}) with initial data   $\rho_{\ep}^0$, then for all $t \in [0,T]$,
\begin{align} \label{EDI}
\cF_{\epsilon}(\rho_\ep(t,\cdot)) \leq C_{T,d},
\quad 
M_2(\rho_\epsilon(t,\cdot))\leq C_{T,d} , \quad \int_0^T\int_{\RR^d} \rho_{\epsilon}|\nabla p_{\epsilon}|^2\leq C_{T,d} .
\end{align}

\end{lem}

\begin{proof}

We begin by estimating the second moment.
Note that since $\rho_{\ep}(t)$ has compact support for all fixed $t\geq 0$, it is valid to integrate $\rho_{\ep}$ against $|x|^2$.
 Let
\[
{ M_{\epsilon}(t):= \int_{\RR^d} \frac{1}{2}(|x|^2+1)\rho_{\epsilon}(t,x)\, dx , \quad  \bar{M}_\epsilon(t) := \sup_{s\in [0,t]} M_\epsilon(s) }
\]
From (\ref{epsiloneqn}) and the Cauchy-Schwartz inequality, it follows that 
\[
M_{\epsilon}'(t)\leq \left( -\int_{\RR^d\times \{t\}} \rho_{\epsilon}(\nabla p_{\epsilon}+v_{\epsilon})\cdot x\right) _+\leq 2M_{\ep}(t)^{1/2}\norm{\rho_{\epsilon}(|\nabla p_{\ep}|^2+|v_{\ep}|^2)}_{L^1(\{t\}\times\RR^d)}^{1/2}.
\]
Integrating in time and applying the Cauchy-Schwartz inequality again gives
\begin{equation}\label{eq:moment_derivative}
M_{\epsilon}(t)\leq M_{\ep}(0)+  2\norm{M_{\ep}}_{L^1([0,t])}^{1/2}\norm{\rho_{\epsilon}(|\nabla p_{\ep}|^2+|v_{\ep}|^2)}_{L^1([0,t]\times\RR^d)}^{1/2}.
\end{equation}

Now we want to use the energy dissipation relation (\ref{eq:edr}) to eliminate $\norm{\rho_{\epsilon}|\nabla p_{\ep}|^2}_{L^1([0,t]\times\RR^d)}$. If we apply Young's inequality to the term $v_{\epsilon}\cdot \nabla p_{\epsilon}$ in (\ref{eq:edr}), we obtain the simpler inequality
\begin{equation}\label{eq:edi_1}
\cF_{\epsilon} (\rho_{\epsilon}(t,\cdot))+\int_0^t\int_{\RR^d} \frac{1}{2}\rho_{\epsilon}|\nabla p_{\epsilon}|^2\leq  \cF_{\epsilon}(\rho_{\ep}^0)+\int_0^t\int_{\RR^d} \frac{1}{2}\rho_{\epsilon}|v_{\epsilon}|^2 .
\end{equation}
Plugging this into (\ref{eq:moment_derivative}), we obtain
\begin{equation}\label{eq:moment_2}
M_{\epsilon}(t)\leq M_{\ep}(0)+  2\norm{M_{\ep}}_{L^1([0,t])}^{1/2}\left( 2\norm{\rho_{\epsilon}|v_{\ep}|^2}_{L^1([0,t]\times\RR^d)}+\cF_{\epsilon}(\rho_{\ep}^0)-\cF_{\epsilon} (\rho_{\epsilon}(t,\cdot))\right) ^{1/2}.
\end{equation}
Next, we eliminate $\norm{\rho_{\epsilon}|v_{\ep}|^2}_{L^1([0,t]\times\RR^d)}$ by noticing that
\begin{equation}\label{eq:v_control}
\norm{\rho_{\epsilon}|v_{\ep}|^2}_{L^1([0,t]\times\RR^d)}=\int_{[0,t]\times\RR^d} \rho_{\ep}(1+|x|^2)\frac{|v_{\ep}|^2}{1+|x|^2}\leq \bar{M}_{\ep}(t) \left\| \frac{|v_{\ep}|^2}{1+|x|^2}\right\|_{L^1([0,t];L^{\infty}(\RR^d))}.
\end{equation}
Thus, we have 
\begin{equation}\label{eq:moment_3}
M_{\epsilon}(t)\leq M_{\ep}(0)+  2\norm{M_{\ep}}_{L^1([0,t])}^{1/2}\left( 2\bar{M}_{\ep}(t)\left\| \frac{|v_{\ep}|^2}{1+|x|^2}\right\|_{L^1([0,t];L^{\infty}(\RR^d))}+\cF_{\epsilon}(\rho_{\ep}^0)-\cF_{\epsilon} (\rho_{\epsilon}(t,\cdot))\right) ^{1/2}.
\end{equation}

To handle $-\F_{\epsilon}(\rho_{\ep}(t,\cdot))$ we use Lemma \ref{lem:f_e_lower_bound} to guarantee the existence of a nonnegative, increasing, concave function $\tilde{H}$ (independent of $\epsilon$ and $\mue$) such that $\tilde{H}(0)=0$ and
\[
-\F_{\epsilon}(\rho_{\ep}(t,\cdot))\leq \tilde{H}\left( \int_{\RR^d} (1+|x|^2)\mue(t,x)\, dx\right) 
\]
Using (\ref{muepM2}) from Lemma \ref{muepqeplem}, there exists a  constant $C>0$ such that
\[
H\left( \int_{\RR^d} (1+|x|^2)\mue(t,x)\, dx\right) \leq H(CM_{\epsilon}(t)).
\]
 
Now we can use the concavity of $H$ and the fact that $M_{\epsilon}(t)\geq 1$ for all $t$, to obtain
\begin{equation}\label{eq:energy_lower_control}
-\F_{\epsilon}(\rho_{\ep}(t,\cdot))\leq H\left(  CM_{\epsilon}(t)\right) \leq H(C)+H'(C)M_{\epsilon}(t).
\end{equation}
Plugging this inequality into (\ref{eq:edi_1}), we see that
\begin{equation*}
M_{\epsilon}(t)\leq M_{\ep}(0)+  2\norm{M_{\ep}}_{L^1([0,t])}^{1/2}\left( \bar{M}_{\ep}(t) 2\left\|\frac{|v_{\ep}|^2}{1+|x|^2}\right\|_{L^1([0,t];L^{\infty}(\RR^d))}+M_\epsilon(t) H'(C) +\cF_{\epsilon}(\rho_{\ep}^0)+H(C)\right) ^{1/2}.
\end{equation*}
From our uniform control on $\cF_{\epsilon}(\rho_{\ep}^0)$, $M_{\ep}(0)$ and $\left\|\frac{|v_{\ep}|^2}{1+|x|^2}\right\|_{L^1([0,t];L^{\infty}(\RR^d))}$, as well as the fact that $M_{\ep}(t)\geq 1$, it follows that   there exists a constant $C_T$ independent of $\epsilon$ such that, for any $r \geq t$,
\begin{equation*}
M_{\epsilon}(t)\leq C_T\norm{M_{\ep}}_{L^1([0,t])}^{1/2}\bar{M}_{\ep}(t)^{1/2}  \leq C_T\norm{M_{\ep}}_{L^1([0,r])}^{1/2}\bar{M}_{\ep}(r)^{1/2}.
\end{equation*}
Thus, 
\[ \bar{M}_\epsilon(r) \leq C_T\norm{M_{\ep}}_{L^1([0,r])}^{1/2}\bar{M}_{\ep}(r)^{1/2} \implies \bar{M}_{\epsilon}(r)\leq C_T^2\norm{M_{\ep}}_{L^1([0,r])}. \]
Finally, we can use Gronwall to conclude that $\bar{M}_{\ep}$ is uniformly bounded with respect to $\epsilon$ on $[0,T]$. The remaining results follow directly from the second moment control and the inequalities (\ref{eq:edi_1}), (\ref{eq:v_control}), and (\ref{eq:energy_lower_control}). 
\end{proof}

\subsubsection{Entropy dissipation and generalized $\dot{H}^1$ bound}

\begin{prop}[Entropy dissipation] \label{genH1bound}
Suppose $f$ and $\varphi_\ep$ satisfy Assumptions \ref{internalas} and \ref{mollifieras}, and we have well-prepared initial data $\{\rho_\epsilon^0\}_{\epsilon\in (0,1)} \subseteq \P_2(\Rd)\cap C^{\infty}_c(\RR^d)$ and velocity fields $\{v_{\ep}\}_{\ep>0}\subseteq C^{\infty}_c([0,+\infty)\times\RR^d)$. 
For any  $T>0$,  there exists a positive constant $C_{d,T}$ independent of $\epsilon$ so that, if $\rho_\epsilon$ is the solution of (\ref{epsiloneqn}) with initial data   $\rho_{\ep}^0$, then for all $t \in [0,T]$,
\begin{align}
\label{entropydissipation}
\int_\Rd \rho_\ep | \log(\rho_\ep)|+\norm{\nabla \mue\cdot \nabla \qe}_{L^1([0,t]\times\RR^d)}\leq C_{d,T}.
\end{align}
\end{prop}

\begin{proof}
The lower bound of  Lemma \ref{lem:DXep bds}, our assumptions on $\rho^0_\epsilon$, and \cite[Lemma 5.5.3]{ambrosiogiglisavare} imply that, { for a.e. $x \in \Rd$},
\begin{align} \label{pushforwardformula}
\rho_\epsilon(t) = X_{\epsilon}(t)_{\#}\rho^0_\epsilon=\rho^0_\epsilon\circ X_{\epsilon}(t)^{-1}\det(DX_{\epsilon}(t)\circ X_{\epsilon}(t)^{-1})^{-1}.
\end{align} 
 Thus, 
\begin{equation}
\label{eq:rhologrho}
\rho_{\epsilon}(t)\log(\rho_{\epsilon}(t))=   X_{\epsilon}(t)_{\#}\rho^0_\epsilon\log(X_{\epsilon}(t)_{\#}\rho^0_\epsilon)= X_{\epsilon}(t)_{\#}\rho^0_\epsilon\log\left(\frac{\rho^0_\epsilon\circ X_{\epsilon}(t)^{-1}}{\det(DX_{\epsilon}(t))\circ X_{\epsilon}(t)^{-1}}\right)
\end{equation}
holds almost everywhere. 
The equality (\ref{eq:rhologrho}) yields,
\begin{align}
\mathcal{S}(\rho_{\ep}(t))&= \int_{\RR^d} X_{\epsilon}(t,x)_{\#}\rho^0_\epsilon(x) \log\left(\frac{\rho^0_\epsilon \circ X_{\epsilon}^{-1}(t,x)}{\det(DX_{\epsilon}(t,x))\circ X_{\epsilon}^{-1}(t,x)}\right)\, dx \nonumber \\
&=\int_{\RR^d} \rho^0_\epsilon(x)\log\left(\frac{\rho_\epsilon^0(x)}{\det\left(DX_{\epsilon}(t,x)\right)}\right)\nonumber\\
&=\int_{\RR^d} \rho^0_\epsilon(x)\log(\rho^0_\epsilon(x)) -\rho^0_\epsilon(x)\log\left(\det\left(DX_{\epsilon}(t,x)\right)\right).
\label{eq:g 2nd expression}
\end{align}

Next, we would like to take the derivative of $\mathcal{S}$ in time. From the formulas (\ref{eq:g 2nd expression}), (\ref{eq:partial t DX}) and the control in (\ref{eq:DXep bds}) we may write 
\begin{align*}
&\mathcal{S}(\rho_{\ep}(t_1))-\mathcal{S}(\rho_{\ep}(t_0))\\
&\quad =-\int_{\RR^d} \rho^0_\epsilon(x)\log\left(\det\left(DX_{\epsilon}(t_1,x)DX_{\epsilon}(t_0,x)^{-1}\right)\right)\\
&\quad =-\int_{\RR^d} \rho^0_\epsilon(x)\log\left(\det\left(I-\int_{t_0}^{t_1} (D^2 p_\ep(s,X_\epsilon(s,x))+D v_{\ep}(s, X_\epsilon(s,x)))D X(s,x) DX(t_0,x)^{-1}\, ds\right)\right).
\end{align*}
Using the matrix identity $\log(\det(A))=\tr(\log(A))$, %we have
%\begin{align*}
%\frac{\mathcal{S}(\rho_{\ep}(t_1))-\mathcal{S}(\rho_{\ep}(t_0))}{t_1-t_0}&=-\int_{\RR^d} \rho^0_\epsilon(x)\tr\left(\frac{1}{t_1-t_0}\log\left(I-\int_{t_0}^{t_1} \Delta p_\ep(s,X_\epsilon(s,x))+\nabla \cdot v_{\ep}(s, X_\epsilon(s,x))\, ds\right)\right).
%\end{align*}
Lemma \ref{muepqeplem}(\ref{item:pepbdd}), and the smoothness of $\rho^0_{\ep}$ and $v_{\ep}$, allow us to use the dominated convergence theorem to conclude, 
\begin{align*}
\frac{d}{dt} \mathcal{S}(\rho_{\ep}(t))&=\int_{\RR^d} \rho^0_\epsilon(x)\left( \Delta p_\ep(t,X_\epsilon(t,x))+\nabla \cdot v_{\ep}(t, X_\epsilon(t,x))\right)\\
 &=\int_{\RR^d} \rho_{\ep}(t,x)\left( \Delta p_\ep(t,x)+\nabla \cdot v_{\ep}(t, x)\right).
\end{align*}
 
Next, recalling the definitions of $p_\ep$, $\mue = \varphi_\epsilon * \rho_\epsilon$ and $\qe = f_\epsilon'(\mue) $  and  applying Lemma \ref{muepqeplem}~(\ref{qeplip}),   
\begin{align*}
 \int_{\RR^d\times \{t\}} \rho_{\epsilon} \Delta p_{\epsilon}&=     \int_{\RR^d\times \{t\}} \rho_\ep \Delta \left(\varphi_\ep *f'_\ep(\varphi_\ep *\rho_\ep)\right) =  \int_{\RR^d\times \{t\}}  \rho_\ep \left(\nabla \varphi_\ep * \nabla f'_\ep (\varphi_\ep *\rho_\ep)\right) \\
 &=  -\int_{\RR^d\times \{t\}} (\rho_\ep*\nabla\varphi_\ep) \left(\nabla f'_\ep(\varphi_\ep*\rho_\ep)\right)=  -\int_{\RR^d\times \{t\}} \nabla \mue \cdot \nabla \qe \\
 &= -\int_{\RR^d\times \{t\}} |\nabla \mue \cdot \nabla \qe|. 
\end{align*}
We can also estimate 
\begin{align*}
 \int_{\RR^d\times \{t\}} \rho_{\epsilon} \nabla \cdot v_{\ep}\leq \norm{\frac{(\nabla \cdot v_{\ep})_+}{1+|x|^2}}_{L^{1}(\{t\}L^{\infty}(\RR^d))}(1+M_2(\rho_{\ep}(t))).  
 \end{align*}
Thus, we have shown that
\begin{equation}\label{eq:almost_entropy_dissipation}
    \mathcal{S}(\rho_{\ep}(t))+\int_{[0,t]\times\RR^d} |\nabla \mue \cdot \nabla \qe|\leq \mathcal{S}(\rho_{\ep}^0)+\norm{\frac{(\nabla \cdot v_{\ep})_+}{1+|x|^2}}_{L^{1}([0,t];L^{\infty}(\RR^d))}(1+ \sup_{s \in [0,t]}M_2(\rho_{\ep}(s))).
\end{equation}
%To complete the proof, we need to show that we can control the negative part of $\log(\rho_{\ep})$. To this end, we use the fact that for any $\alpha>0$, 
%\[
%\sup_{a\in (0,1)}-a^{\alpha}\log(a)\leq \frac{1}{\alpha e} \implies \sup_{a\in (0,1)}-a\log(a)\leq \frac{a^{1-\alpha}}{\alpha e} .
%\]
%Thus, for any $\alpha\in (0,1)$
%\[
 % \int_\Rd  \rho_\epsilon(t) (\log (\rho_\epsilon(t)))_- \leq \frac{1}{\alpha e} \int_{ \RR^d} \rho_{\epsilon}^{1-\alpha}(t).
%\]
%If  $\alpha  <\frac{2}{d+2}$, then it follows from  Lemma \ref{lem:moments_give_integrability} and the fact that second moments control $p^{th}$ moments for $p\in (0,2)$
%that there exists a constant $C_d$ such that
Estimate (\ref{lem:moments_give_integrability}) from Lemma \ref{lem:integrability via moments} applied with, say, $\alpha=1/(d+2)$, implies that there exists a constant $C_d$ such that
\[
 \int_\Rd  \rho_\epsilon(t) (\log (\rho_\epsilon(t)))_- \leq C_d \left(M_{2}(\rho_{\epsilon}(t,\cdot))+1\right)^{1-\alpha}.
\]
Combining this with (\ref{eq:almost_entropy_dissipation}) we have 
\begin{align*}
    &\int_{\{t\}\times\RR^d} \rho_{\ep}|\log(\rho_{\ep})|+\int_{[0,t]\times\RR^d} |\nabla \mue \cdot \nabla \qe|\leq \\  &\mathcal{S}(\rho_{\ep}^0)+\norm{\frac{(\nabla \cdot v_{\ep})_+}{1+|x|^2}}_{L^{1}([0,t];L^{\infty}(\RR^d))}(1+\sup_{s \in [0,T]}M_2(\rho_{\ep}(s))+2 C_d \left(M_{2}(\rho_{\epsilon}(t,\cdot))+1\right)^{1-\alpha}.
\end{align*}
The result now follows since $\rho^0_\ep$ and $v_\ep$ are well prepared, as well as from the bound on $M_2(\rho_\ep(t))$ of Lemma \ref{energydissipationlem}.
\end{proof}

Next, we record the following useful consequence of Lemma \ref{energydissipationlem} and Proposition \ref{genH1bound}. 

\begin{lem}[Estimate on $\nabla q_\ep$] \label{lem:mu_nabla_q}
Suppose $f$ and $\varphi_\ep$ satisfy Assumptions \ref{internalas} and \ref{mollifieras} and that the  initial data $\{\rho_\epsilon^0\}_{\epsilon\in (0,1)} \subseteq \P_2(\Rd)\cap C^{\infty}_c(\RR^d)$ and velocity fields $\{v_{\ep}\}_{\ep>0}\subset C^{\infty}_c([0,+\infty)\times\RR^d)$ are well-prepared. Let $\rho_\epsilon$ be the corresponding solution of (\ref{epsiloneqn}) and let $\mu_\ep = \varphi_\ep*\rho_\epsilon$ and $q_\ep = f_\ep'(\mu_\ep)$. Then there exists a constant $C_{T,d}>0$  independent of $\ep$ so that
\begin{align} \label{lem:mu_nabla_qeq1}
    \norm{\nabla \qe}_{L^2([0,T]\times\RR^d)}^2\leq C_{T,d} (\delta(\epsilon)^{-1}+\delta(\epsilon)).
\end{align}
\end{lem}

\begin{proof}
By Lemma \ref{muepqeplem}(\ref{qeplip}), $|\nabla q_\ep \cdot \nabla \mu_\ep| = |\nabla q_\ep ||\nabla \mu_\ep|$. This equality, followed by Lemma \ref{muepqeplem}(\ref{nabla_q_nabla_mu}), yields
\begin{equation}
\label{eq:bd nabla qe}
|\nabla \qe|= |\nabla \qe|^{1/2} \frac{|\nabla q_\ep \cdot \nabla \mu_\ep|^{1/2}}{|\nabla \mue|^{1/2}}\leq (\delta(\epsilon)^{-1}+\delta(\epsilon))^{1/2}|\nabla \mue\cdot\nabla \qe|^{1/2}
\end{equation}
    almost everywhere.  
 Therefore, 
 \[
 \norm{\nabla \qe}_{L^2([0,T]\times\RR^d)}^2\leq (\delta(\epsilon)^{-1}+\delta(\epsilon))\norm{\nabla \qe\cdot\nabla \mue}_{L^1([0,T]\times\RR^d)}.
 \]
 The result now follows from applying Proposition \ref{genH1bound} to control $\norm{\nabla \qe\cdot \nabla \mue}_{L^1([0,T]\times\RR^d)}$. %Here, we also use the uniform bound on the second moment from Lemma \ref{energydissipationlem} and the fact that the entropy can be bounded from below in terms of the second moment  (estimate (\ref{lem:moments_give_integrability}) from Lemma \ref{lem:integrability via moments}).

\end{proof}

\section{Convergence Proof}\label{sec:convergence}

\subsection{Mollifier exchange}

As in previous work on the porous medium case\cite{CarrilloCraigPatacchini, craig2022blob, carrillo2023nonlocal}, a key step in the convergence proof is that we can exchange the mollifier between the velocity field $\nabla p _\epsilon = \nabla \varphi_\epsilon * (f'_\ep(\varphi_\epsilon * \rho_\epsilon)) $ and the density $\rho_\epsilon$. See also equation (\ref{Cdotdef}),  where we recall the notation $\dot{C}^{\theta}(\Rd)$.

\begin{prop}[Mollifier exchange]
\label{lem:mollifier exchange}
Suppose $f$ and $\varphi_\ep$ satisfy Assumptions \ref{internalas} and \ref{mollifieras} and that the  initial data $\{\rho_\epsilon^0\}_{\epsilon\in (0,1)} \subseteq \P_2(\Rd)\cap C^{\infty}_c(\RR^d)$ and velocity fields $\{v_{\ep}\}_{\ep>0}\subseteq C^{\infty}_c([0,+\infty)\times\RR^d)$ are well-prepared. Let $\rho_\epsilon$ be the corresponding solution of (\ref{epsiloneqn}) with pressure $p_\epsilon$, and let $\mu_\ep = \varphi_\ep*\rho_\epsilon$ and $q_\ep = f_\ep'(\mu_\ep)$.

Then there exists $C_{d, T} >0$ so that, for all  $\theta\in (0,1)$  and $g \in L^2([0,T] ; \dot{C}^{\theta}( \Rd))$,  we have {for $\epsilon>0$ sufficiently small}
\begin{align} \label{eq:mollifierexchange}
\left| \int_0^T\int_{\RR^d}     g  \left[ \rho_\ep  \nabla p_\ep-\mu_\ep \nabla q_\ep \right] \, dx\, dt\right| \leq C_{\varphi,d, T} \epsilon^{\frac{\theta(r-d)}{r}}\delta(\epsilon)^{-\frac{r-\theta}{r}} \norm{g}_{L^2([0,T];\dot{C}^{\theta}(\RR^d))}.
\end{align}
\end{prop}
\begin{proof}

In order to prove (\ref{eq:mollifierexchange}), let us begin by considering a single time slice.  To simplify notation,  we will completely suppress the time dependence of the variables until the very end of the proof.
Expanding the convolutions, we  estimate as follows,
\begin{align}
\left|  \int_{\RR^d} g  \left[ \rho_\ep  \nabla p_\ep-\mu_\ep \nabla q_\ep \right]    \right|  
& = \left| \iint_{\Rd \times \Rd} \varphi (z) g(x) \left[ \rho_\ep(x) \nabla q_\ep(x+\ep z) - \rho_\ep(x+\ep z) \nabla q_\ep(x) \right] dz dx \right| \nonumber \\
& = \left| \iint_{\Rd \times \Rd} \rho_\ep(x+\ep z) \nabla q_\ep(x) \varphi (z) \left[ g(x+\ep z) - g(x) \right] dx dz \right|  \nonumber \\
&\leq \label{eq:mollifier_exchange_1}
\epsilon^{\theta} \norm{g}_{\dot{C}^{\theta}(\RR^d)}\iint_{\Rd \times \Rd} \rho_{\epsilon}(x+\ep z)|\nabla \qe(x)|\vp(z)|z|^{\theta}\, dz\, dx,
\end{align}
where the second equality follows from the change of variables $z \mapsto -z$ and $x \mapsto x+\epsilon z$ in the first term, using the fact that $\vp$ is even, and the inequality holds for any $\theta\in(0,1)$.

We will now focus on estimating the inner $z$ integral in (\ref{eq:mollifier_exchange_1}).
To do so, we will fix some $b>0$ and split the domain of integration into $|z|\leq b$ and $|z|>b$.  For $|z|\leq b$ we have the estimate
\begin{equation*} 
\int_{|z|\leq b} \rho_{\epsilon}(x+\ep z)\vp(z)|z|^{\theta}\, dz\leq b^{\theta}\int_{\RR^d} \rho_{\epsilon}(x+\ep z)\vp(z)\, dz=b^{\theta}\mue(x).
\end{equation*}
On the other hand,  for $|z|>b$ we have the estimate
\begin{align*}  
\int_{|z|> b} \rho_{\epsilon}(x+\ep z)\vp(z)|z|^{\theta}\, dz\leq \left(\sup_{|z|>b}|z|^{\theta}\vp(z)\right)\int_{\RR^d} \rho_{\epsilon}(x+\ep z)\, dz = \left(\sup_{|z|>b}|z|^{\theta}\vp(z)\right)\epsilon^{-d}\int_{\RR^d} \rho_{\epsilon}(y) dy
\end{align*}
where we change variables by setting $y=x+\epsilon z$.
 
Combining the estimates for $|z|\leq b$ and $|z|>b$ and using $\int \rho_\ep =1$, it follows that for any $b>0$,
\begin{equation}\label{eq:z_integral_estimate_total}
    \int_{\RR^d} \rho_{\epsilon}(x+\ep z)\vp(z)|z|^{\theta}\, dz\leq b^{\theta}\mue(x)+ \epsilon^{-d}\left(\sup_{|z|>b}|z|^{\theta}\vp(z)\right) \leq b^\theta \mu_\ep(x) + \ep^{-d} C_\varphi b^{\theta -r} \ , 
\end{equation}
where, in the second inequality, we use Assumption \ref{mollifieras}(\ref{secondmolas}) for   $r>\max\{d,2\}$.

Now, we  optimize over the choice of $b$ in inequality (\ref{eq:z_integral_estimate_total}). Letting $b=(C_\varphi \epsilon^{-d}\mue(x)^{-1})^{1/r}$
gives
\begin{equation}\label{eq:z_integral_estimate}
\int_{\RR^d} \rho_{\epsilon}(x+\ep z)\vp(z)|z|^{\theta}\, dz\leq 2C_\varphi^{\theta/r}\epsilon^{-d\theta/r}\mue(x)^{\frac{r-\theta}{r}}.
\end{equation}

Combining  estimate (\ref{eq:z_integral_estimate}) with inequality (\ref{eq:mollifier_exchange_1}) and rearranging terms, we see that 
\begin{align}
\left|  \int_{\RR^d} g  \left[ \rho_\ep  \nabla p_\ep-\mu_\ep \nabla q_\ep \right]    \right| &\leq 2 C_\varphi \epsilon^{\frac{\theta(r-d)}{r}}\norm{g}_{\dot{C}^{\theta}(\RR^d)}\int_{\RR^d}\mue(x)^{\frac{r-\theta}{r}} |\nabla \qe(x)|\, dx  \nonumber \\
&\leq 2 C_\varphi \epsilon^{\frac{\theta(r-d)}{r}}\norm{g}_{\dot{C}^{\theta}(\RR^d)}\norm{\nabla \qe}_{L^2(\RR^d)}  \| \mue^{\frac{2(r-\theta)}{r}} \|_{L^1(\RR^d)}^{1/2} \nonumber \\
&\leq 2 C_\varphi \epsilon^{\frac{\theta(r-d)}{r}}\norm{g}_{\dot{C}^{\theta}(\RR^d)}\norm{\nabla \qe}_{L^2(\RR^d)}\norm{\mue}_{L^2(\RR^d)}^{\frac{r-2\theta}{r}}, \label{eq:mollifier_exchange_4}
\end{align}
where the last inequality follows by Jensen's inequality: since $r>\max(d,2)\geq 2\theta$, we have
\[ 0<\frac{2(r-\theta)}{r}-1<1 \implies \norm{\mue^{\frac{2(r-\theta)}{r}}}_{L^1(\RR^d)}^{1/2}\leq \norm{\mue}_{L^2(\RR^d)}^{\frac{r-2\theta}{r}} . \]

At last, we are ready to integrate over the time slices and reintroduce the time variable. Note that, by  Lemmas \ref{lem:f_e_lower_bound} and \ref{energydissipationlem},   $\forall \ t\in [0,T]$,
\[
\norm{\mue}_{L^2(\{t\}\times \RR^d)}^2\lesssim_{d,T}  \delta(\epsilon)^{-1}.
\] 
Thus, there exists a constant $C_{\varphi,d,T}$ so that, integrating (\ref{eq:mollifier_exchange_4}) in time, we obtain
\begin{align*}
\int_0^T\Big|\int_{\RR^d}     g  \left[ \rho_\ep  \nabla p_\ep-\mu_\ep \nabla q_\ep \right]   \big|  &\leq 2 C_\varphi \epsilon^{\frac{\theta(r-d)}{r}} \int_0^T \norm{g}_{\dot{C}^{\theta}(\{t\}\times \RR^d)}\norm{\nabla \qe}_{L^2(\{t\}\times\RR^d)}\norm{\mue}_{L^2(\{t\}\times\RR^d)}^{\frac{r-2\theta}{r}}\, dt \\
 &\leq  C_{\varphi,d,T} \epsilon^{\frac{\theta(r-d)}{r}}\delta(\epsilon)^{-\frac{r-2\theta}{2r}}\int_0^T \norm{g}_{\dot{C}^{0,\theta}(\{t\}\times \RR^d)}\norm{\nabla \qe}_{L^2(\{t\}\times\RR^d)} dt \\
 &\leq C_{\varphi,d,T} \epsilon^{\frac{\theta(r-d)}{r}}\delta(\epsilon)^{-\frac{r-2\theta}{2r}} \norm{g}_{L^2([0,T];\dot{C}^{\theta}(\RR^d))}\norm{\nabla \qe}_{L^2([0,T]\times\RR^d)}.
\end{align*}
The result then follows from  Lemma \ref{lem:mu_nabla_q}.

\end{proof}

\subsection{Compactness properties}

First, we prove a lemma regarding the uniform regularity in time of solutions to (\ref{epsiloneqn}) with well-prepared initial data and velocities. 
\begin{lem}[Time regularity of (\ref{epsiloneqn})] \label{rhoe_time}
Suppose $f$ and $\varphi_\ep$ satisfy Assumptions \ref{internalas} and \ref{mollifieras} and that the  initial data $\{\rho_\epsilon^0\}_{\epsilon\in (0,1)} \subseteq \P_2(\Rd)\cap C^{\infty}_c(\RR^d)$ and velocity fields $\{v_{\ep}\}_{\ep>0}\subset C^{\infty}_c([0,+\infty)\times\RR^d)$ are well-prepared. Let $\rho_\epsilon$ be the corresponding solution of (\ref{epsiloneqn}).

Then, for all $T>0$, 
\begin{equation}\label{eq:rho_time}
\sup_{\epsilon>0} \,\norm{\rho_{\ep}}_{C^{1/2}([0,T];\P_1(\RR^d))}+\norm{\rho_{\ep}}_{H^1([0,T];W^{-1,1}(\RR^d))}<+\infty.
\end{equation}
\end{lem}
\begin{proof}
Letting $X_\ep$ be as in  Lemma \ref{wellposedpdeepslem}, and applying  Lemma \ref{wellposedpdeepslem}, we find, for all $0 \leq t \leq s \leq T$ and $\epsilon>0$,
\begin{align}
 W_1( \rho_\epsilon(t),\rho_\epsilon(s)) &\leq  \int_\Rd |X_\epsilon(t,x) - X_\epsilon(s,x)|  \rho_\epsilon^0(x) dx  \nonumber \\
&\leq  \int_t^{s} \int_\Rd |\nabla p_\epsilon(r, X_\epsilon(r,x)) + v_\epsilon(r, X_\epsilon(r,x))|  \rho_\epsilon^0(x) dx dr \nonumber \\
&=  \int_t^{s} \int_\Rd |\nabla p_\epsilon(r,x) + v_\epsilon(r,x) |  \rho_\epsilon(r,x) dx dr.   \label{timeregeqn}
\end{align}

Therefore,
\begin{align*}
 W_1( \rho_\epsilon(t),\rho_\epsilon(s)) &\leq  \left( \int_t^{s} \int_\Rd |\nabla p_\epsilon + v_\epsilon|^2  \rho_\epsilon \right)^{1/2} \left( \int_t^{s} \int_\Rd1  \rho_\epsilon \right)^{1/2} \\
&\leq  (s-t)^{1/2}\left( \int_t^{s} \int_\Rd 2(|\nabla p_\epsilon|^2 + |v_\epsilon|^2)  \rho_\epsilon \right)^{1/2} .
\end{align*} 
By the energy dissipation and second moment bounds from Lemma \ref{energydissipationlem} and our uniform control on $v_\epsilon$, from Assumption \ref{wellprepareddef}, we have
\begin{align} \label{mepsbound}  \sup_{\epsilon>0}  \int_0^{T} \int_\Rd 2(|\nabla p_\epsilon|^2 + |v_\epsilon|^2)  \rho_\epsilon  <+\infty
\end{align}
Thus, $\sup_{\epsilon>0} \,\norm{\rho_{\ep}}_{C^{1/2}([0,T];\P_1(\RR^d))} <+\infty$.

 Next, note that the fact that the $W^{-1,1}(\Rd)$ norm is bounded above by the 1-Wasserstein distance (see inequality (\ref{dualsobolevtoW2})), combined with inequality (\ref{timeregeqn}), ensures, for all $0 \leq t \leq s \leq T$,
 \begin{align*}
 \| \rho_\epsilon(t) - \rho_\epsilon(s)\|_{W^{-1,1}(\Rd)} \leq \int_t^s m_\epsilon(r) dr \quad  \text{ for } \quad  m_\epsilon(r) :=\int_\Rd |\nabla p_\epsilon(r,x) + v_\epsilon(r,x) |  \rho_\epsilon(r,x) dx .
 \end{align*}
 Furthermore, inequality (\ref{mepsbound}) and Jensen's inequality ensure that 
 \[ \sup_{\epsilon>0} \|m_\epsilon\|_{L^2([0,T])} <+\infty.\]
 Thus, $\rho_\epsilon \in AC^2([0,T]; W^{-1,1}(\Rd))$. Then, by  \cite[Remark 1.1.3]{ambrosiogiglisavare}, for almost every $t \in [0,T]$, $\rho_\epsilon(t)$ is differentiable with respect to $W^{-1,1}(\Rd)$, $\dot{\rho}_\epsilon \in L^2([0,T]; W^{-1,1}(\Rd))$, and $\|\dot{\rho}_\epsilon(t)\|_{W^{-1,1}(\Rd)} \leq |m_\epsilon(t)|$ for a.e. $t \in [0,T]$. Therefore,
 \[
 \sup_{\epsilon >0} \|\dot{\rho}_\epsilon\|_{L^2([0,T]; W^{-1,1}(\Rd))} \leq \sup_{\epsilon >0}  \|m_\epsilon\|_{L^2([0,T])} <+\infty.
 \]
 Combining this with the fact that
 \[ \| \rho_\epsilon\|_{W^{-1,1}(\Rd)} = \sup_{\|\psi\|_{W^{1,\infty}(\Rd)} \leq 1} \int \psi \rho_\epsilon \leq \sup_{\|\psi\|_{L^{\infty}(\Rd)} \leq 1} \int \psi \rho_\epsilon = 1 ,\]
 we have
  \[
  \sup_{\epsilon >0} \|{\rho}_\epsilon\|_{L^2([0,T]; W^{-1,1}(\R^d))} \leq T^{1/2} .
  \]
 Therefore, $\sup_{\epsilon>0} \norm{\rho_{\ep}}_{H^1([0,T];W^{-1,1}(\RR^d))} <+\infty$.
  \end{proof}

 Next, we establish some compactness properties of the densities.
 \begin{lem}[Compactness for the densities] \label{rhoepscpt}
Suppose $f$ and $\varphi_\ep$ satisfy Assumptions \ref{internalas} and \ref{mollifieras} and that the  initial data $\{\rho_\epsilon^0\}_{\epsilon\in (0,1)} \subseteq \P_2(\Rd)\cap C^{\infty}_c(\RR^d)$ and velocity fields $\{v_{\ep}\}_{\ep>0}\subset C^{\infty}_c([0,+\infty)\times\RR^d)$ are well-prepared. Let $\rho_\epsilon$ be the corresponding solution of (\ref{epsiloneqn}), and let $\mu_\epsilon = \varphi_\epsilon * \rho_\epsilon$.
\begin{enumerate}[(i)]

\item For all $t \geq 0$, $ d\rho_\epsilon(t)   = \rho_\epsilon(t,x) dx  $ for  $\rho_\epsilon(t, \cdot) \in L^1(\Rd)$  and $ \{\rho_\epsilon(t)\}_{\epsilon>0}, \{\mu_\epsilon(t)\}_{\epsilon>0} \subseteq L^1( \Rd)$   are uniformly integrable.   \label{rhoepunifint}

    \item There exists $\rho \in C_{\loc}^{1/2}([0,+\infty);\P_1(\Rd))\cap L^{\infty}_{\loc}([0,+\infty);\P_2(\RR^d)\cap L^1(\R^d))$ with $d\rho(t)=\rho(t,x)\, dx$ so that, up to a subsequence, $\rho_\epsilon(t) \to \rho(t)$ and $\mu_{\epsilon}(t)\to \rho(t)$ in 1-Wasserstein and weakly in $L^1(\R^d)$,  for all $t \geq 0$. \label{aalimitrho}

\item   $\sup_{t \in [0,T]} \mathcal{S}(\rho(t))  < + \infty$ and $\sup_{t \in [0,T]}  \mathcal{F}(\rho(t))   <+\infty$.\label{lem_item:entropy_energy_control}
\end{enumerate}
 \end{lem}

 \begin{proof}
First, we show (\ref{rhoepunifint}). By the   moment bounds, Lemma \ref{energydissipationlem}, and the entropy dissipation inequality, Proposition \ref{genH1bound}, we have for any $T>0$
 \begin{align} \label{pointwisetimesecondmomentropybound}
 \sup_{\epsilon>0}\sup_{t\in [0,T]} M_2(\rho_{\ep}(t))+\S(\rho_\epsilon(t))<+\infty .
 \end{align}
 Thus,
 \[ d\rho_\epsilon(t)(x) dt = \rho_\epsilon(t,x) dxdt \text{ for }\rho_\epsilon \in L^\infty([0,T] ; L^1( \Rd)). \]
  We will use $\rho_\epsilon$ to both denote the element of $AC^2([0,T];\P_2(\Rd))$, on the left hand side, and the element of $L^\infty([0,T] ; L^1( \Rd))$, on the right hand side.   By the bounds in inequality (\ref{pointwisetimesecondmomentropybound}), we have that $\{ \rho_\epsilon(t)\}_{\epsilon >0}$ is precompact in $\P_1(\Rd)$ and $L^1(\R^d)$; see, for example,   \cite[Proposition 7.1.5, equation (5.2.20)]{ambrosiogiglisavare} and Lemma \ref{dunfordpettislemma}.   Since we also have 
  \begin{align}  \label{M2muepbound}
 \sup_{\epsilon>0}\sup_{t\in [0,T]} M_2(\mu_{\ep}(t))+\S(\mu_\epsilon(t))<+\infty
 \end{align}
 by Lemma \ref{muepqeplem}(\ref{muepM2}) and Jensen's inequality for $s \mapsto s \log s$,
   the same precompactness holds true for $\{\mue(t)\}_{\ep>0}$.

Now we  will show (\ref{aalimitrho}) by applying the Arzel\'a-Ascoli theorem \cite[Proposition 3.3.1]{ambrosiogiglisavare}.
On any time interval $[0,T]$, Lemma \ref{rhoe_time} implies that $\rho_\epsilon$ is uniformly H\"older continuous with respect to 1-Wasserstein, hence equicontinuous. As already shown above, $\{\rho_\ep(t)\}_{\ep >0}$ is also precompact, for all $t \geq 0$.  Therefore, by   Arzel\'a-Ascoli, there exists $\rho \in C_{\loc}([0,+\infty); \P_1(\Rd))$ so that, up to a subsequence, $\rho_\epsilon(t) \to \rho(t)$ in 1-Wasserstein for all $t \geq 0$. Since $\mu_\epsilon = \varphi_\epsilon *\rho_\epsilon$, we likewise have $\mu_\epsilon(t) \to \rho(t)$ narrowly for all $t \geq 0$ \cite[Lemma 2.3]{CarrilloCraigPatacchini}. As shown above, since $\{\mue(t)\}_{\ep>0}$ is precompact in $\mathcal{P}_1(\Rd)$, we conclude that $\mu_\ep(t) \to \rho(t)$ in 1-Wasserstein for all $t \geq 0$.

Lower semicontinuity of the 1-Wasserstein metric then implies that $\rho \in C^{1/2}_{\loc}([0,+\infty); \P_1(\Rd))$.   Finally, note that, by lower semicontinuity of $\mu \mapsto M_2(\mu)$ with respect to 1-Wasserstein convergence \cite[Proposition 5.1.7]{ambrosiogiglisavare} we have $M_2(\rho(t)) \leq \liminf_{\ep \to 0} M_2(\rho_\ep(t)) \leq \sup_{\ep >0, t \in[0,T]} M_2(\rho_\ep(t)) < +\infty$, for all $t \in [0,T].$ Thus $\rho \in L^{\infty}_{\loc}([0,+\infty);\P_2(\RR^d))$.

Before proving the rest of (\ref{aalimitrho}), we turn to the first part of (\ref{lem_item:entropy_energy_control}). This follows immediately from the fact that   $\mathcal{S}$ is  lower semicontinuous with respect to 1-Wasserstein convergence \cite[Remark 9.3.8]{ambrosiogiglisavare}, so $ \sup_{t \in [0,T]} \mathcal{S}(\rho(t))\leq \sup_{t \in [0,T]} \liminf_{\epsilon\to 0} \mathcal{S}(\rho_{\ep}(t))<\infty$.  
Returning back to (\ref{aalimitrho}),  we see that the preceeding bound on the entropy implies $\rho(t)$ is absolutely continuous with respect to Lebesgue measure for all $t \geq 0$, and we we may write $d\rho(t)=\rho(t,x)\, dx$.  Since $\rho(t)$ has mass one for all $t \geq 0$, this shows $\rho \in L^{\infty}_{\loc}([0,+\infty);L^1(\RR^d))$.

To complete our proof of (\ref{aalimitrho}), it suffices to show the convergence of $\rho_\ep(t)$ and $\mu_\ep(t)$ in weak $L^1(\Rd)$ for any $t \geq 0$. As argued above, $\{ \rho_\ep(t) \}_{\ep >0}, \{ \mu_\ep(t) \}_{\ep >0}$ are precompact weakly in $L^1(\Rd)$. Since their 1-Wasserstein convergence to $\rho(t)$ implies convergence in distribution, by uniqueness of limits, we likewise have convergence weakly in $L^1(\Rd)$.

It remains to show the bound $\sup_{t\in[0,T]}\mathcal{F}(\rho(t)) <+\infty$ from part (\ref{lem_item:entropy_energy_control}).
By Lemma \ref{energydissipationlem}, it suffices to show there exists $C_T>0$ so that $\F(\rho(t)) \leq \liminf_{\epsilon \to 0} \F_\epsilon(\rho_\epsilon(t)) +C_T $ for all $t \in [0,T]$. Let $B_R$ be the ball of radius $R>0$ centered at the origin. 
 
By Lemma \ref{lem:f_e_lower_bound}, there exists a concave, increasing function $\tilde{H}:[0,+\infty)\to [0,+\infty)$ such that 
\[
\int_{ B_R^c} f_{\epsilon}(\mue(t))\geq -\tilde{H}\left(\int_{B_R^c} (1+|x|^2)\mue(t,x)\, dx\right)  .
\]
Therefore,
\[
\mathcal{F}_{\epsilon}(\rho_{\epsilon}(t))\geq  \int_{  B_R} f_{\epsilon}(\mue(t))-\tilde{H}\left(\int_{B_R^c} (1+|x|^2)\mue(t,x)\, dx\right)\geq \int_{  B_R} f_{\epsilon}(\mue(t))-\tilde{H}\left(1+M_2(\mue(t))\right) .
\]

By Lemma \ref{reginternalenergydensityprop}(\ref{fep_to_f}) $f_{\epsilon}$ epi-converges  to $f$. Furthermore, Assumption \ref{internalas}(\ref{convexlctsas}), Lemma \ref{reginternalenergydensityprop}(\ref{fepstronglyconvex}), and Lemma 2.2 ensure  the hypotheses of Lemma \ref{lem:varying_lower_semicontinuity} are satisfied. Thus, for all $t \geq 0$,
\[ 
\liminf_{\epsilon\to 0}\int_{ B_R} f_{\epsilon}(\mue(t))\geq \int_{ B_R} f(\rho(t)).
\]
Hence, 
\begin{align} \label{almostdonewithFrho}
\liminf_{\epsilon\to 0}\mathcal{F}_{\epsilon}(\rho_{\epsilon}(t))\geq \int_{  B_R} f(\rho(t))- \sup_{\ep >0, t \in[0,T]}\tilde{H}\left(1+ M_2(\mu_\ep(t))\right).
\end{align}
Finally, by applying Assumption \ref{internalas}(\ref{moment_control}) to $\tilde{\rho} = \rho \mathbf{1}_{\{x: f(\rho(x))<0\}}$, which satisfies $M_2(\tilde{\rho}) \leq M_2(\rho) <+\infty$, we see that
\[ -\int (f(\rho))_-  = \int f(\tilde{\rho}) \geq -H \left( \int_\Rd (1+|x|^2) d \tilde{\rho}(x) \right) >-\infty . \]
Thus, by the monotone convergence theorem,
\[ \int f(\rho) = \int (f(\rho))_+ - \int (f(\rho))_- = \lim_{R \to +\infty} \int_{B_R} (f(\rho))_+ - \int (f(\rho))_- \leq \limsup_{R \to +\infty} \int_{B_R}f(\rho) .\]
Combining this with inequality (\ref{almostdonewithFrho}), we obtain for $C_T:= \sup_{\ep >0, t \in[0,T]}\tilde{H}\left( 1+M_2(\rho_\ep(t))\right)>0$, where $C_T<+\infty$  by  (\ref{pointwisetimesecondmomentropybound}),
\begin{align*}
    \liminf_{\epsilon\to 0}\mathcal{F}_{\epsilon}(\rho_{\epsilon}(t))\geq \limsup_{R \to +\infty}\int_{  B_R} f(\rho(t))-C_T \geq \int f(\rho(t)) -C_T = \F(\rho(t)) -C_T,
\end{align*}
completing the proof.

 \end{proof}

Our final compactness result concerns the $\dot{H}^{-1}$ dual variable $f_\ep^*(q_\ep)$, which we denoted in the introduction by $\zeta_\ep$; see equation (\ref{eq:zeta def}).
\begin{lem}[Compactness for the $\zeta_\ep$]\label{lem:f_*_integrable}
Suppose $f$ and $\varphi_\ep$ satisfy Assumptions \ref{internalas} and \ref{mollifieras} and that the  initial data $\{\rho_\epsilon^0\}_{\epsilon\in (0,1)} \subseteq \P_2(\Rd)\cap C^{\infty}_c(\RR^d)$ and velocity fields $\{v_{\ep}\}_{\ep>0}\subseteq C^{\infty}_c([0,+\infty)\times\RR^d)$ are well-prepared. Let $\rho_\epsilon$ be the corresponding solution of (\ref{epsiloneqn}), and let $\mu_\epsilon = \varphi_\epsilon * \rho_\epsilon$ and $q_\ep = f_\ep'(\mu_\ep)$.

 For $r$ as in Assumption \ref{mollifieras}, assume there is some $\theta \in (0,1)$ so that
\begin{align} \label{epsilondeltarelationbounded2} \lim_{\ep \to 0 } \epsilon^{\frac{\theta(r-d)}{r}}\delta(\epsilon)^{-\frac{r-\theta}{r}} <+\infty. 
\end{align}
Then, for any $T>0$ and $r^*\in (1, \frac{d}{d-1+\theta})$
\begin{align} \label{Lrbound} \sup_{\epsilon >0 }\norm{\nabla f^*_{\ep}(\qe)}_{L^2([0,T];\dot{W}^{-\theta, 1}(\RR^d))}+\norm{f^*_{\epsilon}(\qe)}_{L^{1}([0,T]\times \RR^d)}+\norm{f^*_{\epsilon}(\qe)}_{L^{2}([0,T];L^{r^*}(\RR^d))}< +\infty . 
\end{align}
 Furthermore, $\{f^*_{\epsilon}(\qe)\}_{\ep>0}$ is weakly precompact in $L^1_{\loc}([0,+\infty)\times\RR^d)$. 
\end{lem}

\begin{proof}
We begin by proving the first bound in  (\ref{Lrbound}).   Note that, combining the fact that $\rho_\ep$ is a probability measure with the uniform    bounds from Lemma \ref{energydissipationlem}, we have 
\begin{align*}
\sup_{\epsilon >0} \int_0^T \norm{(1+|\cdot|)^2\rho_{\ep}}_{L^1(\{t\}\times\RR^d)}  \norm{\rho_{\ep}|\nabla p_{\ep}|^2}_{L^1(\{t\}\times\RR^d)}  \, dt  <+\infty  . 
\end{align*}
Furthermore, for any $h \in \dot{C}^{\theta}(\RR^d)$ with $h(0) = 0$, we have $|h(x)| \leq \norm{h}_{\dot{C}^{\theta}(\RR^d)} (1+|x|)$. Therefore, for any $g \in L^{2}([0,T]; \dot{C}^{\theta}(\RR^d))$,
\begin{align}
&\int_0^T \int_\Rd (g(t,x) - g(t,0)) \rho_\ep(t,x) \nabla p_\ep(t,x) dx dt \leq
\int_0^T \|g(t)\|_{\dot{C}^{\theta}(\Rd)}  \int_\Rd (1+|x|)\rho_{\ep}(t,x) |\nabla p_{\ep}(t,x)| dx  dt \nonumber \\
&\quad \leq \norm{g}_{L^{2}([0,T]; \dot{C}^{\theta}(\RR^d))} \left( \int_0^T\norm{(1+|\cdot|)^2\rho_{\ep}}_{L^1(\{t\}\times\RR^d)}  \norm{\rho_{\ep}|\nabla p_{\ep}|^2}_{L^1(\{t\}\times\RR^d)} dt \right)^{1/2} ,
 \end{align}
where the right hand side is bounded uniformly in $\epsilon$.
  Combining this   with the mollifier exchange bound, Lemma  \ref{lem:mollifier exchange}, we conclude that 
\begin{align} \label{withconstantterm}
\sup_{\ep >0}  \ \sup_{\norm{g}_{L^{2}([0,T]; \dot{C}^{\theta}(\RR^d))}\leq 1} 
\int_0^T\int_{\RR^d}     (g(t,x)-g(t,0))  \mu_\ep(t,x) \nabla q_\ep(t,x)  <+\infty .
\end{align}
By Lemma \ref{muepqeplem}(\ref{qeplip}), $\mue\nabla \qe = \nabla f_{\ep}^*(\qe)$ almost everywhere, where $f_\ep^*(q_\ep) \in W^{1,\infty}(\Rd)$. Furthermore, recalling the duality relation
\begin{align} 
\label{recallduality} 
f^*_{\ep}(\qe)=\mue \qe-f_{\ep}(\mue)
\end{align} 
from Lemma \ref{muepqeplem} (\ref{FYidentity}), and the facts that $\mu_\ep \in L^1(\Rd)$, $\qe \in L^\infty(\Rd)$, and $f_{\ep}(\mue) \in L^1(\Rd)$, we see that  $f_\ep^*(\qe) \in L^1(\Rd)$. Likewise, since  $ \nabla \qe \in L^\infty(\Rd)$, we also have $\nabla f_{\ep}^*(\qe) \in L^1(\Rd)$, so  $f_{\ep}^*(\qe) \in W^{1,1}(\Rd)$.  Together with the divergence theorem, this ensures that $\int_\Rd \nabla f_{\ep}^*(\qe) = 0$ for almost every $t \in [0,T]$. Thus, using this to cancel out the spatially constant term in inequality (\ref{withconstantterm}), we obtain
\begin{align*}
\sup_{\ep >0}  \ \sup_{\norm{g}_{L^{2}([0,T]; \dot{C}^{\theta}(\RR^d))}\leq 1} 
\int_0^T\int_{\RR^d}     g(t,x)  \nabla f_{\ep}^*(\qe)     <+\infty .
\end{align*}
This shows  $\nabla f_{\ep}^*(\qe)\in L^{2}([0,T];\dot{W}^{-\theta,1}(\RR^d))$ uniformly in $\epsilon >0$.

Now, we prove the remaining two estimates in  inequality (\ref{Lrbound}).
Choose some $r^*\in (1, \frac{d}{d-1+\theta})$.
Note that, for any $B>0$,
\begin{align}\label{letsestimatel1}
\norm{f^*_{\epsilon}(\qe)}_{L^{1}([0,T]\times \RR^d)}\leq B^{1-r^*}\norm{f^*_{\epsilon}(\qe)}_{L^{1}([0,T];L^{r^*}(\RR^d))}^{r^*}+\norm{f^*_{\epsilon}(\qe)}_{L^{1}\left(Q_B\right)}
\end{align}
where $Q_B=\{(t,x)\in [0,T]\times\RR^d: f^*_{\ep}(\qe)\leq B\}$.
By Assumption \ref{internalas}(\ref{nontrivial_convex}), there exists $a_0>0$ such that $f(a_0)\neq +\infty$. Thus, for all $\epsilon >0$ we have $f^*_{\ep}(b)\geq a_0b-f_{\epsilon}(a_0)\geq \frac{a_0b}{2}$  once $b\geq 0$ is sufficiently large.  Thus, there exists some $B_0\geq 0$ such that $Q_B\subset \{(t,x)\in [0,T]\times\RR^d: \qe \leq \frac{2\max(B,B_0)}{a_0}\}$.  Combining this with the duality relation (\ref{recallduality}), it follows that there exists $C_{d,T}>0$ so that
\begin{align*}
\norm{f^*_{\epsilon}(\qe)}_{L^{1}\left(Q_B\right)}&\leq \frac{2\max(B,B_0)}{a_0}\norm{\mue}_{L^1([0,T]\times\RR^d)}-\int_0^T \mathcal{F}_{\ep}(\rho_{\ep}) \leq \frac{2\max(B,B_0)T}{a_0}+C_{d,T},
\end{align*}
where we have used Lemma \ref{f_e_lower_bound} and our uniform bounds on the second moments from Lemma \ref{energydissipationlem} and Lemma \ref{muepqeplem} (\ref{muepM2}) to control the internal energy. 
Combining these estimates with inequality (\ref{letsestimatel1}), we obtain
\[
\norm{f^*_{\epsilon}(\qe)}_{L^{1}([0,T]\times \RR^d)}\leq B^{1-r^*}\norm{f^*_{\epsilon}(\qe)}_{L^{1}([0,T];L^{r^*}(\RR^d))}^{r^*}+\frac{2 (B+B_0)T}{a_0}+C_{d,T}.
\]
Taking $B = \|f_\ep^*(q_\ep)\|_{L^{1}([0,T];L^{r^*}(\RR^d))} $ then gives
\[
\norm{f^*_{\epsilon}(\qe)}_{L^{1}([0,T]\times \RR^d)}\leq (1+2T/a_0)\norm{f^*_{\epsilon}(\qe)}_{L^{1}([0,T];L^{r^*}(\RR^d))} +\frac{2TB_0}{a_0}+C_{d,T} .
\]

Next, by fractional Sobolev interpolation inequalities (see, for instance, Lemma \ref{lem:f_equi}) {and H\"older's inequality in time}, there exists $C_{ \theta}>0$ so that  
\begin{equation}\label{eq:fractional_gagliardo_f_star}
\norm{f^*_{\epsilon}(\qe)}_{L^{1}([0,T];L^{r^*}(\RR^d))}\leq  C_{\theta}\norm{f^*_{\epsilon}(\qe)}_{L^{1}([0,T];L^1(\RR^d))}^{1-\frac{d(r^*-1)}{r^*(1-\theta)}}\norm{\nabla f^*_{\ep}(\qe)}_{L^1([0,T];\dot{W}^{-\theta,1}(\RR^d))}^{\frac{d(r^*-1)}{r^*(1-\theta)}} .
\end{equation}
Thus, there exists $C_{d,T,\theta}>0$ so that
\begin{align*}
\norm{f^*_{\epsilon}(\qe)}_{L^{1}([0,T]\times \RR^d)}\leq C_{d,T,\theta} \norm{f^*_{\epsilon}(\qe)}_{L^{1}([0,T]\times \RR^d)}^{1-\frac{d(r^*-1)}{r^*(1-\theta)}}\norm{\nabla f^*_{\ep}(\qe)}_{L^1([0,T];\dot{W}^{-\theta,1}(\RR^d))}^{\frac{d(r^*-1)}{(1-\theta)}}  +C_{d,T,\theta}.
\end{align*}
Since $0<1-\frac{d(r^*-1)}{r^*(1-\theta)}<1$, we can bound 
$\norm{f^*_{\epsilon}(\qe)}_{L^{1}([0,T] \times \RR^d)}$ in terms of $\norm{\nabla f^*_{\ep}(\qe)}_{L^2([0,T];\dot{W}^{-\theta,1}(\RR^d))}$. This then  implies via (\ref{eq:fractional_gagliardo_f_star}) that  $\norm{f^*_{\epsilon}(\qe)}_{L^{1}([0,T];L^{r^*}(\RR^d))}$ can be bounded in terms of $\norm{\nabla f^*_{\ep}(\qe)}_{L^2([0,T];\dot{W}^{-\theta,1}(\RR^d))}$, which completes the proof of  inequality (\ref{Lrbound}).

Finally, the uniform bound on $\norm{f^*_{\epsilon}(\qe)}_{L^{1}([0,T]\times \RR^d)}+\norm{f^*_{\epsilon}(\qe)}_{L^{2}([0,T];L^{r^*}\RR^d))}$ for any $r^*\in (1, \frac{d}{d-1+\theta})$ implies that $f^*_{\ep}(\qe)$ is weakly precompact in $L^{\beta}([0,T]\times\RR^d)$ for any $\beta\in (1, \min(2, \frac{d}{d-1+\theta}))$.  In particular,  $f_{\ep}^*(\qe)$ is weakly precompact in $L^1_{\loc}([0,+\infty)\times\RR^d)$. 
\end{proof}

\subsection{Identifying the limit}

Now that we have shown that we can extract subsequential limits of the  $\rho_\ep$ and $\zeta_\ep=f_\ep^*(q_\ep)$, we will establish that the limiting quantities  $\rho$ and $\zeta$ satisfy the desired duality relations (Proposition \ref{identifylimitpoint} and Lemma \ref{existenceofpwow}.) Then, we will conclude the section with the proof of Theorem \ref{mainthm}. 

First, we show that  $\rho\in \partial e(\zeta)$ holds almost everywhere.

\begin{prop}[Duality relation between $\rho$ and $\zeta$] \label{identifylimitpoint} 
Under the hypotheses of Lemma \ref{lem:f_*_integrable}, suppose $\rho_\epsilon$ is the solution of (\ref{epsiloneqn}), and let $\mu_\epsilon = \varphi_\epsilon * \rho_\epsilon$ and $q_\ep = f_\ep'(\mu_\ep)$.
Then, if $(\rho, \zeta)$ is a weak $L^1_\loc([0,+\infty) ; L^1(\RR^d))\times L^1_\loc([0,+\infty)\times \RR^d)$ limit point of $(\mue, f^*_{\ep}(\qe))$, we have 
\begin{align} \label{dualityrelationwoohoo}
 \rho\zeta=e(\rho)+e^*(\zeta) \text{ almost everywhere on } [0,+\infty) \times \Rd \ , 
 \end{align}
 where $e: \R \to \R \cup \{+\infty\}$ is the $\dot{H}^{-1}$ energy density from Definition \ref{hm1energydendef}.
\end{prop}

\begin{proof}
Fix a compact set $E\subseteq [0,+\infty)\times\RR^d$, and let $\epsilon_k$  be a subsequence such that $(\mu_{\ep_k}, f^*_{\ep_k}(q_{\ep_k}))$ converges weakly in $L^1(E)\times L^1(E)$ to $(\rho, \zeta)$.  Recall from Lemma \ref{convexconjincreasing} and the properties of $f_\epsilon$ shown in Lemma \ref{reginternalenergydensityprop} (\ref{fepstronglyconvex})  that $f_\epsilon^*$ is nonnegative and nondecreasing for all $\epsilon >0$. For any $m\in \mathbb{N}$, define 
\[ \zeta_{\epsilon,m}:=\min(f^*_{\epsilon}(q_{\epsilon}), m). \]
 Since $0 \leq \zeta_{\epsilon,m} \leq f_\epsilon^*(q_\epsilon)$, by the Dunford-Pettis theorem, the weak $L^1_{\loc}([0,+\infty)\times\RR^d)$ compactness of $ f_\epsilon^*(q_\epsilon)$ proved in Lemma \ref{lem:f_*_integrable}   ensures the weak $L^1_{\loc}([0,+\infty)\times\RR^d)$ compactness of $\zeta_{\epsilon, m}$, for each $m \in \mathbb{N}$; see for example,  \cite[Theorem 4.30]{BrezisFunctionalAnalysis}.

Define $\epsilon_k^0:=\epsilon_k$ and, for each $m\in  \mathbb{N}$, we choose $\epsilon^m_k$ to be a subsequence of $\epsilon_k^{m-1}$ such that $\zeta_{\epsilon_k^m,m}$ converges weakly in $L^1(E)$ to some limit point $\zeta_m$ as $k\to+\infty$. For simplicity of notation, let $\zeta_{k,m}:= \zeta_{\epsilon_k^m,m}=\min(f^*_{\epsilon_k^m}(q_{\epsilon_k^m}), m)$.  By construction, $\zeta_m\leq \zeta_{m+1}\leq\zeta$ almost everywhere and 
   \begin{align*}
&   \lim_{m\to\infty} \norm{\zeta-\zeta_m}_{L^1(E)}=\lim_{m\to\infty}\int_{E} \zeta-\zeta_m=\lim_{m\to\infty}\lim_{k\to\infty}\int_{E} \max(f^*_{\ep_k^m}(q_{\ep_k^m})-m,0) .
\end{align*}
We may bound   the preceding equation from above by %the following $L^p([0,T] \times \Rd)$ norms, with $p = r^*$ and its H\"older conjugate $(r^*)'$,
%\begin{align*}
% &   \lim_{m \to \infty} \lim_{k \to  \infty} \left\|\mathbf{1}_{\{f^*_{\epsilon^m_k}(q_{\epsilon^m_k}) \geq m \}} \right\|_{ (r^*)'} \left\| \max(f^*_{\ep_k^m}(q_{\ep_k^m})-m,0) \right\|_{r^*}   \\ 
% &\quad \leq \lim_{m \to \infty} \lim_{k \to  \infty} |\{f^*_{\epsilon^m_k}(q_{\epsilon^m_k}) \geq m \}|^{1/(r^*)'} \left\| f^*_{\ep_k^m}(q_{\ep_k^m} )\right\|_{r^*}\\
%   &\quad \leq
%   \lim_{m\to\infty} \sup_{\epsilon>0} \frac{\norm{f^*_{\epsilon}(q_{\epsilon})}_{r^*}^{r^*/(r^*)'}}{m^{r_*/(r^*)'}} \left\| f^*_{\ep_k^m}(q_{\ep_k^m} )\right\|_{r^*} \\
 %  &\quad \leq
\begin{align*}
   \lim_{m\to\infty} \sup_{\epsilon>0}  \norm{f^*_{\epsilon}(q_{\epsilon})}_{r^*}^{r^*}m^{1-r_*} =0 ,
   \end{align*}
where we use the facts that $r^*<2$ and, by Lemma \ref{lem:f_*_integrable}, $\sup_{\ep > 0} \norm{f^*_{\epsilon}(\qe)}_{L^{2}([0,T];L^{r^*}(\RR^d))}< +\infty$. Thus, we conclude that 
\begin{align}\label{strongconvergencezetam}   \lim_{m\to\infty} \norm{\zeta-\zeta_m}_{L^1(E)} = 0 . 
\end{align}
  
Now we fix $m$ and consider the product $\mu_{\ep_k^m}\zeta_{k,m}$. We seek to apply  Lemma \ref{lem:compensated_compactness}.
By our definition of $\zeta_{k,m}$, 
\[ \zeta_{k,m}  :=\min(f^*_{\epsilon_k^m}(q_{\epsilon_k^m}), m) \]
and the integrability of $f_\ep^*(q_\ep)$ shown in Lemma \ref{lem:f_*_integrable}, we see that $\zeta_{k,m}$ satisfies the integrability and boundedness hypotheses of $h_k$ in Lemma \ref{lem:compensated_compactness}.
To obtain the required Besov norm bound,  we first note that for any $u\in B_{1,\infty}^{\frac{1-\theta}{2}}(\R^d)$ and any Lipschitz function $\ell:\RR\to\RR$ such that $\ell(0)=0$, we have $\norm{\ell\circ h}_{B_{1,\infty}^{\frac{1-\theta}{2}}(\R^d)}\leq \norm{\nabla \ell}_{L^{\infty}(\R^d)}\norm{h}_{B_{1,\infty}^{\frac{1-\theta}{2}}(\R^d)}$.  Therefore,
\begin{align*}
\sup_{k\in \N}\norm{\zeta_{k,m}}_{L^1([0,T];B_{1,\infty}^{\frac{1-\theta}{2}}(\R^d))}&\leq \sup_{k\in \N}\norm{f^*_{\epsilon_k^m}(q_{\epsilon_k^m})}_{L^1([0,T];B_{1,\infty}^{\frac{1-\theta}{2}}(\R^d))}\\
&\leq \sup_{\epsilon>0} \norm{f^*_{\epsilon}(q_{\epsilon})}_{L^1([0,T];B_{1,\infty}^{\frac{1-\theta}{2}}(\R^d))}\\
&\leq  \sup_{\epsilon>0} C_{d,\theta}\norm{\nabla f^*_{\epsilon}(q_{\epsilon})}_{L^1([0,T];\dot{W}^{-\theta,1}(\RR^d))}^{1/2}\norm{f^*_{\epsilon}(q_{\epsilon})}_{L^{ 1}([0,T]\times \RR^d)}^{1/2} +\norm{f^*_{\epsilon}(q_{\epsilon})}_{L^{ 1}([0,T]\times \RR^d)}\\
&\leq  \sup_{\epsilon>0} C_{d,\theta,T}\norm{\nabla f^*_{\epsilon}(q_{\epsilon})}_{L^2([0,T];\dot{W}^{-\theta,1}(\RR^d))}^{1/4}\norm{f^*_{\epsilon}(q_{\epsilon})}_{L^{ 1}([0,T]\times \RR^d)}^{1/2} +\norm{f^*_{\epsilon}(q_{\epsilon})}_{L^{ 1}([0,T]\times \RR^d)} ,
\end{align*}
where we use Lemma \ref{lem:f_equi} { and H\"older's inequality in time} in the second to last inequality.  Hence, using the $L^2([0,T];\dot{W}^{-1,\theta}(\RR^d))$ bound from Lemma \ref{lem:f_*_integrable}, we see that $\zeta_{k,m}$ satisfies all of the hypotheses of $h_k$ in Lemma \ref{lem:compensated_compactness}.
Likewise, $\mu_{\epsilon_k}^m$ satisfies the integrability hypotheses of $g_k$ and, from Lemma \ref{rhoe_time}, we have
\[
\sup_{\epsilon>0}\norm{\mue}_{H^1([0,T];W^{-1,1}(\RR^d))}\leq \sup_{\epsilon>0}\norm{\rho_{\ep}}_{H^1([0,T];W^{-1,1}(\RR^d))}<+\infty .
\]
Thus,  Lemma \ref{lem:compensated_compactness} ensures
\begin{equation}\label{eq:product_limit_above}
 \liminf_{k\to\infty} \int_{E} \mu_{\ep_k^m}\zeta_{k,m}\leq \int_{E} \rho \zeta_m .
\end{equation}
 
Now, recall the construction of the $\dot{H}^{-1}$ energy densities $e$,  $e_\epsilon$ and their truncations $e_m$, $e_{m,\epsilon}$ from Section \ref{sss:e}. In what follows, we shall define $e_{m,k}:=e_{m, \epsilon_k^m}$ for notational convenience. By definition of $\qe = f_\ep'(\mu_\ep)$ and equation (\ref{eq:emep'}),  we have 
\[ e_{m,k}'(\mu_{\ep_k^m}) = \min\{ f_{\ep_k}^*(f'_{\ep_k}(\mu_{\ep,k})), m \} =  \min\{ f_{\ep_k}^*(q_{\ep_k}  ), m \} =  \zeta_{k,m} \quad  \text{for all $(t,x)$ so that $\mu_{\ep_k}(t,x)>0$,} \]
or, equivalently,  
\[ \mu_{\ep_k^m}\zeta_{k,m}=e_{m,k}(\mu_{\ep_k^m})+e^*_{m,k}(\zeta_{k,m})  \quad  \text{for all $(t,x)$ so that $\mu_{\ep_k}(t,x)>0$.} \]
On the other hand, recall that since the right derivative of $f_\epsilon$ at $0$ satisfies $f_\epsilon'(0) \in \partial f_\epsilon(0)$,  by the Fenchel-Young identity,
\[ 0 =  f_\epsilon'(0) \cdot 0  = f_\epsilon(0) + f_\epsilon^*(f_\epsilon'(0)) = f_\epsilon^*(f_\epsilon'(0)) , \]
where we use the result from Lemma \ref{reginternalenergydensityprop}(\ref{fepstronglyconvex}) that $f_\ep(0) = 0$.
Thus, if $\mu_{\ep_k^m}(t,x) = 0$, then $\zeta_{k,m}(t,x) = \min( f^*_{\ep_k^m}(q_{\ep^m_k}(t,x)), m) = \min( f^*_{\ep_k^m}(f_{\ep_k^m}'(\mu_{\ep_k^m}(t,x))), m) = 0 $. We recall $e_{m,k}(0) = 0$ and $e_{m,k}^*(0)=0$ (the former follows from Lemma \ref{lem:truncation}(\ref{item:e_m_convex}) , applied to $e_{\ep_m^k}$ instead of $e$, and the latter from Lemma \ref{lem:conjugate em}(\ref{item:e*(0)}).)  Therefore, we have
\[ \mu_{\ep_k^m}\zeta_{k,m}=e_{m,k}(\mu_{\ep_k^m})+e^*_{m,k}(\zeta_{k,m})  \quad  \text{for all }(t,x) \in [0,+\infty) \times \Rd . \]
%where we use that $e_{m,k}(0) = 0$ Lemma \ref{eepsproperties}(\ref{eepsconvex}) and Lemma \ref{lem:conjugate em}(\ref{item:e*(0)}) 
%that $e_{m,k}(0) = 0$, since $e_k(0)=0$ and $0\leq e_{m,k}(a)\leq e_k(a)$ for all $a$; see Lemma \ref{eepsproperties} and the definition of $e_{m,k}$ in equation (\ref{eq:def eepm}).

Combining this with  (\ref{eq:product_limit_above}), we obtain
\[
 \liminf_{k\to\infty} \int_{E}  \left(e_{m,k}(\mu_{\ep_k^m})+e^*_{m,k}(\zeta_{k,m})\right) =  \liminf_{k\to\infty} \int_{E} \mu_{\ep_k^m}\zeta_{k,m}  \leq \int_{E} \rho \zeta_m .
\]
 
We now seek to apply  Lemma \ref{lem:varying_lower_semicontinuity} to each of $e_{m,k}$ and $e_{m,k}^*$ to pass to the limit as $k\to +\infty$ on the left hand side. By Lemma \ref{lem:truncation} (\ref{item:e_m_convex}) and Lemma \ref{lem:conjugate em}(\ref{item:e* proper}), $e_{m,k}$, $e_{m,k}^*$,  $e_m$, and $e_m^*$ are proper, lower semicontinuous, and convex. 
Lemma  \ref{lem:truncation} (\ref{domainitem}) %ensures ${\rm int} (\dom(e_m)) \neq \emptyset$.
 and Lemma \ref{lem:conjugate em}(\ref{item:e* proper}) ensure the domains of $e_m$ and $e_m^*$ have nonempty interior. %$\interior(\dom(e_m^*)) \neq \emptyset$. 
By Lemma \ref{ephm1limits} (\ref{emepstoem}),  $e_{m,k}$ converges pointwise to $e_m$. Thus $e_{m,k}$ epi-converges to $e_m$ \cite[Theorem 7.17]{rockafellar2009variational}, hence $e_{m,k}^*$ epi-converges to $e_m^*$ \cite[Theorem 11.34]{rockafellar2009variational}. Finally, since by Lemma \ref{rhoepscpt} (\ref{aalimitrho}), $\mu_{\ep_k^m}(t) \to \rho(t)$ weakly in $L^1(\Rd)$ for all $t \geq 0$, the dominated convergence theorem ensures $\mu_{\ep_k^m} \to \rho$ weakly in  $L^1(E)$. Combining this with the weakly $L^1(E)$ convergence of $\zeta_{k,m}$ to $\zeta_m$, we  obtain, by Lemma \ref{lem:varying_lower_semicontinuity},
\[
\int_{E} \rho \zeta_m\geq \liminf_{k\to\infty} \int_{E}  \left(e_{m,k}(\mu_{\ep_k^m})+e^*_{m,k}(\zeta_{k,m})\right)\geq \int_{E} e_m(\rho)+e_m^*(\zeta_m) .
\]
 Combining this with Young's inequality, it follows that $\rho\zeta_m= e_m(\rho)+e_m^*(\zeta_m)$ almost everywhere on $E$.
Since $E$ was arbitrary, we  conclude that, almost everywhere on $[0,+\infty)\times \RR^d$,
\[
\rho\zeta_m= e_m(\rho)+e_m^*(\zeta_m).
\]

It remains to show that $\rho\zeta=e(\rho)+e^*(\zeta)$ almost everywhere. Using Lemma \ref{lem:conjugate em}(\ref{item:em* decreasing}) to bound the right-hand side of the previous line from below yields, 
%By definition, $e_m$ is increasing with respect to $m$ and $e_m \leq e$, so $e_m^*$ is decreasing with respect to $m$ and $e_m^* \geq e^*$. 
almost everywhere on $[0,+\infty)\times \RR^d$,
\[
\rho\zeta_m%=e_m(\rho)+e^*_m(\zeta_m)
\geq e_m(\rho)+e^*(\zeta_m).
\]
As shown in equation  (\ref{strongconvergencezetam}), $\zeta_m$ converges strongly to  $\zeta$ in $L^1(E)$, for $E \subseteq [0,+\infty)\times \RR^d$ arbitrary, so up to a subsequence, $\zeta_m \to \zeta$ almost everywhere. By Lemma \ref{lem:truncation} (\ref{emtoe}), $\lim_{m\to\infty} e_m(a)= e(a)$ for any $a\in \RR$, so by the lower semicontinuity of $e^*$, we have $\rho\zeta\geq e(\rho)+e^*(\zeta)$ almost everywhere.  Combining this with Young's inequality we deduce that $\rho\zeta=e(\rho)+e^*(\zeta)$ almost everywhere. 
 
\end{proof}

Now that we have recovered the duality relation between $\rho$ and the $\dot{H}^{-1}$ dual variable $\zeta$ in Proposition \ref{identifylimitpoint}, we can do the same for the $W_2$ dual variable $p$.

\begin{lem}[Duality relation for $\rho$ and $p$] \label{existenceofpwow}
Under the hypotheses of Lemma \ref{lem:f_*_integrable}, suppose $\rho_\epsilon$ is the solution of (\ref{epsiloneqn}), and let $\mu_\epsilon = \varphi_\epsilon * \rho_\epsilon$ and $q_\ep = f_\ep'(\mu_\ep)$.
Then, if $(\rho, \zeta)$ is a weak $L^1_\loc([0,+\infty) ; L^1(\RR^d))\times L^1_\loc([0,+\infty)\times \RR^d)$ limit point of $(\mue, f^*_{\ep}(\qe))$,  there exists a Lebesgue measurable function $p:[0,+\infty)\times\RR^d\to [-\infty, \infty)$ so that 
\begin{enumerate}[(i)]
\item $p\in L^1_{\loc}([0,+\infty);L^1(\rho))$, \label{pL1rho}
\item $ f^*(p)= \zeta$ almost everywhere, \label{secondpartyay}
\item $\rho p=f(\rho)+f^*(p)$ almost everywhere, \label{thirdpartyay}
\end{enumerate}
where we define $f^*(-\infty)=\lim_{b\to-\infty} f^*(b) = 0$ and take the convention $a \cdot 0=0$ for all $a\in [-\infty, \infty]$. 
\end{lem}

\begin{proof} First we will show that $\zeta\in f^*(\dom(f^*))\cup\{0\}$. 
By the duality relation in  Proposition \ref{identifylimitpoint}, 
\begin{equation}
    \label{eq:zeta in e'rho}
    \zeta(t,x)\in \partial e(\rho(t,x)) \text{ almost everywhere.}
\end{equation}
Thus, we can use   Lemma \ref{eepsproperties}(\ref{subdiffe}), to guarantee that $\zeta\in \{f^*(b): b\in \partial f(\rho)\} \subseteq f^*(\dom(f^*))$ for almost every $(t,x)$ with $\rho(t,x)>0$. As both $\rho$ and $\zeta$ are nonnegative,   Lemma \ref{eepsproperties}(\ref{item:e_at_zero}) implies that
\begin{align}\label{rhozeta} \{ (t,x) : \rho(t,x) = 0\} \subseteq \{ (t,x) : \zeta(t,x) = 0 \}   \text{ almost everywhere}. 
\end{align}
 Thus, we have $\zeta\in f^*(\dom(f^*))\cup\{0\}$ for almost every $(t,x)$.  
 
 Now we   turn to constructing $p$.
We begin with the following small observation: by Lemma  \ref{eepsproperties}(\ref{eepsconvex}), $e^*(0) = - \inf_a e(a) = - e(0) = 0$. Furthermore, by Proposition \ref{identifylimitpoint}, $\rho\zeta=e(\rho)+e^*(\zeta)$ almost everywhere. Thus, 
\begin{align} \label{zetatoerho}
\{ (t,x) : \zeta(t,x) = 0 \} \subseteq \{(t,x) : e(\rho(t,x)) = 0 \} \text{ almost everywhere.}
\end{align}

We begin by disposing of the trivial case where $f^*(\dom(f^*)))=\{0\}$.  In this case, $\zeta=0$ almost everywhere, so equation (\ref{zetatoerho}) ensures $e(\rho)=0$ almost everywhere.  Lemma \ref{eepsproperties}(\ref{f_e_linear})   implies that there exists some $b\in\RR$ such that
$f(a)=ba$ for all $a\in e^{-1}(0)$, i.e. $f(\rho)=b\rho$ almost everywhere. 
Let $a^*=\sup e^{-1}(0)$, so $e^{-1}(0)=[0,a^*]$. Noting that convex functions are differentiable a.e. on the interior of their domains and $f(0)=0$, we can compute
\[
f^*(b)=\sup_{a \in \dom(f)} ab-f(a) = \sup_{a \in \dom(f)}  \int_0^a b-f'(\theta) d \theta = \sup_{a\in \dom(f)} \int_{\min(a,a^*)}^a b-f'(\theta) d\theta \leq 0,
\]
where we use that $f'(\theta) = b$ for $\theta \in e^{-1}(0)=[0,a^*]$ and $f'$ is an increasing function almost everywhere. On the other hand, since $f^* \geq 0$, we must have $f^*(b) = 0$.
Thus, if we choose $p=b$, then $\rho p=f(\rho)=f(\rho)+f^*(p)$ and $f^*(p)=0=\zeta$ almost everywhere.

Now for the rest of the argument, we may assume $f^*(\dom(f^*))\cap (0,\infty)\neq \emptyset$.
Fix  $\gamma>0$ small enough that  $f^*(\dom(f^*))\cap [\gamma,\infty)\neq \emptyset$. Note that if $f^*(b_0)\geq \gamma$, then, since Lemma \ref{convexconjincreasing} ensures $\lim_{b \to -\infty} f^*(b) = 0$, $b_0$ is not a minimizer of $f^*$, so   $0\notin \partial f^*(b_0)$ and $f^*$ is strictly increasing on $\{x \in \dom(f^*):f^*(x) \geq \gamma\}$.  Thus, there exists a unique Lipschitz inverse $\tilde{g}_{\gamma}$ of $f^*$ on $E_{\gamma}=f^*(\dom(f^*))\cap [\gamma,\infty)$.
We then define an extension $g_{\gamma}:f^*(\dom(f^*))\cup\{0\}\to \RR$ by setting 
\begin{align} \label{ggammaeqn}
g_{\gamma}(a)=\begin{cases}
    \tilde{g}_{\gamma}(a) & \textup{if}\; a \geq \gamma,\\
    \tilde{g}_{\gamma}(\gamma) & \textup{if}\; a< \gamma. 
\end{cases}
\end{align}
Note that $g_{\gamma}(b)$ is nondecreasing with respect to both $\gamma$ and  $b\in f^*(\dom(f^*))\cup\{0\}$. 

We now use $g_{\gamma}$ to approximate $(f^*)^{-1}(\zeta)$.
For a.e. $(t,x)$ such that $\zeta\in f^*(\dom(f^*))\cup\{0\}$, define $\bar{p}_{\gamma}(t,x):=g_{\gamma}(\zeta(t,x))$. Then $\bar{p}_{\gamma}$ is  monotone nondecreasing in $\gamma$ and   
\begin{align} \label{fstarpgamma}f^*(\bar{p}_{\gamma})=\max(\zeta, \gamma) \text{ a.e.,}
\end{align} 
so, in particular, $\bar{p}_\gamma \subseteq \dom(f^*)$ almost everywhere.
The definition of $\bar{p}_\gamma$, and the fact that $g_{\gamma}$ is Lipschitz, imply that, for almost every $(t,x)$, we have $0\leq \bar{p}_{\gamma}(t,x)-g_{\gamma}(0)=g_\gamma(\zeta(t,x))-g_\gamma(0)\leq C_\gamma \zeta$ for $C_\gamma>0$. Thus, $\bar{p}_{\gamma}\in L^1_{\loc}([0,+\infty)\times\RR^d).$  Since $\bar{p}_{\gamma}$ is monotone in $\gamma$, there exists a measurable function $p:[0,+\infty)\times \Rd \to [-\infty,+\infty)$ such that 
 \[ p=\lim_{\gamma\to 0} \bar{p}_{\gamma} \text{ pointwise almost everywhere.} \]  Thus, since $\bar{p}_\gamma \subseteq \dom(f^*)$, $\lim_{b \to -\infty} f^*(b) = 0$, and $f^*$ is continuous on the interior of its domain,
\begin{align} \label{fstarpeqn} f^*(p)=\lim_{\gamma\to 0} f^*(p_{\gamma})=\lim_{\gamma\to 0} \max(\zeta, \gamma)=\zeta \text{ a.e.} 
\end{align}
This shows (\ref{secondpartyay}).

Define
 $Q_{\gamma}=\{(t,x)\in [0,+\infty)\times\RR^d: \zeta(t,x)\geq \gamma\}$. 
 Then, (\ref{rhozeta}) implies that $\rho>0$ almost everywhere on $Q_{\gamma}$. Thus, (\ref{eq:zeta in e'rho}) and Lemma \ref{eepsproperties}(\ref{subdiffe}) ensure $\zeta\in f^*(\partial f(\rho))$ almost everywhere on $Q_{\gamma}$
 and, since $\zeta(Q_{\gamma}) \subseteq E_\gamma$ for a.e. $(t,x)$,
 \[
\bar{p}_{\gamma}(t,x)=g_{\gamma}\left(\zeta(t,x)\right) = (f^*)^{-1}\left( \zeta(t,x) \right) \in  \partial f(\rho(t,x)) ,
\]
for almost every $(t,x)\in Q_{\gamma}$.   Therefore, $\bar{p}_{\gamma}\rho=f(\rho)+f^*(\bar{p}_{\gamma})$ holds almost everywhere on $Q_{\gamma}$. Furthermore, note that, almost everywhere on $Q_\gamma$, for any $\gamma'\leq \gamma$, we have $\bar{p}_{\gamma'} = g_{\gamma'}(\zeta) =  g_{\gamma}(\zeta) = \bar{p}_{\gamma} $, so  $p = \lim_{\gamma' \to 0} \bar{p}_{\gamma'} =\bar{p}_{\gamma}$. We can now conclude that $p\rho=f(\rho)+f^*(p)$ a.e. on $\bigcup_{\gamma>0} Q_{\gamma}=\{(t,x): \zeta(t,x)>0\}$. 

It remains to show that $\rho p=f(\rho)+f^*(p)$ still holds almost everywhere on $\{(t,x): \zeta(t,x)=0\}$. By equation (\ref{fstarpeqn}), $f^*(p)=\zeta = 0$ a.e. on this set.  Thus, it suffices to show that $\rho p=f(\rho)$ almost everywhere on $A:=\{(t,x): \zeta(t,x)=0, \rho(t,x)>0\}$. By equation (\ref{zetatoerho}),  $e(\rho)=0$ a.e. on $A$. Without loss of generality, we may assume $A$ has positive measure. Thus, by Lemma \ref{eepsproperties}(\ref{f_e_linear}), there exists $b \in \R$ so that $f(\rho)=b\rho$ a.e. on $A$ and, as argued above, $f^*(b) =0$. Since Lemma \ref{convexconjincreasing} ensures $f^*$ is nondecreasing, the fact that $f^*(g_\gamma(\gamma)) = \gamma>0$ ensures $g_\gamma(\gamma)>b$. If $b\neq \lim_{\gamma\to 0} g_{\gamma}(\gamma)$, then there exists some $b'>b$ such that $f^*(b')=0$. However, this gives us
\[
f^*(b')=0=\sup_{a\geq 0} ab'-f(a)\geq \sup_{a\in e^{-1}(0)} a(b'-b)
\]
implying that $e^{-1}(0)=\{0\}$, which would force $A$ to have measure zero.  
Thus, 
\[ 
p (t,x)= \lim_{\gamma \to 0} \bar{p}_\gamma (t,x)= \lim_{\gamma \to 0} g_\gamma (\zeta(t,x)) = \lim_{\gamma \to 0} g_\gamma(0) = \lim_{\gamma \to 0}  {g}_\gamma(\gamma) = b \text{ a.e. on $A$,} 
\]
which implies $\rho p=f(\rho)$ almost everywhere on $A$.
At last, this allows us to conclude that $\rho p  = f(\rho)+f^*(p)$ almost everywhere, showing (\ref{thirdpartyay}).

Finally, it remains to prove (\ref{pL1rho}).  Recall, from Lemma \ref{rhoepscpt} that $f(\rho) \in  L^1_{\loc}([0,+\infty);L^1(\RR^d))$. Likewise, by Lemma \ref{lem:f_*_integrable}, $\zeta_\ep  = f_\ep^*(q_\ep) \in L^1_\loc([0,+\infty); L^1(\Rd))$, so lower semicontinuity of the $L^1$ norms with respect to weak convergence ensure $\zeta\in L^1_{\loc}([0,+\infty);L^1(\RR^d))$. Using the relation $\rho p=f(\rho)+f^*(p)=f(\rho)+\zeta$ and the fact that $ \zeta\in  L^1_{\loc}([0,+\infty);L^1(\RR^d))$ we obtain that $p\in L^1_{\loc}([0,+\infty);L^1(\rho))$.

\end{proof}

\begin{proof}[Proof of Theorem \ref{mainthm}]

By Lemma \ref{rhoepscpt} and the dominated convergence theorem, there exists $\rho \in C_{\loc}^{1/2}([0,+\infty);\P_1(\Rd))\cap L^{\infty}_{\loc}([0,+\infty);\P_2(\RR^d)\cap L^1(\R^d))$  so that, up to a subsequence, $\rho_\epsilon(t,x) \to \rho(t,x)$ and $\mu_\epsilon(t,x) \to \rho(t,x)$ weakly in $L^1_\loc([0,+\infty);L^1( \Rd))$.  
By Lemma \ref{lem:f_*_integrable}, $\{f^*_\epsilon(q_\epsilon) \}_{\epsilon >0}$ is also weakly precompact in $L^1_\loc([0,+\infty) \times \Rd)$. Let $(\rho_\epsilon, f^*_\epsilon(q_\epsilon))$ denote a weakly convergent subsequence with limit point $(\rho,\zeta)$. By Lemma \ref{existenceofpwow},  there exists  $p\in L^1_{\loc}([0,+\infty);L^1(\rho))$  so that $f^*(p) = \zeta$ and $p \in \partial f(\rho)$ almost everywhere on $[0,+\infty) \times \Rd$. 

Next, we will show that $\zeta \in L^1_\loc([0,+\infty); W^{1,1}(\Rd))$. Since we already have $\zeta\in L^1_{\loc}([0,+\infty);L^1(\RR^d))$, it suffices to consider the derivative of $\zeta$. Note that, for any $\psi \in C^\infty_c([0,+\infty) \times \Rd)$, Proposition \ref{lem:mollifier exchange} and our hypothesis on the rate of decay of $\delta(\epsilon)$, equation (\ref{deltaepshyp}) ensure
\begin{align} \label{puppies}
\lim_{\epsilon \to 0} \int_0^T \int_\Rd  \psi\rho_\epsilon \nabla p_\epsilon   &=\lim_{\epsilon \to 0} \int_0^T \int_\Rd \psi \mu_\epsilon \nabla q_\epsilon  =\lim_{\epsilon \to 0} \int_0^T \int_\Rd \psi \nabla f_\epsilon^*(q_\epsilon) =  -\int_0^T \int_\Rd \zeta \nabla \psi,
\end{align}
where the second equation follows from Lemma \ref{muepqeplem}, which  ensures $ f^*_{\ep}(\qe) \in L^1_\loc([0,+\infty);W^{1,\infty}(\Rd))$ and $\nabla f^*_{\ep}(\qe)=\mue\nabla \qe$ and and the third equation follows from the weak $L^1_\loc([0,+\infty) \times \Rd)$ convergence of $ f_\epsilon^*(q_\epsilon)$ to $\zeta$.
Using the Cauchy-Schwartz inequality and uniform control on the kinetic energy from  Lemma \ref{energydissipationlem}, the left hand side is bounded from above by
\begin{align*}
\sup_{\ep >0}\int_0^T \|\psi\|_{L^{\infty}(\{t\}\times\Rd)} \|\rho_{\epsilon}|\nabla p_\epsilon|^2\|_{L^1(\{t\}\times\R^d)}^{1/2}\norm{\rho_{\ep}}_{L^1(\{t\}\times \RR^d)}^{1/2} dt \leq C_{d,T} \|\psi\|_{L^2([0,T];L^{\infty}( \Rd))} .
\end{align*}
Combining this with (\ref{puppies}), we see that $\zeta \in L^2([0,T]; BV(\Rd))$. 

Furthermore, note that (\ref{puppies}) also shows that $\rho_\epsilon\nabla p_\epsilon  \to \nabla \zeta$ in duality with $C^\infty_c([0,+\infty) \times \Rd)$.  By the equivalence of convergence in distribution and narrow convergence, lower semicontinuity of the Benamou-Brenier functional (see, e.g., \cite[Proposition 5.18]{santambrogio2015optimal}) ensures $\nabla \zeta  \ll \rho $ and
\begin{align*}
  +\infty >   \liminf_{\ep \to 0} \left( \int_0^T \int_\Rd |  \nabla  p_\epsilon|^2\rho_\epsilon \right)^{1/2}= \liminf_{\ep \to 0} \left( \int_0^T \int_\Rd \frac{ |\rho_\epsilon \nabla  p_\epsilon|^2}{\rho_\epsilon} \right)^{1/2} \geq \left( \int_0^T \int_\Rd \frac{ |\nabla \zeta|^2}{\rho} \right)^{1/2}.
\end{align*}
 Let $\xi \in L^1(\rho)$ be the Radon Nikodym derivative of $\nabla \zeta$ with respect to $\rho$, $\nabla \zeta = \xi \rho$.
Then,
\begin{align} \label{kitties} \int_0^T \int_\Rd |\nabla \zeta|  = \int_0^T \int_\Rd   |\xi| d \rho \leq \left( \int_0^T \int_\Rd   |\xi|^2 d \rho\right)^{1/2}  = \left( \int_0^T \int_\Rd \frac{ |\nabla \zeta|^2}{\rho} \right)^{1/2}  < +\infty .
\end{align}
Thus $\nabla \zeta \in L^1_\loc([0,+\infty) ; L^1( \Rd))$. This shows $\zeta \in L^1_\loc([0,+\infty); W^{1,1}(\Rd))$.

Next, we show $\lim_{\epsilon \to 0} \int_0^T \int_\Rd \left|  (v_\epsilon -v) \rho_\epsilon \right|=0$.
Fix $\delta>0$, and let 
\[E_{\delta, \epsilon}(t):=\{x\in \RR^d: |v(t,x)-v_{\ep}(t,x)|>\delta \} . \] 
We then have 
\begin{align*}
 \int_0^T \int_\Rd \left|  (v_\epsilon -v) \rho_\epsilon \right| &\leq  \int_0^T \int_{E_{\delta,\epsilon}(t)^c} \left|  (v_\epsilon -v) \rho_\epsilon \right| +  \int_0^T \int_{E_{\delta,\epsilon}(t)} \left|  (v_\epsilon -v) \rho_\epsilon \right| \\
 &\leq T\delta+\int_0^T \left\|\frac{|v(t)|+|v_{\ep}(t)|}{1+|\cdot|}\right\|_{L^\infty(\Rd)}\int_{ E_{\delta,\epsilon}(t)} \rho_{\epsilon}(t)(1+|x|),
\end{align*}
Since our velocities are well-prepared, see Definition \ref{wellprepareddef},  $v_{\epsilon} \to v$ strongly in $L^1_{\loc}([0,+\infty)\times\RR^d)$. In particular, for almost every $t\geq 0$, $v_\epsilon(t, \cdot) \to v(t, \cdot)$ in $L^1_\loc(\Rd)$. It then follows that, for any $R, \delta>0$, we have $ \lim_{\epsilon\to 0} |E_{\delta,\epsilon}(t)\cap B_R(0)|=0$ for a.e. $t \geq 0$.
Hence, for a.e. $t\geq0$,
\begin{align*}
\limsup_{\epsilon\to 0}\int_{E_{\delta, \epsilon}(t)} \rho_{\ep}(t)(1+|x|) &= \limsup_{\epsilon\to 0}\int_{E_{\delta, \epsilon}(t)\cap B_R(0)^c} \rho_{\ep}(t)(1+|x|) + \int_{E_{\delta, \epsilon}(t)\cap B_R(0)} \rho_{\ep}(t)(1+|x|)  \\
&\leq \sup_{\epsilon>0} \left( \int_{B_R(0)^c} \rho_\ep(t)+ \frac{1}{R}M_2(\rho_{\ep}(t))\right)+\limsup_{\epsilon\to 0}(1+R)\int_{E_{\delta, \epsilon}(t)\cap B_R(0)} \rho_{\ep}(t) \\
&=\sup_{\epsilon>0}  \int_{B_R(0)^c} \rho_\ep(t) + \frac{1}{R}M_2(\rho_{\ep}(t)),
\end{align*}
where the last equality follows from the uniform  integrability of  $\rho_{\ep}(t)$; see Lemma \ref{rhoepscpt}(\ref{rhoepunifint}). Next, we send $R\to \infty$, using the weak $L^1(\Rd)$ compactness, hence tightness, of  $\rho_{\ep}(t)$; see Lemma \ref{rhoepscpt}(\ref{aalimitrho}). It follows that
\[
\limsup_{\epsilon\to 0}\int_{E_{\delta, \epsilon}(t)} \rho_{\ep}(t)(1+|x|)=0 .
\]
for almost every $t\geq 0$ and any $\delta>0$.  Now, we can use the dominated convergence theorem, due to the uniform in time control on $\int_\Rd \rho_\ep(t) \equiv 1$ and $M_2(\rho_\ep(t))$ (see  Lemma \ref{energydissipationlem}) to conclude that
\[
\limsup_{\epsilon\to 0} \int_0^T \int_\Rd \left|  (v_\epsilon -v) \rho_\epsilon \right|\leq T\delta+\limsup_{\epsilon \to 0}\int_0^T \left\| \frac{|v(t)|+|v_{\ep}(t)|}{1+|\cdot|} \right\|_{\infty}\int_{ E_{\delta,\epsilon}(t)} \rho_{\epsilon}(t)(1+|x|)=T\delta .
\]
Since $\delta >0$ was arbitrary, we obtain   $\lim_{\epsilon \to 0} \int_0^T \int_\Rd \left|  (v_\epsilon -v) \rho_\epsilon \right|=0$.

Now, using the fact that $\rho_\epsilon$ is a weak solution of (\ref{epsiloneqn}), equation (\ref{puppies}), and the convergence of the velocities, we obtain that, for all $g \in C^\infty_c([0,T] \times \Rd)$,
\begin{align}
0   &= \lim_{\epsilon \to 0} \int_0^T \int_\Rd \rho_\epsilon\partial_t g  - \rho_\epsilon  (\nabla p_\epsilon +v_\epsilon) \cdot \nabla g \nonumber \\
&=    \int_0^T \int_\Rd \partial_t g \rho + \zeta\Delta g    -v \rho\cdot \nabla g  \nonumber \\
&=    \int_0^T \int_\Rd \partial_t g \rho  - \nabla \zeta \cdot \nabla g    -v \rho\cdot \nabla g 
\label{teamcoconut}
\end{align}
where, in the last equation, we use that $\zeta \in L^1_\loc([0,+\infty); W^{1,1}(\Rd))$. Finally, Lemma \ref{existenceofpwow}  ensures that there exists $p \in \partial f(\rho)$ so that $f^*(p) = \zeta$ almost everywhere. This shows that $\rho$ is a weak solution  of (\ref{PDE}).

Finally, using that $\nabla \zeta = \xi \rho$, we see $\rho \in C([0,T];\P_1(\Rd) )\cap L^{\infty}_{\loc}([0,+\infty);\P_2(\RR^d)\cap L^1(\R^d))$ solves
\[ \partial_t \rho - \nabla \cdot ((\xi  + v) \rho) = 0 \quad \text{ in the duality with } C^\infty_c((0,+\infty) \times \Rd) .\]
Inequality (\ref{kitties}) and our hypothesis on $v$ ensure that, for all $T >0$,
\[ \int_0^T \int_\Rd |\xi +v|^2 \rho   < +\infty . \]
Thus, \cite[Theorem 8.3.1]{ambrosiogiglisavare} ensures $\rho \in AC^2([0,T]; \P_2(\Rd))$.

\end{proof}

\appendix

\section{Appendix}

\subsection{Properties of a continuity equation}
{We begin with the following lemma on well-posedness of a   continuity equation, in which the velocity field $w:\PP_2(\RR^d)\to {C_b([0,+\infty); W^{1,\infty}(\Rd))}$  has uniformly bounded range  and is uniformly Lipschitz with respect to $\PP_2(\Rd)$.}

\begin{lem}[Well-posedness of a  continuity equation]
\label{lem:wp continuity eqn}
   Consider $w:\PP_2(\RR^d)\to {C_b([0,+\infty); W^{1,\infty}(\Rd))}$ for which  there exists $C_E, C_w>0$ so that
   \begin{align} \label{continuityinspace}
 {  \norm{w(\mu)}_{L^\infty([0,+\infty);W^{1,\infty}(\Rd))} }&\leq C_E , \forall \mu \in \PP_2(\Rd) \\
{   \norm{w(\mu_1)-w(\mu_0)}_{L^\infty([0,+\infty) \times\RR^d)}} &\leq C_w W_2(\mu_0,\mu_1)  , \ \forall \mu_1, \mu_0\in \PP_2(\RR^d) .\label{continuityofw}
   \end{align}
Then, for any initial data $\rho^0 \in \PP_2(\RR^d)$, there exists a unique $\rho \in AC^2_\loc([0,+\infty); \P_2(\Rd))$ that solves
\begin{align} \label{wpde}
\begin{cases} \partial_t \rho-\nabla \cdot (\rho w(\rho))=0 , \\ \left. \rho \right|_{t =0} = \rho^0 ,\end{cases}
\end{align}
in the sense of distributions and a unique flow map $X\in W^{1,\infty}_\loc([0,+\infty);W^{1,\infty}_{\loc}(\RR^d))$ satisfying
\begin{equation}
    {\partial_t X( t,x)=-w(\rho(t)) (t,X(t,x)), \quad X(0,x)=x,}
\end{equation}
such that {$\rho(t)=X(t)_{\#}\rho^0$}.  Furthermore, for any two choices of initial data $\rho^{0,1}$, $\rho^{0,2} \in \P_2(\Rd)$, we have
\begin{align} \label{stabilityweqn}
W_2(\rho^{1}(t), \rho^2(t))   \leq W_2(\rho^{0,1}, \rho^{0,2}) \left( 1+ (C_E +C_w)t e^{(C_E+ C_w)t} \right) ,\quad  \ \forall t \geq 0.
\end{align}
\end{lem}

\begin{proof}
First, we show that, for any $\mu \in AC^2_\loc([0,+\infty); \P_2(\Rd))$, the Lagrangian flow equations
\begin{align} \label{genLagrangianflow}
{\partial_t X(t,x) = - w(\mu(t))(t,X(t,x))} , \quad X(0, x) = x, 
\end{align}
are well-posed. Since $w(\mu(t))(t,x)$ is  bounded and Lipschitz in space, uniformly in time, by Cauchy-Lipschitz theory, it suffices to verify that $w(\mu(t))(t,x)$ is continuous in time. By 
our continuity hypothesis on $w$, inequality (\ref{continuityofw}), we have
 \begin{align*}
 &\| w(\mu(t))(t, \cdot) - w(\mu(s))(s, \cdot)\|_{L^\infty(\Rd)} \\
  &\quad \leq  \| w(\mu(t))(t, \cdot) - w(\mu(t))(s, \cdot)+ w(\mu(t))(s, \cdot) -w(\mu(s))(s, \cdot)\|_{L^\infty(\Rd)}\\
 &\quad \leq    \|w(\mu(t))(t,\cdot)\|_{W^{1,\infty}(\Rd)}|t-s| + C_w W_2(\mu(t), \mu(s)) , \end{align*}
which gives the result since $w: \P_2(\Rd) \to C_b([0, +\infty);W^{1,\infty}(\Rd))$ and  $\mu \in AC^2_\loc([0,+\infty); \P_2(\Rd))$.  Thus the Lagrangian flow (\ref{genLagrangianflow}) is well-posed.

Now, we establish existence of solutions to (\ref{wpde}).  Given initial data $\rho^0\in \PP_2(\RR^d)$ we define a sequence of approximate solutions $\rho^m$  and flow maps $X^m$ by setting $X^0(t,x)=x$, $\rho^m(t,\cdot)=X^m(t,\cdot)_{\#}\rho^0$ and iteratively solving the Lagrangian flow equations
\begin{align} \label{flowmapdef}
 \partial_t X^{m+1}(t,x)=- w(\rho^{m}(t))( t, X^{m+1}(t,x)), \quad X^{m+1}(0,x)=x .
\end{align}
For $m = 1$, (\ref{flowmapdef}) is clearly well-posed since $\rho^0$ is constant, hence absolutely continuous in time. Furthermore, if the equation is well-posed for $m$, then  $\rho^{m+1}(t,\cdot) = X^{m+1}(t,\cdot) \# \rho^0$ satisfies
\begin{align*}
 |{\rho^{m+1}}'|(t) &= \lim_{s \to t} \frac{W_2(\rho^{m+1}(s), \rho^{m+1}(t))}{|s-t|} \leq \lim_{s \to t} \frac{   \|X^{m+1}(t,\cdot) - X^{m+1}(s,\cdot) \|_{L^2(\rho^0)}}{|s-t|}  \\
 & =\lim_{s \to t} \frac{   \left\| \int_s^t \partial_r X^{m+1}(r,\cdot) dr \right\|_{L^2(\rho^0)}}{|s-t|} \leq \lim_{s \to t} \frac{ \int_s^t   \left\| w(\rho^m(r))( r,X^{m+1}(r,\cdot)) \right\|_{L^2(\rho^0)}dr}{|s-t|}
\leq C_E ,
\end{align*}
where the first inequality follows from \cite[equation (7.1.6)]{ambrosiogiglisavare}. Thus, $\rho^{m+1} \in AC^2_\loc([0,+\infty); \P_2(\Rd))$, and hence (\ref{flowmapdef}) is well-posed for $m+1$. 

Now we wish to show that the flow maps defined by (\ref{flowmapdef}) form a Cauchy sequence.  We note,
\begin{align*}
&\frac{d}{dt} \|X^{m+1} - X^m\|_{L^2(\rho^0)} \\
&\quad \leq \norm{w(\rho^m)(t, X^{m+1}(t,\cdot)) - w(\rho^m)(t, X^{m}(t,\cdot)) + w(\rho^m)(t, X^{m}(t,\cdot))-w(\rho^{m-1})(t, X^m(t,\cdot))}_{L^{2}(\rho^0)} \\
&\quad \leq \norm{w(\rho^m)( t, X^{m+1}(t,\cdot))-w(\rho^{m})(t, X^m(t,\cdot))}_{L^{2}(\rho^0)} +  \norm{ w(\rho^m)-w(\rho^{m-1})}_{L^{\infty}([0,+\infty)\times \Rd)}\\
&\quad \leq  C_E \norm{  X^{m+1}-   X^m}_{L^{2}(\rho^0)} + C_w W_2(\rho^m,\rho^{m-1})\\
&\quad \leq {(C_E+C_w) \norm{  X^{m+1}-   X^m}_{L^{2}(\rho^0)} .}
\end{align*}
{Integrating in time, we obtain,}
\begin{align*}
\|X^{m+1} - X^m\|_{L^2(\rho^0)} \leq { (C_E+C_w)}  \int_0^t \|X^m-X^{m-1} \|_{L^2(\rho^0)}.
\end{align*}
Hence,
\begin{align*}
\|X^{m+1} - X^m\|_{L^\infty([0,T];L^2(\rho^0))} &\leq {(C_E+C_w)T}   \|X^m-X^{m-1} \|_{L^\infty([0,T];L^2(\rho^0))} \\&\leq \dots \leq {((C_E+C_W) T )^m} \|X^1 - X^0\|_{L^\infty([0,T];L^2(\rho^0))} .
\end{align*}
In particular, for $T>0$ sufficiently small, the flow maps form a Cauchy sequence in $L^\infty([0,T];L^2(\rho^0))$. 

Let $X(t,x)$ denote the limit and let $\rho(t) = X(t) \# \rho^0$. Then,
\begin{align*}
&\|  w(\rho^m(t))( t, X^{m+1}(t,\cdot)) -  w(\rho(t))(t, X(t,\cdot)) \|_{L^2(\rho^0)} \\
&\quad \leq \|  w(\rho^m(t)) (t, X^{m+1}(t,\cdot))- w(\rho^m(t)) (t, X(t,\cdot)) + w(\rho^m(t)) ( t, X(t,\cdot)) -  w(\rho(t)) ( t,X(t,\cdot)) \|_{L^2(\rho^0)} \\
&\quad   \leq C_E \|   X^{m+1}(t,\cdot)- X(t,\cdot) \|_{L^2(\rho^0)} + C_w W_2(\rho^m(t),\rho(t)) \\
&\quad \leq (C_E+ C_w)  \|   X^{m+1}(t,x)- X(t,x) \|_{L^2(\rho^0)} .
\end{align*}
Thus, the left hand side converges to zero, uniformly in $t\in [0,T]$ and we may conclude
\begin{align*} X(t,x) &= \lim_{m \to +\infty} X^{m+1}(t,x) = \lim_{m \to +\infty} - \int_0^t w(\rho^m(r)) ( r,X^{m+1}(r,x)) dr + x  \\&=  - \int_0^t w(\rho(r))(r,X(r,x)) dr + x , \text{ for } \rho^0\text{-a.e. } x \in \Rd \end{align*}
  The argument of \cite[Proposition 8.1.8]{ambrosiogiglisavare} yields that $\rho$ is a solution to the continuity equation on $[0,T]\times\RR^d$ with velocity $w(\rho(t))(t,x)$. Restarting the process at time $T$ and composing flow maps,
we can construct solutions on any time interval.  Thus, existence is complete.

To establish well-posedness, suppose that $\rho^1, \rho^2$ are solutions starting from initial data $\rho^{0,1}$ and $\rho^{0,2}$ respectively, with associated flow maps $X^1$ and $X^2$, $\partial_t X^i(t,x) = - w(\rho^i(t)) (t,X^i(t,x))$, so that $X^1(t) \# \rho^{0,1} = \rho^1(t)$ and $X^2(t) \# \rho^{0,2} = \rho^2(t)$.  We begin by supposing that $\rho^{0,1}$ is absolutely continuous with respect to Lebesgue measure, so there exists an optimal transport map $\bt \# \rho^{0,1} = \rho^{0,2}$.  Then,
\begin{align*}
&\|X^1(t, \cdot) - X^2(t, \cdot) \circ \bt \|_{L^2(\rho^{0,1})} \\
&= \left\| \int_0^t  \left( w(\rho^1(r))(r,X^1(r, \cdot)) - w(\rho^2(r))(r, X^2(r, \cdot)) \circ \bt \right) dr + \id - \bt \right\|_{L^2(\rho^{0,1})}
\\
&\leq \int_0^t  \left\|  w(\rho^1(r))(r,X^1(r, \cdot))  - w(\rho^1(r))(r, X^2(r, \cdot)) \circ \bt + \left(w(\rho^1(r))- w(\rho^2(r)) \right)(r, X^2(r, \cdot)) \circ  \bt \right\|_{L^2(\rho^{0,1})} dr \\
&\qquad + \|\id - \bt \|_{L^2(\rho^{0,1})} \\ 
&\leq  \int_0^t  \left( C_{E}\left\|   X^1(r, \cdot) - X^2(r, \cdot)\circ \bt \right\|_{L^2(\rho^{0,1})} + C_w W_2(\rho^1(r),\rho^2(r)) \right) dr    + W_2(\rho^{0,1}, \rho^{0,2}) \\
&\leq  (C_{E} + C_w)  \int_0^t  \left\|   X^1(r, \cdot) - X^2(r, \cdot)\circ \bt \right\|_{L^2(\rho^{0,1})}  dr +W_2(\rho^{0,1}, \rho^{0,2}) .
\end{align*}
By Gronwall's inequality, this implies
\begin{align*}
 \|X^1(t, \cdot) - X^2(t, \cdot) \circ \bt \|_{L^2(\rho^{0,1})}  \leq W_2(\rho^{0,1}, \rho^{0,2}) \left( 1+ (C_E +C_w)t e^{(C_E+ C_w)t} \right)  .
\end{align*}
Using the fact that the left hand side is an upper bound for the Wasserstein distance between $\rho^1(t)$ and $\rho^2(t)$ gives the stability inequality (\ref{stabilityweqn}).

Finally, we remove the assumption that $\rho^{0,1}$ is absolutely continuous with respect to Lebesgue. In particular, for any $\rho^{0,1}, \rho^{0,2} \in \P_2(\Rd)$ and $\rho^{0,3} \in \P_{2,ac}(\Rd)$, inequality  (\ref{stabilityweqn}) implies
\begin{align*}
W_2(\rho^{1}(t), \rho^2(t)) &\leq W_2(\rho^1(t), \rho^3(t)) + W_2(\rho^3(t), \rho^2(t)) \\
&\leq \left( W_2(\rho^{0,1}, \rho^{0,3}) + W_2(\rho^{0,3}, \rho^{0,2})  \right) \left( 1+ (C_E +C_w)t e^{(C_E+ C_w)t} \right) .
\end{align*}
Thus, using the density of $\P_{2,ac}(\Rd)$ in $\P_2(\Rd)$ with respect to $W_2$, we may choose $\rho^{0,3}$ arbitrarily close to $\rho^{0,2}$ to get the result.
\end{proof}

Next, for the reader's convenience, we recall some basic properties of the flow map $X_\epsilon$.
\begin{lem}[Estimates on the flow map] \label{lem:DXep bds} Fix $\ep>0$,  $T>0$,  and let $X_\ep$ be as in Lemma \ref{wellposedpdeepslem}.  
Then {$X_\ep(t, \cdot)$} is invertible for all $t \geq 0$, 
the equality \begin{equation}
    \label{eq:partial t DX}
    \partial_t DX_{\epsilon}(t,x)=- \left(D^2 p_{\epsilon}(t,X_{\epsilon}(t, x))+D v_{\epsilon}(t,X_{\epsilon}(t, x))\right)DX_{\epsilon}(t,x)
\end{equation}
holds for {a.e.} $t\in [0,T]$, and there exists $\lambda_\epsilon>0$ such that, for  all $t\in [0,T]$ and {a.e.} $x\in \Rd$,
\begin{equation}
\label{eq:DXep bds}
\frac{1}{\lambda_\epsilon} I\prec DX_{\epsilon}(t,x)\prec \lambda_\epsilon I.
\end{equation}
\end{lem}
\begin{proof}
Differentiating (\ref{eq:Xep}) in $t$ { and $x$} yields $\text{ for all }t \in [0,+\infty), \text{ and a.e.} \ x \in \Rd$,
\begin{equation}
    \partial_t DX_{\epsilon}(t,x)=- \left(D^2 p_{\epsilon}(t,X_{\epsilon}(t, x))+D v_{\epsilon}(t,X_{\epsilon}(t, x))\right)DX_{\epsilon}(t,x) , .
\end{equation}
Hence, for { a.e. $t$ and $x$},
\[
\frac{d}{dt} \left( \exp\left(\int_0^t D^2p_{\epsilon}(s,X_{\epsilon}(s,x))+Dv_{\epsilon}(s,X_{\epsilon}(s,x))\, ds\right)DX_{\epsilon}(t,x)\right)=0.
\]
{ Integrating in time and using $X_\ep(0,x)=x$, we obtain that, for all $t \in [0,+\infty)$,} and a.e. $x \in \Rd$.
\[
 \exp\left(\int_0^t D^2p_{\epsilon}(s,X_{\epsilon}(s,x))+Dv_{\epsilon}(s,X_{\epsilon}(s,x))\, ds\right)DX_{\epsilon}(t,x)=DX_{\epsilon}(0,x)=Id .
\]
Thus,  for all $t\in [0,+\infty)$ and almost every $x \in \Rd$,
\[
DX_{\epsilon}(t,x)=\exp\left(-\int_0^t \left(D^2p_{\epsilon}(s,X_{\epsilon}(s,x))+Dv_{\epsilon}(s,X_{\epsilon}(s,x)\right)\, ds\right).
\]
The result follows since, due to the definition of $p_\ep$, we have that  $\norm{D^2p_\ep}_{L^{\infty}([0,t];L^{\infty}(\RR^d))}$  and $\norm{Dv_\ep}_{L^\infty([0,t];L^\infty(\Rd))}$  are bounded for any $t\geq 0$.
\end{proof}

\subsection{Well-prepared and particle data}
We now prove that, under appropriate hypotheses, well-prepared sequences of initial data and velocities exist.

\begin{lem}[Well-prepared initial data and velocity field] \label{wellpreparedlemma}
Given an internal energy density $f$,  a velocity $v$, a mollifier $\varphi_\epsilon$, and initial data $\rho^0$ satisfying     Assumptions \ref{internalas}, \ref{velocityas}, \ref{mollifieras}, and \ref{as:id},  there exists a sequence of  initial data $\{\rho_\epsilon^0\}_{\epsilon \in (0,1)} \subseteq \P_2(\Rd)\cap C^{\infty}_c(\RR^d)$ and velocity fields $\{v_{\ep}\}_{\epsilon \in (0,1)} \subseteq   C^{\infty}_c([0,+\infty)\times\RR^d)$ that is \emph{well-prepared} in the sense of Definition \ref{wellprepareddef}.
\end{lem}

\begin{proof}
Fix $\psi \in C^\infty_c(\R^n)$ vanishing outside $B_1(0)$ with $\psi \geq 0$, $\int \psi = 1$, and $\psi(x) = \psi(-x)$, and for any $\alpha >0$, let $\psi_\alpha(x) = \frac{1}{\alpha^n} \psi \left( \frac{x}{\alpha} \right)$. For $R >0$, let $\eta_R \in C^\infty_c(\R^n)$ be a smooth cutoff function with $\eta_R(z) = 1$ for $|z|\leq R$, $\eta_R(z) = 0$ for $|z| \geq R+1$ and $\| \nabla \eta_R\|_\infty \leq 1$.

First, we show that a sequence of well-prepared initial data $\{\rho^0_\epsilon\}_{\epsilon \in (0,1)}$ exists. Note that Assumption  \ref{as:id} ensures that $\S(\rho^0) < +\infty$ and $M_2(\rho^0)<+\infty$, so $d\rho^0 = \rho^0(x) dx$ and $\rho^0 \log \rho^0 \in L^1(\Rd)$. Likewise Assumption  \ref{as:id} ensures $\F(\rho^0)<+\infty$ which, combined with Remark \ref{coercivityremark}, ensures $f(\rho^0) \in L^1(\Rd)$. In particular, we have $f(\rho^0(x))< +\infty$ for a.e. $x \in \Rd$. Let  the domain of the mollifier $\psi_\alpha$  be $\R^{d}$, let $\alpha = \alpha_\epsilon \in (0,1)$ be a sequence satisfying $\lim_{\epsilon \to 0} \alpha_\epsilon = 0$ and   define
\begin{align*}
\rho_\epsilon^0  = \psi_\alpha*\tilde{\rho}^0_\alpha , \quad \tilde{\rho}^0_\alpha(x) =   1_{B_{\alpha^{-1}}}(x/R) \rho^0(x/R) , \  R  = \|\rho^0\|_{L^1(B_{\alpha^{-1}})}^{-1/d} .
\end{align*}
By definition $\{\rho^0_\epsilon\}_{\epsilon \in (0,1)} \in \P_2(\Rd) \cap C^\infty_c(\Rd) $. Note that $R \xrightarrow{\epsilon \to 0} 1$.

Next, we prove the bounds in equation (\ref{eq:data_uniform_bounds}). We start with the second moment. First, note that   $M_2(\rho^0_\epsilon) \leq 2M_2(\psi_\alpha) + 2M_2(\tilde{\rho}^0_\alpha)$ and $M_2(\psi_\alpha) \leq 1$. Furthermore,
\begin{align} \label{tilderhoalpham2}
M_2(\tilde{\rho}^0_\alpha) = \int_{B_{R/\alpha}} |x|^2 \rho^0(x/R) dx =  R^{d} \int_{B_{\alpha^{-1}}} |Ry|^2 \rho^0(y) dy \xrightarrow{\epsilon \to 0} M_2(\rho^0) .
\end{align}
Thus, by the dominated convergence theorem, $M_2(\rho^0_\epsilon)$ bounded uniformly $\epsilon \in (0,1)$.

For the entropy, note that, by Jensen's inequality,
\begin{align*}
\S(\rho^0_\epsilon) \leq \int \psi_\alpha*(\tilde{\rho}^0_\alpha \log(\tilde{\rho}^0_\alpha)) = \int \tilde{\rho}^0_\alpha \log(\tilde{\rho}^0_\alpha) = \int_{B_{R/\alpha}} \rho^0(\cdot/R) \log(\rho^0(\cdot/R)) = R^d \int_{B_{\alpha^{-1}}} \rho^0 \log(\rho^0),
\end{align*}
which converges to $\S(\rho^0)$ by the dominated convergence theorem for the integrable function $\rho^0 \log(\rho^0)$, hence is bounded uniformly in $\epsilon \in (0,1)$.

By definition of $\F_\epsilon$ and Jensen's inequality for the convex function $f_\epsilon$,
\begin{align*}
\F_\epsilon(\rho^0_\epsilon) & = \int f_\epsilon (\varphi_\ep* \rho^0_\epsilon) \leq \int \varphi_\epsilon* (f_\epsilon(\rho^0_\epsilon)) = \int  f_\epsilon(\rho^0_\epsilon)  = \frac{\delta(\epsilon)}{2} \left\|\rho^0_\epsilon \right\|_{L^2(\Rd)}^2 + \int_{\{\rho^0_\epsilon >0\}} {}^{\delta(\epsilon)} f(\rho^0_\epsilon) - {}^{\delta(\epsilon)} f(0) \\
&\leq \frac{\delta(\epsilon)}{2} \|\varphi_\alpha\|_{L^2(\Rd)}^2 \left\| \tilde{\rho}^0_\alpha\right\|_{L^1(\Rd)}^2 + \int_{\{\rho^0_\epsilon>0\}} f(\rho^0_\epsilon) - |\{\rho^0_\epsilon >0\}| {}^{\delta(\epsilon)}f(0). 
\end{align*}
Since $\left\| \tilde{\rho}^0_\alpha\right\|_{L^1(\Rd)}^2 \equiv 1$, we may choose $\alpha = \alpha_\epsilon$ decaying to zero sufficiently slowly so that the first term is bounded in $\epsilon \in (0,1)$. Likewise, since ${}^{\delta(\epsilon)} f(0) \to f(0) = 0$, we may choose $\alpha = \alpha_\epsilon$ decaying to zero sufficiently slowly so that the third term is bounded. Finally, using Jensen's inequality for the convex function $f$, along with the fact  that  $f(0)=0$, we may bound the middle term by 
\begin{align*}
\int \psi_\alpha*(f(\tilde{\rho}^0_\alpha)) = \int f(\tilde{\rho}^0_\alpha) =  \int_{B_{R/\alpha}} f \left(  \rho^0(\cdot /R) \right)     =  R^d \int_{B_{\alpha^{-1}}} f(\rho^0 )  ,
\end{align*} 
which converges to $\F(\rho^0)$ by the dominated convergence theorem for the integrable function $f(\rho^0)$. This completes the proof of inequality (\ref{eq:data_uniform_bounds}) for the sequence of initial data $\{\rho^0_\epsilon\}_{\epsilon \in (0,1)}$.

Next, we show that $\rho^0_\ep \to \rho^0$ in $W_2$. By the triangle inequality and \cite[Lemma 5.2]{santambrogio2015optimal},
\begin{align*}
W_2(\rho^0_\ep, \rho^0) \leq W_2(\psi_\alpha*\tilde{\rho}^0_\alpha, \psi_\alpha*\rho^0) + W_2(\psi_\alpha*\rho^0, \rho^0) \leq  W_2( \tilde{\rho}^0_\alpha,  \rho^0) + W_2(\psi_\alpha*\rho^0, \rho^0) .
\end{align*}
The second term goes to zero as $\epsilon \to 0$. For the first term, note that for any $g \in C^\infty_c(\Rd)$, 
\begin{align*}
\int g (\tilde{\rho}^0_\alpha - \rho^0) =  \int \left( R^d g(Ry) 1_{B_{\alpha^{-1}}}(y) - g(y) \right) \rho^0(y) dy \xrightarrow{\epsilon \to 0} 0 ,
\end{align*}
by the dominated convergence theorem. Thus $\tilde{\rho}^\alpha_0$ converges narrowly to $\rho^0$. Combining this with the convergence of the second moments (\ref{tilderhoalpham2}), we obtain that $W_2(\tilde{\rho}^0_\alpha, \rho) \xrightarrow{\epsilon \to 0} 0$   \cite[Lemma 5.11]{santambrogio2015optimal}. This completes our proof that a sequence of well-prepared initial data exists.

Now, we show that a sequence of well-prepared velocities $\{v_{\epsilon}\}_{\epsilon \in (0,1)}$ exits.
We consider $v(t,x)$ to be a function on all of $\R \times \Rd$ by setting it equal to zero whenever $t<0$. Let  the domain of the mollifier $\psi_\alpha$ and cutoff function $\eta_R$ be $\R^{d+1}$ and let $\alpha =\alpha_\epsilon \in (0,1)$ satisfy $\lim_{\epsilon \to 0} \alpha_\epsilon = 0$. Define
\begin{align} \label{vepsilondef} v_\epsilon(t,x) = (1+|x|^2)  \iint \psi_{\alpha}( t-s,x-y) \eta_{\alpha^{-1}}(s,y)   \frac{ v(s,y)}{1+|y|^2} dy ds ,
\end{align}
By definition, $\{v_{\epsilon}\}_{\epsilon \in (0,1)} \subseteq   C^{\infty}_c(\R \times\RR^d)$.

Next, we prove the bounds in equation (\ref{eq:data_uniform_bounds}). For any $x \in \Rd$ and $t \geq 0$,
\begin{align*}
 \frac{|v_\epsilon(t,x)|^2}{1+|x|^2}   
 & \leq  \iint \psi_{\alpha}(t-s,x-y)  \left\|   \frac{|v(s,\cdot)|^2}{1+|\cdot|^2}  \right\|_{L^\infty(\Rd)} dy d s  =  \iint \psi_{\alpha}(t-s,y) \left\|   \frac{|v(s,\cdot)|^2}{1+|\cdot|^2}  \right\|_{L^\infty(\Rd)} dy d s .
\end{align*}
Thus, for all $\epsilon \in (0,1)$ and $-\infty< T_0 \leq T_1 <+\infty$, 
\begin{align} \label{firstvepsestimate}
\int_{T_0}^{T_1}   \frac{|v_\epsilon(t,x)|^2}{1+|x|^2}   dt  
&\leq \int_{T_0}^{T_1}  \int_{T_{0}-1}^{T_1+1} \int_\Rd \psi_{\alpha}(y,t-s) \left\|   \frac{|v(s,\cdot)|^2}{1+|\cdot|^2}  \right\|_{L^\infty(\Rd)} dy d s dt \\
& \leq \int_{T_0-1}^{T_1+1} \left\|   \frac{|v(s,\cdot)|^2}{1+|\cdot|^2}  \right\|_{L^\infty(\Rd)}  ds , \nonumber
\end{align}
which gives the bound on $v_\epsilon(t,x)/(1+|x|^2)$ in (\ref{eq:data_uniform_bounds}).

Likewise, for the bound on $(\nabla \cdot v_\epsilon(t,x))_+/(1+|x|^2)$, we note that
\begin{align*}
\partial_{x_i} v_{\epsilon,i}(t,x) &= 2x_i  \iint \psi_{\alpha}(t-s,x-y) \eta_{\alpha^{-1}}(s,y)  \frac{ v_i(s,y)}{1+|y|^2} dy ds  \\
 & \qquad + (1+|x|^2) \iint \psi_{\alpha}(t-s,x-y)   \partial_{y_i}  \eta_{\alpha^{-1}}(s,y)  \frac{ v_i(s,y)}{1+|y|^2} dyds  \\
 &  \qquad  + (1+|x|^2) \iint \psi_{\alpha}(t-s,x-y)    \eta_{\alpha^{-1}}(s,y)  \left( \frac{ \partial_{y_i}  v_i(s,y)}{1+|y|^2} - \frac{2y_i v_i(s,y)}{(1+|y|^2)^2}  \right) dy ds .
\end{align*}
Thus,
\begin{align*}
\nabla \cdot v_\epsilon(t,x) &\leq  2(1+|x|^2) +2(1+|x|^2) \iint \psi_{\alpha}(t-s,x-y)  \frac{ |v(s,y)|^2}{1+|y|^2} dy ds  \\
 &  \quad  + (1+|x|^2) \iint \psi_{\alpha}(t-s,x-y)    \eta_{\alpha^{-1}}(s,y)  \left( \frac{ \nabla \cdot  v(s,y)}{1+|y|^2} + \frac{2 |y| |v(s,y)|}{(1+|y|^2)^2}  \right) dy ds .
\end{align*}
Simplifying and applying Jensen's inequality  for the increasing, convex, subadditive function $s \mapsto (s)_+$,
\begin{align*}
(\nabla \cdot v_{\epsilon}(t,x))_+ &\leq  3(1+|x|^2) + 3(1+|x|^2) \iint \psi_{\alpha}(t-s,x-y)  \frac{ |v(s,y)|^2}{1+|y|^2} dy ds  \\
 &  \quad  + (1+|x|^2)  \iint \psi_{\alpha}(t-s,x-y)    \eta_{\alpha^{-1}}(s,y)   \frac{ ( \nabla \cdot  v(s,y))_+}{1+|y|^2} dy ds . 
\end{align*}
Therefore, for all $T>0$,
\begin{align*}
\int_0^T \left\| \frac{(\nabla \cdot v_{\epsilon}(t,\cdot))_+}{1+|\cdot|^2} \right\|_{L^\infty(\Rd)} dt  \leq   3+3 \int_0^{T+1} \left\|   \frac{|v(s,\cdot)|^2}{1+|\cdot|^2}  \right\|_{L^\infty(\Rd)}  ds +  \int_0^{T+1} \left\|   \frac{(\nabla \cdot v(s,\cdot))_+}{1+|\cdot|^2}  \right\|_{L^\infty(\Rd)}  ds .
\end{align*}
This completes the proof of (\ref{eq:data_uniform_bounds}).

It remains to show that $v_\epsilon \to v$ in $L^1_\loc([0,+\infty) \times \Rd)$.
First, we show that $v_\epsilon(t,x) \to v(t,x)$ for a.e. $(t,x)\in \R^{d+1}$. We begin by estimating,  
 \begin{align}
&|v_\epsilon(t,x) - v(t,x)| \label{vepvpointwise} \\
&\leq \iint \psi_\alpha(t-s,x-y) \left| \eta_{\alpha^{-1}}(s,y) \frac{1+|x|^2}{1+|y|^2} v(s,y) - v(t,x) \right| dy ds \nonumber \\
&\leq \iint \psi_\alpha(t-s,x-y) \left|\eta_{\alpha^{-1}}(s,y) \frac{1+|x|^2}{1+|y|^2} (v(s,y)-v(t,x)) + \left(\eta_{\alpha^{-1}}(s,y) \frac{1+|x|^2}{1+|y|^2}-1 \right)  v(t,x) \right| dy ds \nonumber
\end{align}
By the Lebesgue differentiation theorem, for almost every $(t,x) \in \R^{d+1}$,
\begin{align*}
& \iint \psi_\alpha(t-s,x-y)  \eta_{\alpha^{-1}}(s,y) \frac{1+|x|^2}{1+|y|^2} |v(s,y)-v(t,x)| dyds \\
&\quad \leq (1+|x|^2) \frac{\|\psi\|_\infty}{\alpha^{d+1}} \iint_{B_\alpha(t,x)} |v(s,y)-v(t,x)| dyds \xrightarrow{\epsilon \to 0} 0 .
\end{align*}
On the other hand, by the dominated convergence theorem, for almost every $(t, x)\in \R^{d+1}$,
\begin{align*}
&\iint   \psi_\alpha(t-s,x-y)    \left| \eta_{\alpha^{-1}}(s,y) \frac{1+|x|^2}{1+|y|^2}-1 \right|  |v(t,x)| dy ds \\
&\quad \leq |v(t,x)| \iint \psi(r,z)   \left|  \eta_{\alpha^{-1}}(t-\alpha r,x - \alpha z) \frac{1+|x|^2}{1+|x - \alpha z|^2}-1   \right| dz dr \xrightarrow{\epsilon \to 0} 0.
\end{align*}
Combining both of these estimates with inequality (\ref{vepvpointwise}), we conclude that  $v_\epsilon(t,x) \to v(t,x)$ for a.e. $(t,x) \in \R^{d+1}$.

To upgrade the convergence to $v_\epsilon \to v$ in $L^1_\loc([0,+\infty) \times \Rd)$, we apply the Kolmogorov-Riesz-Fr\'echet theorem \cite[Theorem 4.26]{Brezis}.  Fix $R>0$ and let
\begin{align*}
w_{\epsilon,R}(t,x) = \eta_{R-1}(t,x) v_\epsilon(t,x) , \ \forall (t,x) \in \R \times \Rd .
\end{align*}
First, note that, if $\omega_d$ is the volume of the $d$-dimensional unit ball, inequality (\ref{firstvepsestimate}) ensures
\begin{align} \label{secondvepsestimate}
\|w_{\epsilon,R} \|_{L^1(\R^{d+1})} &\leq \|v_\epsilon\|_{L^1(B_{R})} \leq  \int_{-R}^R \int_{|x|\leq R}  (1+|x|^2) \left|\frac{|v_\epsilon(t,x)|^2}{1+|x|^2} \right| dx dt \\
& \leq (1+R^2) \omega_d R^d \int_{-R}^R  \left\|\frac{|v_\epsilon(t,\cdot)|^2}{1+|\cdot|^2} \right\|_{L^\infty(\Rd)} dt\leq  (1+R^2) \omega_d R^d  \int_{R-1}^{R+1} \left\|   \frac{|v(s,\cdot)|^2}{1+|\cdot|^2}  \right\|_{L^\infty(\Rd)}  ds . \nonumber
\end{align}
Thus, $\{w_{\epsilon,R} \}_{\epsilon \in(0,1)} \subseteq L^1(\R^{d+1})$ is bounded.
It remains to show that, 
\begin{align} \label{KRF}
\lim_{h \to 0}  \sup_{(y,s) \in B_h(0)} \iint | w_{\epsilon,R}(t-s,x-y) -w_{\epsilon,R}(t,x) | dx dt = 0 , \ \text{ uniformly in $\epsilon \in (0,1)$.}
\end{align}
Then the Kolmogorov-Riesz-Fr\'echet theorem will ensure that $\{w_{\epsilon,R}\}_{\epsilon \in(0,1)}$ is  relatively compact in $L^1_\loc(\R^{d+1})$. Thus, for any $K \subseteq \subseteq [0,+\infty) \times \Rd$, choosing $R>0$ sufficiently large so that $w_{\epsilon, R} = v_\epsilon$ on $K$, we see that $\{v_\epsilon\}_{\epsilon \in (0,1}$ is relatively compact in $L^1(K)$. Thus, $\{ v_\epsilon \}_{\epsilon \in(0,1)}$ is relatively compact in $L^1_\loc([0,+\infty) \times \R^{d})$.  By the pointwise a.e. convergence of $v_\epsilon$ to $v$ and uniqueness of limits, this shows that $v_\epsilon \to v$ in $L^1_\loc([0,+\infty) \times \Rd)$.

The remainder of the proof is devoted to showing (\ref{KRF}). For   $h \in (0,1)$ and $(s,y) \in B_h(0)$,
\begin{align*}
&\iint | w_{\epsilon,R}(t-s,x-y) -w_{\epsilon,R}(t,x) | dx dt\\
 &\quad \leq \iint | \eta_{R-1}(t-s,x-y) -\eta_{R-1}(t,x)||v_\epsilon(t-s,x-y) |  + |\eta_{R-1}(t,x)||v_\epsilon(t-s,x-y)-v_\epsilon(t,x) | dx dt \\
  &\quad \leq h  \|\nabla \eta_{R-1} \|_{\infty} \|v_\epsilon\|_{L^1(B_{R+1})}  + \int_{B_{R}} |v_\epsilon(t-s,x-y)-v_\epsilon(t,x) | dx dt.
\end{align*}
By the uniform bound on $\|v_\epsilon\|_{L^1(B_{R+1})}$ shown in  inequality (\ref{secondvepsestimate}), for arbitrary $R>0$,  as $h \to 0$, the first term goes to zero uniformly in $\epsilon \in (0,1)$ and $(y,s) \in B_h(0)$. It remains  to show the same is true of the second term.

By the definition of $v_\epsilon$ in equation (\ref{vepsilondef}), we may rewrite the second term as
\begin{align*}
&\iint_{B_{R}} \left|  \iint \left( (1+|x-y|^2)  \psi_{\alpha}( t-s-r,x-y-z)  - (1+|x|^2)    \psi_{\alpha}( t-r,x-z) \right) \eta_{\alpha^{-1}}(r,z)   \frac{ v(r,z)}{1+|z|^2} dz dr \right| dx dt \\
&\leq \iint_{B_{R}} \iint_{B_{R+2}} \left||y|^2-2xy \right| \psi_{\alpha}( t-s-r,x-y-z)    \frac{ |v(r,z)|}{1+|z|^2}  dzdr \ dx dt  \\
& \quad   + \iint_{B_{R}} (1+|x|^2)  \left|  \iint_{B_{R+2}} \left( \psi_{\alpha}( t-s-r,x-y-z) - \psi_{\alpha}( t-r,x-z) \right) \eta_{\alpha^{-1}}(r,z)   \frac{ v(r,z)}{1+|z|^2} dz dr  \right|   dx dt \\
&\leq (h^2 + 2hR) \omega_d (R+2)^d  \int_{-R-2}^{R+2} \left\|      \frac{ |v(r,\cdot)|}{1+|\cdot|^2} \right\|_\infty dr   \\
&\quad   + (1+R^2) \iint    \left|  \iint_{B_{R+2}}   \psi_{\alpha}( t-r,x-z)  \left( \eta_{\alpha^{-1}}(r+s,z+y)   \frac{ v(r+s,z+y)}{1+|z+y|^2}  -  \eta_{\alpha^{-1}}(r,z)   \frac{ v(r,z)}{1+|z|^2}   \right)dz dr  \right|   dx dt .\\
\end{align*}
Again, the first term goes to zero as $h \to 0$, so it suffices to consider the last term.  We may bound this from above by $(1+R^2)$ times the following quantity:
\begin{align} \label{lastKRF}
&  \iint_{B_{R+2}}      \left| \eta_{\alpha^{-1}}(r+s,z+y)   \frac{ v(r+s,z+y)}{1+|z+y|^2}  -  \eta_{\alpha^{-1}}(r,z)   \frac{ v(r,z)}{1+|z|^2}   \right|dz dr     \\
&\leq    \iint_{\R^{d+1}}      \left| 1_{B_{R+2}}(r+s,z+y)  \eta_{\alpha^{-1}}(r+s,z+y)  \frac{ v(r+s,z+y)}{1+|z+y|^2}  -  1_{B_{R+2}}(r,z) \eta_{\alpha^{-1}}(r,z)   \frac{ v(r,z)}{1+|z|^2}   \right|dz dr     \nonumber
\end{align}
By the dominated convergence theorem, 
\[ 1_{R+2}(r,z) \eta_{\alpha^{-1}}(r,z)  v(r,z)/(1+|z|^2) \xrightarrow{\epsilon \to 0} 1_{R+2}(r,z) v(r,z)/(1+|z|^2) \quad \text{ in } \quad L^1(\R^{d+1}) . \]
 Thus, $\{1_{2R+2}(r,z) \eta_{\alpha^{-1}}(r,z)  v(r,z)/(1+|z|^2)\}_{\alpha  \in (0,1)}$ is relatively compact in $L^1(\Rd)$. Thus, by the converse to Kolmogorov-Riesz-Fr\'echet \cite[Corollary 4.27]{Brezis}, the right hand side of (\ref{lastKRF}) goes to zero uniformly in $\epsilon \in (0,1)$ and $(s,y) \in B_h(0)$ as $h \to 0$. 
 
\end{proof}

We now prove Corollary \ref{particlecorollary}, which shows convergence of solutions of (\ref{epsiloneqn}) with particle initial data to solutions of (\ref{PDE}), provided that the particle initial data approximates the well-prepared initial data sufficiently quickly.

\begin{proof}[Proof of Corollary \ref{particlecorollary}]
Let    $\rho_\epsilon \in AC^2_\loc([0,+\infty), \P_2(\Rd))$ be the solution of (\ref{epsiloneqn}) with initial data $\rho^0_\epsilon$ and velocity $v_\ep$. By Theorem \ref{mainthm}, there there exists $\rho \in AC^2_\loc([0,+\infty); \P_2(\Rd)) $ so that, up to a subsequence, $\rho_\epsilon(t) \to \rho(t)$ in   1-Wasserstein for all $t \geq 0$, and $\rho$ is a solution of (\ref{PDE}) with initial data $\rho^0$. Committing a mild abuse of notation, let $\rho_\epsilon$ denote this convergent subsequence.

By the stability estimate in Lemma \ref{wellposedpdeepslem}, for all $t \in [0,T]$,
\begin{align*}
W_2(\rho_\epsilon(t), \hat{\rho}_{\epsilon}(t)) \leq W_2(\rho_{\epsilon}^0, \hat\rho_{\epsilon}^0) (1+ C_\epsilon T e^{C_\epsilon T}) .
\end{align*}
This implies that $W_1(\rho_\epsilon(t), \hat{\rho}_{\epsilon}(t)) \xrightarrow{\epsilon \to 0} 0$, so   $\hat{\rho}_\epsilon(t) \to \rho(t)$ in $W_1$ for all $t \in [0,T]$.
\end{proof}

\subsection{Energy estimates and compactness}

We continue with a bound, in terms of the second moment, on the negative part of the entropy.

\begin{lem}
\label{lem:integrability via moments}
    For any  $\alpha \in \left( 0,\frac{2}{2+d}\right)$, there exists $C_{\alpha, d}>0$ so that, for all  nonnegative functions $g\in L^1(\R^d)$,
\begin{align}
\label{eq:integrability via moments}
\int_{\RR^d} g^{1-\alpha}\, dx&\leq C_{\alpha,d} \left\| g(x)(1+|x|)^2 \right\|_{L^1(\Rd)}^{1-\alpha} 
\text{ and}\\
\int_{\Rd} g(x)\log(g(x))_-\, dx &\leq C_{\alpha, d} \left( \|g\|_{L^1(\Rd)} + M_2(g)\right)^{(1-\alpha)}. \label{lem:moments_give_integrability}
\end{align}
\end{lem}
\begin{proof}
%If $\alpha =0$, the inequality is immediate, since $1 \leq (1+|x|)^p$ for any $p >0$. Thus, we may assume without loss of generality that $\alpha >0$. 
For fixed $r\in [1,2)$, H\"older's inequality  with   $p=\frac{1}{1-\alpha}$ and $q=\frac{1}{\alpha}$ implies
\[
\int_{\RR^d} g(x)^{1-\alpha}=\int_{\RR^d} g(x)^{1-\alpha}(1+|x|)^{r/2}(1+|x|)^{-r/2} \leq 
\norm{g(1+|x|)^{\frac{r}{2(1-\alpha)}}}_{1}^{1-\alpha}\norm{(1+|x|)^{-\frac{r}{2\alpha}}}_{1}^{\alpha}.
\]
 Since $\alpha < \frac{2}{2+d}$, there exists $p \in (0,2)$ so that $\alpha < \frac{p}{p+d}$. Thus, letting $r=2p(1-\alpha)$, we see that  $\frac{r}{2\alpha} = \frac{p(1-\alpha)}{\alpha}>d$. Thus, for  $C_{\alpha,d} = \norm{(1+|x|)^{-\frac{p(1-\alpha)}{\alpha}} }_{1}^{\alpha}$, we have 
 \[
 \int_{\RR^d} g^{1-\alpha}\, dx\leq C_{\alpha,d} \left\| g(x)(1+|x|)^p \right\|_{1}^{1-\alpha}. 
 \]
 Using that $p<2$, $(1+|x|)\geq 1$, and $g$ nonnegative therefore yields
 (\ref{eq:integrability via moments}). 
 
To prove (\ref{lem:moments_give_integrability}), 
note that for any $\alpha>0$, 
$\sup_{a\in (0,1)}-a^{\alpha}\log(a)\leq \frac{1}{\alpha e}$, which implies $-a\log(a)\leq \frac{a^{1-\alpha}}{\alpha e}$ for $a \in (0,1)$. 
So, for $\alpha \in \left( 0,\frac{2}{2+d}\right)$, we find, 
\begin{align*}
     \int_\Rd  g(x) (\log (g(x)))_- \, dx &\leq \frac{1}{\alpha e} \int_{ \RR^d} g(x)^{1-\alpha} \, dx\leq C_{\alpha,d}\left\| g(x)(1+|x|)^2 \right\|_{1}^{1-\alpha},
\end{align*}
where the second inequality follows from (\ref{eq:integrability via moments}). 

\end{proof}

It is a well-known fact that collections of probability measures with uniformly bounded second moment and entropy have densities that are precompact weakly in $L^1(\Rd)$. For lack of a reference, we give a proof, which combines  estimate (\ref{lem:moments_give_integrability}) from Lemma \ref{lem:integrability via moments} and the Dunford-Pettis Theorem  \cite[Theorem 4.30]{BrezisFunctionalAnalysis}.

 \begin{lem}[Variant of Dunford-Pettis]\label{dunfordpettislemma}
 Consider $\{\nu_n\}_{n \in \mathbb{N}} \subseteq \P_2(\Rd)$ with $\sup_n M_2(\nu_n)+ \S(\nu_n)<+\infty$, so that, in particular, we have  $\nu_n \ll dx$ $\forall n \in \mathbb{N}$. Then, letting $\nu_n$   denote the corresponding densities, $d \nu_n(x) = \nu_n(x) d x$, we have that the densities $\{\nu_n\}_{n \in \mathbb{N}}$ are uniformly integrable and precompact weakly in $L^1(\Rd)$.
 \end{lem}

\begin{proof}

%First, we use the bound on the second moments to control the negative part of the entropy.
 %To this end, we use the fact that for any $\alpha>0$, 
%\[
%\sup_{a\in (0,1)}-a^{\alpha}\log(a)\leq \frac{1}{\alpha e} \implies \sup_{a\in (0,1)}-a\log(a)\leq \frac{a^{1-\alpha}}{\alpha e} .
%\]
%Thus, for any $\alpha\in (0,1)$
%\[
 % \int_\Rd  \nu_n(x) (\log (\nu_n(x)))_- \, dx \leq \frac{1}{\alpha e} \int_{ \RR^d} \nu_n(x)^{1-\alpha} \, dx.
%\]
%If  $\alpha  <\frac{2}{d+2}$, then it follows from  Lemma \ref{lem:moments_give_integrability}
%that there exists a constant $C_d$ such that
%\[
 %\int_\Rd  \nu_n(x) (\log (\nu_n(x))_- \, dx \leq C_d \left(M_{2}(\nu_n)+1\right)^{1-\alpha} .
%\]
Estimate (\ref{lem:moments_give_integrability}) from Lemma \ref{lem:integrability via moments} 
applied with, say, $\alpha=1/(2+d)$, implies that there exists a constant $C_d$ such that, for all $n\in \N$,
\[
 \int_\Rd  \nu_n(x) (\log (\nu_n(x))_- \, dx \leq C_d \left(M_{2}(\nu_n)+1\right)^{1-\alpha} .
\]
Therefore, by our hypotheses that $\sup_n M_2(\nu_n) <+\infty$ and $\sup_n \S(\nu_n)<+\infty$, we have
\begin{align} \label{absentropybound} \sup_n \int_\Rd \nu(x) (\log(\nu_n(x)))_+ \, dx <+\infty . 
\end{align}

Now, we prove that $\{\nu_n\}_{n \in \mathbb{N}}$ is weakly precompact in $L^1(\Rd)$. By the Dunford-Pettis Theorem   \cite[Theorem 4.30]{BrezisFunctionalAnalysis}, it suffices to show $\{\nu_n\}_{n \in \mathbb{N}}$ is tight and uniformly integrable.
 Since   $\sup_{n \in \mathbb{N}} M_2(\nu_n)<+\infty$, we immediately obtain $\{\nu_n\}_{n \in \mathbb{N}}$ is tight: that is, for all $\delta >0$, there exists $K_\delta \subseteq   \Rd$ compact so that $\sup_{n \in \mathbb{N}} \nu_n(K_\delta^c) = \sup_{n \in \mathbb{N}} \int_{K_\delta^c}  \nu_n(x) dx< \delta$. % Thus, we conclude that $\{\nu_n\}_{n \in \mathbb{N}} \subseteq L^1(\Rd)$ is tight.

Next, we show $\{\nu_n\}_{n \in \mathbb{N}} \subseteq L^1(\Rd)$ is uniformly integrable. 
 Fix $\delta >0$ and choose $K_{\delta/2}$ as above. Since $\Phi(x) = x (\log(x))_+$ is a nonnegative, nondecreasing function with $\lim_{x \to +\infty} \Phi(x)/x=+\infty$ and inequality (\ref{absentropybound}) ensures 
 \[ \sup_n \int \Phi(|\nu_n(x)|) \, dx <+\infty ,\]
 by the theorem of de la Vall\'ee Poussin \cite[Theorem 4.5.9]{bogachev2007measure}, $\mathbf{1}_{K_{\delta/2}}(x) \nu_n(x)$ is uniformly integrable on $ \Rd$. Thus, there exists $\delta'>0$ so that, for all Borel subsets $A$ of $ \Rd$ with $\mathcal{L}^{d}(A)< \delta'$,
\[ \int_A \nu_n  = \int_{A \cap  K_{\delta/2}^c} \nu_n + \int_{ A \cap    K_{\delta/2}} \nu_n \leq \int_{K_{\delta/2}^c}   \nu_n + \int_A \mathbf{1}_{ K_{\delta/2}}(x) \nu_n(x) dx < \frac{\delta}{2}+\frac{\delta}{2} = \delta .\]
This shows $\{\nu_n\}_{n \in \mathbb{N}} \subseteq L^1( \Rd)$ is uniformly integrable, which completes the proof.

\end{proof}

Next we proof Lemma \ref{energydissipationrelationlem}, which shows the energy dissipation inequality for solutions of (\ref{epsiloneqn}).

\begin{proof}[Proof of Lemma \ref{energydissipationrelationlem}]
We aim to compute the time derivative of $\int_{\R^d}f_\ep(\mu_\ep)$, where $\mue=\phie*\rho_\ep$. To this end, we first show that $s\mapsto\mu_\ep(s,x)$ is differentiable for all $x\in \R^d$, for almost every $s>0$. To see this, we first use the representation $\rho_\ep(s)=X_\ep(s)_\#\rho^0_\ep$ as well as the change of variables formula for the pushforward to find,
\[
\mu_\ep(s,x)=(\varphi_\ep*X_\ep(s)_\#\rho^0_\ep)(x)=\int\phie(x-y)(X_\ep(s)_\#\rho^0_\ep)(y)\, dy = \int \phie(X_\ep(s,y)-x)\rho_\ep^0(y)\, dy.
\]
And, we have, for almost every $s>0$, 
\begin{align*}
\frac{d}{ds}(\phie(X_\ep(s,y)-x)\rho_\ep^0(y))&=\nabla\phie(X_\ep(s,y)-x)\cdot \partial_sX_\ep(s,y)\rho_\ep^0(y))\\
&=\nabla\phie(X_\ep(s,y)-x)\cdot(\nabla p_\ep(X_\ep(s,y))+v_\ep(X_\ep(s,y)))\rho_\ep^0(y)
\end{align*}
where the second equality follows from the definition of $X_\ep$. Now, using the fact that $\nabla\varphi_\ep$, $\nabla p_\ep$, and $v_\ep$ are bounded in $L^\infty$ (see  Lemma \ref{muepqeplem}(\ref{item:pepbdd}) and Assumption \ref{mollifieras}) yields that the expression in the previous line is integrable. Combining the two previous displayed equations thus yields that $s\mapsto\mu_\ep(s,x)$ is differentiable, with
\[
\frac{d}{ds}\mu_\ep(s,x)=\int \nabla\phie(X_\ep(s,y)-x)\cdot(\nabla p_\ep(X_\ep(s,y))+v_\ep(X_\ep(s,y)))\rho_\ep^0(y)\, dy.
\]
for almost every $s>0$. 

Since $f_\ep$ is differentiable on $[0,+\infty)$, this yields that the composition $s\mapsto f_\ep(\mu_\ep(s,x))$ is differentiable, with
\begin{equation}
\label{eq:ddt fep}
\frac{d}{ds}f_\ep(\mu_\ep(s,x))=f'_\ep(\mu_\ep(s,x))\int \nabla\phie(X_\ep(s,y)-x)\cdot(\nabla p_\ep(X_\ep(s,y))+v_\ep(X_\ep(s,y)))\rho_\ep^0(y)\, dy,
\end{equation}
for almost every $s>0$. 
The desired estimate will follow by integrating in $s$ and $x$, and interchanging the order of integration. To do this, we consider,
\[
g(s,x,y)=f'_\ep(\mu_\ep(s,x)) \nabla\phie(X_\ep(s,y)-x)\cdot(\nabla p_\ep(X_\ep(s,y))+v_\ep(X_\ep(s,y)))\rho_\ep^0(y).
\]
We want to establish that $|g|$ is integrable. We have,
\begin{align*}
\int_0^t\iint |g|\, dx\, dy\, ds&=
\int_0^t\int \rho_\ep^0(y) \left |\nabla p_\ep(X_\ep(s,y))+v_\ep(X_\ep(s,y)))\right| \int \left|f'_\ep(\mu_\ep(s,x))\right|\left| \nabla\phie(X_\ep(s,y)-x)\right|\, dx\, dy\, ds\\
&=\int_0^t\int \rho_\ep^0(y) \left |\nabla p_\ep(X_\ep(s,y))+v_\ep(X_\ep(s,y)))\right| \left(\left|f'_\ep(\mu_\ep)\right|*\left| \nabla\phie\right|\right)(X_\ep(s,y))\, dy\, ds\\
&= \int_0^t\int \rho_\ep(s,y) \left |\nabla p_\ep(s,y)+v_\ep(s,y)\right| \left(\left|f'_\ep(\mu_\ep)\right|*\left| \nabla\phie\right|\right)(s,y)\, dy\,ds
\end{align*}
where we've used the representation $\rho_\ep(s)=X_\ep(s)_\#\rho^0_\ep$ as well as the change of variables formula for the pushforward to obtain the third equality.  Our assumption $\nabla\varphi_\ep\in L^1(\rr^d)$, and the estimate of Lemma \ref{muepqeplem}(\ref{item:qepbdd}),  guarantee $\|\left|f'_\ep(\mu_\ep)\right|*\left| \nabla\phie\right|\|_{L^\infty(\R^d)}<+\infty$. Using this, together with Lemma \ref{muepqeplem}(\ref{item:pepbdd}), and our hypotheses on $v_\ep$,  we find
\begin{align}
\label{eq:|g|}
\int_0^t\iint |g|\, dx\, dy\,ds<+\infty.
\end{align}

Now, performing the same calculation but with $g$ instead of $|g|$, we find,
\begin{align*}
    \int_0^t\iint g\, dx\, dy\, ds&=
 \int_0^t\int \rho_\ep(s,y) \left(\nabla p_\ep(s,y)+v_\ep(s,y)\right) \left(f'_\ep(\mu_\ep)* \nabla\phie\right)(s,y)\, dy\,ds.
 \end{align*}
 We now note $f'_\ep(\mu_\ep)* \nabla\phie = \nabla (f'_\ep(\mu_\ep)* \phie)=\nabla p_\ep$. Using this in the previous line yields,
 \begin{align}
 \label{eq:int g}
    \int_0^t\iint g\, dx\, dy\, ds&=
 \int_0^t\int \rho_\ep(s,y) \left(\nabla p_\ep(s,y)+v_\ep(s,y)\right) \nabla p_\ep(s,y)\, dy\, ds.
 \end{align}

Next, using the definition of $\F_\ep$, followed by  (\ref{eq:ddt fep}), as well as Fubini's Theorem, we find,
\begin{align*}
    \F_\ep(\rho_\ep(t))-\F_\ep(\rho_\ep(0))&= \int_0^t \frac{d}{ds}\int f_\ep(\mu_\ep(s,x))\, dx\, ds\\
    &= \int_0^t\iint g\, dx\, dy\, ds\\
    &=
 \int_0^t\int \rho_\ep(s,y) \left(\nabla p_\ep(s,y)+v_\ep(s,y)\right) \nabla p_\ep(s,y)\, dy\, ds,
\end{align*}
where the final equality follows from (\ref{eq:int g}). Rearranging yields the result.
\end{proof}

\subsection{Basic properties of integrals of convex functions}

In the next lemma, we show that to verify that Assumption \ref{internalas}(\ref{moment_control}) holds for an energy density $f$, it suffices to consider $\rho$ that are absolutely continuous with respect to $dx$.
\begin{lem}
\label{lem:F}
Let $f$ be an energy density satisfying  Assumption \ref{internalas} items (\ref{convexlctsas}) and  (\ref{nontrivial_convex}). Suppose  (\ref{seconddersing}) of Assumption \ref{internalas}(\ref{moment_control}) is satisfied for any finite Borel measure $\rho\in \mathcal{M}(\Rd)$ with finite second moment and such that $\rho \ll dx$. Then  Assumption \ref{internalas}(\ref{moment_control}) holds for $f$. %  . 
\end{lem}
\begin{proof}
    Fix a finite Borel measure $\rho\in \mathcal{M}(\Rd)$ with finite second moment. By the definition of $\F(\rho)$ via $W_2$-lower semicontinuity \cite[equation(9.3.10)]{ambrosiogiglisavare}, we have
    \[
    \F(\rho)=\inf\left\{\liminf_{n\rightarrow\infty}\F(\rho_n) \right\},
    \]
    where the infimum is taken over sequences of finite Borel measures $\{\rho_n\}\subset \mathcal{M}(\R^d)$ with $\rho_n\ll dx$ and such that $W_2(\rho_n, \rho)\rightarrow 0$. By assumption, for any such $\rho_n$, we have
    \[
    -H(1+M_2(\rho_n))\leq \F(\rho_n).
    \]
    Since $H$ is continuous, and $W_2(\rho_n, \rho)\rightarrow 0$ implies $M_2(\rho_n)\rightarrow M_2(\rho)$, we find,
    \[
    -H(1+M_2(\rho))\leq \liminf_{n\rightarrow\infty}\F(\rho_n) \leq F(\rho),
    \]
    where the last equality follows from the definition of $\F(\rho)$.
\end{proof}

We now use Lemmas  \ref{lem:integrability via moments} and \ref{lem:F} to verify Assumption \ref{internalas} in several important cases. 
\begin{lem}[Main examples]
\label{lem:example}
    Assumption \ref{internalas} is satisfied by the internal energy densities corresponding to the heat equation, fast diffusion equations, and height constrained transport, defined, respectively, by
    \[
    f_{\rm{heat}}(s)=s\log s-s,\quad f_{\rm{fast}}(s)=\frac{1}{m-1}s^m,\quad f_{height}(s)=\iota_{[0,1]}
    \]
    for $s\geq 0$ and taking the value $+\infty$ for $s<0$, and with $m\in \left(1-\frac{2}{d+2}, 1\right)$. 
\end{lem}

\begin{proof}
    It is clear that these three energy densities satisfy Assumption \ref{internalas} items (\ref{convexlctsas}) and (\ref{nontrivial_convex}); since $f_{\rm{height}}$ is nonnegative, it is clear that it satisfies item (\ref{moment_control}) as well. Thus, it is left only to verify that item (\ref{moment_control}) holds for $f_{\rm{heat}}$ and $f_{\rm{fast}}$. By Lemma \ref{lem:F}, it suffices to consider $\rho\in \mathcal{M}(\Rd)$ with finite second moment that are absolutely continuous with respect to Lebesgue measure: $\rho=\rho(x)\, dx$.

    For  $m\in \left(1-\frac{2}{d+2}, 1\right)$, let $\alpha=1-m\in \left(0, \frac{2}{d+2}\right)$.  By (\ref{eq:integrability via moments}) of Lemma \ref{lem:integrability via moments}, we have 
    \[
    \int_{\Rd}\rho^m=\int_{\Rd}\rho^{1-\alpha}\leq C_{\alpha, d}\left(\int_{\Rd}\rho(x)(1+|x|^2)\right)^{1-\alpha}.
    \]
    Multiplying by $\frac{1}{m-1}$ yields that (\ref{seconddersing}) holds for $f_{\rm{fast}}$ with  $H(u)=u^{1-\alpha}=u^m$.

    Similarly, applying (\ref{lem:moments_give_integrability}) of Lemma \ref{lem:integrability via moments} with $\alpha = \frac{1}{d+2}$, say, we have 
\[
\int_\Rd f_{\rm{heat}}(\rho) =\int_\Rd \rho\log\rho - \int_\Rd \rho \geq -C_{d}\left(\|\rho\|_1 +M_2(\rho)\right)^{(d+1)/(d+2)}- \|\rho\|_1,
\]
which yields that (\ref{seconddersing}) holds for $f_{\rm{heat}}$.
\end{proof}

Next, we recall a classical lower semicontinuity result for varying integral functionals determined by convex internal energy densities that epi-converge. For lack of a reference, we include a proof. For further background on epi-convergence, see Rockafellar and Wets \cite[Definition 7.1]{rockafellar2009variational}.
\begin{lem}[Varying internal energy densities]\label{lem:varying_lower_semicontinuity}
Suppose that $\nu$ is a nonnegative Borel measure on $[0, +\infty)\times\RR^d$ and $E\subseteq [0, +\infty)\times\RR^d$ is a $\nu$ measurable set such that $\nu(E)<+\infty$.
 Consider functions $g, g_n:\R \to \RR\cup\{+\infty\}$ that are proper, convex, and lower semicontinuous such that $g_n$ epi-converges to $g$, ${\rm int}(\dom(g))\neq \emptyset$, and  $\interior(\dom(g^*))\neq\emptyset$.
If  $\{\phi_n\}_{n \in \mathbb{N}} \subseteq L^1(\nu)$ is a sequence of real valued functions converging weakly in $L^1(\nu)$  to a limit $\phi \in L^1(\nu)$, then
\begin{align} \label{generalliminf} \liminf_{n \to +\infty} \int_{E} g_n(\phi_n)d\nu \geq \int_E g(\phi)d\nu.
\end{align}
 \end{lem}

\begin{proof}

Let $A\subseteq \interior\left(\dom(g^*)\right)$ be a compact set and let $\eta\in L^{\infty}(\nu)$ be a function such that ${\rm image}(\eta)\subseteq A$.  The fact that $g_n$ epi converges to $g$ is equivalent to the fact that $g_n^*$ epi converges to $g^*$ and implies that   $g_n^*$ converges uniformly to $g^*$ on $A$; see, e.g. \cite[Theorem 7.17]{rockafellar2009variational}. Therefore, since $\nu(E)<+\infty$,
\[
\lim_{n\to\infty} \int_E g_n^*(\eta)d\nu=\int_E g^*(\eta)d\nu.
\]

Using the convexity of $g_n$, we have the lower bound
\[
 \int_{E} g_n(\phi_n)d\nu \geq \int_E \left(\eta\phi_n-g^*_n(\eta)\right)d\nu.
\]
Thus, taking limits, it follows that
\begin{align}\label{firstgliminf}
 \liminf_{n\to\infty} \int_{E} g_n(\phi_n)d\nu \geq \liminf_{n\to\infty}\int_E \left(\eta\phi_n-g^*_n(\eta)\right)d\nu=\int_E \left(\eta\phi-g^*(\eta)\right)d\nu.
\end{align}

Now, fix $\psi \in L^\infty(\nu)$ so that $g^*(\psi) \in L^1(\nu)$. In particular, this implies that ${\rm image}(\psi) \subseteq \dom (g^*)$, up to almost everywhere equivalence. The convexity of $g^*$ implies that $\dom(g^*)$ is an interval.  Furthermore, since  $\interior(\dom(g^*))\neq\emptyset$, we can  construct a sequence of nested compact sets $A_1\subseteq A_2 \subseteq \cdots \subseteq A_k\subseteq \cdots$ such that $A_k\subseteq \interior(\dom(g^*))$ for all $k$ and $\bigcup_{k=1}^{\infty} A_k= \dom(g^*)$. Fix $b \in \dom(g^*)$ and let
\[ \eta_k = \psi \mathbf{1}_{\{z: \psi(z) \subseteq A_k\}} + b\mathbf{1}_{\{z: \psi(z) \not \subseteq A_k\}} . \]
Then $\eta_k \to \psi$ pointwise, $|\eta_k| \leq |\psi| + |b|$, and $|g^*(\eta_k)| \leq |g^*(\psi)| + |b|$.   Thus, by the dominated convergence theorem, we have
\begin{align*}
\lim_{k \to +\infty} \int_E \left( \eta_k \phi - g^*(\eta_k)\right)d\nu = \int_E \left(\psi \phi - g^*(\psi)\right)d\nu .
\end{align*}

Finally, combining this above estimate with inequality (\ref{firstgliminf}), taking the supremum over $\psi \in L^\infty(\nu)$ with $g^*(\psi) \in L^1(\nu)$, and applying the duality theorem for convex normal integrands \cite[Theorem 2]{rockafellar1968integrals}, we obtain
\begin{align*}
 \liminf_{n\to\infty} \int_{E} g_n(\phi_n)d\nu &\geq \sup_{\{\psi : \psi \in L^\infty(\Rd),   \ g^*(\psi) \in L^1(\nu)\}} \int_E \left(\psi \phi-g^*(\psi)\right)d\nu  \\
 &= \sup_{\psi \in L^\infty(\Rd) } \int_E \left(\psi \phi-g^*(\psi)\right)d\nu  
 = \int_E g(\phi)d\nu .
\end{align*}
\end{proof}

\subsection{Interpolation inequalities and compensated compactness}

We now recall certain classical interpolation inequalities for fractional Sobolev spaces. For lack of a reference, we provide a proof.
 
\begin{lem}[Interpolation inequalities]
\label{lem:f_equi}
    Suppose  $h\in L^1(\RR^d)$  and $\nabla h\in \dot{W}^{-\theta,1}(\RR^d)$ for some $\theta\in (0,1)$. Then there exists a constant $C_{\theta}>0$ such that, for any $y\in \RR^d$,
   \[
 \norm{h}_{B_{1,\infty}^{\frac{1-\theta}{2}}(\RR^d)}  \leq C_{\theta} \norm{\nabla h}_{\dot{W}^{-\theta,1}(\RR^d)}^{1/2}\norm{h}_{L^{ 1}(\RR^d)}^{1/2} +  \|h\|_{L^1(\Rd)}  .
    \]
   
    Furthermore, for any $r\in (1,\frac{d}{d-1+\theta})$, there exists a constant $C_{d,\theta}'>0$ such that
    \[
    \norm{h}_{L^r(\RR^d)}\leq C_{d,\theta}'\norm{h}_{L^1(\R^d)}^{1-\frac{d(r-1)}{r(1-\theta)}}\norm{\nabla h}_{\dot{W}^{-\theta,1}(\RR^d)}^{\frac{d(r-1)}{r(1-\theta)}}. 
    \]
\end{lem}
\begin{proof}

Given a function $\phi\in C^{\infty}_c(\RR^d)$, let $u(t,x)=(H_t*\phi)(x)$, where $H_t$ is the heat kernel at time $t$.  It then follows that $\phi(x)=u(t,x)-\int_0^t \Delta u(s,x)\, ds$.  Fixing some time $t>0$, and using this representation, we have
\begin{align*}
&\int_{\RR^d} \phi(x) (h(x)-h(x+y)) dx\\
\quad =&\int_{\RR^d} \left(u(t,x)-\int_0^t\Delta u(s,x)\, ds \right)(h(x)-h(x+y)) dx\\
\quad=& \int_{\RR^d} u(t,x)(h(x)-h(x+y)) - \left(\int_0^t\nabla u(s,x)\cdot \nabla (h(x)-h(x+y))\, ds\, \right)  dx\\
\quad=& \int_{\RR^d} h(x)(u(t,x)-u(t,x-y))-\nabla h(x) \cdot \int_0^t \nabla (u(s,x)-u(s,x-y))\, ds\,  dx\\
\quad\leq &\norm{h}_{L^1(\RR^d)}\norm{u- u\circ \tau_{y}}_{L^{\infty}(\{t\}\times \RR^d)}+\norm{\nabla h}_{\dot{W}^{-\theta,1}(\RR^d)}\norm{\nabla u-\nabla u\circ \tau_{y}}_{L^1([0,t];\dot{C}^{\theta}(\RR^d))}
\end{align*}
where $\tau_y(x)=x-y$ is the spatial shift operator.

Note that $\frac{1+\theta}{2} - \frac{1-\theta}{2} = \theta$. Thus, given any function $g\in \dot{C}^{\frac{1+\theta}{2}}(\Rd)$, one can compute, for all $y \in \Rd$,
\begin{align*}
\norm{  g-  g\circ \tau_{y}}_{\dot{C}^{\theta}(\RR^d)} &=\sup_{x,z} \frac{ |  g(x) -   g(x+y) -   g(z) +   g(z+y)|}{|x-z|^\theta |y|^{(1-\theta)/2}} |y|^{(1-\theta)/2} \\
& = \begin{cases}
  \sup_{x,z} \frac{ |  g(x) -   g(x+y) -   g(z) +   g(z+y)|}{|x-z|^{(1+\theta)/2}} \frac{|x-z|^{(1-\theta)/2}}{ |y|^{(1-\theta)/2}} |y|^{(1-\theta)/2} &\text{ if } |x-z| < |y| ,\\
    \sup_{x,z} \frac{ |  g(x) -   g(x+y) -   g(z) +   g(z+y)|}{|x-z|^{\theta} |y|^{-\theta} |y|^{(1+\theta)/2}}  |y|^{(1-\theta)/2} &\text{ if } |x-z| \geq |y| ,
  \end{cases} \\
  & \leq 2 \| g\|_{\dot{C}^{\frac{1+\theta}{2}}(\Rd)} |y|^{(1-\theta)/2} .
\end{align*}

Therefore, returning to our previous calculation, we see that
\begin{align} 
&\int_{\RR^d} \phi(x) (h(x)-h(x+y)) dx \nonumber \\
&\quad \leq \left(\norm{h}_{L^1(\RR^d)}\norm{u}_{\dot{C}^{\frac{1-\theta}{2}}(\{t\}\times \RR^d)}+\norm{\nabla h}_{\dot{W}^{-\theta,1}(\RR^d)}\norm{\nabla u}_{L^1([0,t];\dot{C}^{\frac{1+\theta}{2}}(\RR^d))}\right)|y|^{(1-\theta)/2}.\label{namename}
\end{align}

Now we note that for any $\alpha> 0$ and $s\in [0,t]$
\begin{align*}
\norm{u}_{\dot{C}^{\alpha}(\{s\}\times \RR^d)}&=\sup_{x_1, x_0\in \RR^d} \frac{1}{|x_1-x_0|^{\alpha}}\Big|\int_{\RR^d} \phi(z)\left(H_s(z-x_1)-H_s(z-x_0)\right)\, dz\Big| \ , \\
&\leq \norm{\phi}_{L^{\infty}(\RR^d)}\sup_{x_1, x_0\in \RR^d} \int_{\RR^d}\frac{|H_s(z-x_1)-H_s(z-x_0)|}{|x_1-x_0|^{\alpha}}\, dz \ , \\
&=\norm{\phi}_{L^{\infty}(\RR^d)}\sup_{x_1, x_0\in \RR^d} \int_{\RR^d}s^{-d/2}\frac{|H_1((z-x_1)/s^{1/2})-H_1((z-x_0)/s^{1/2})|}{|x_1-x_0|^{\alpha}}\, dz \ , \\
&=s^{-\alpha/2}\norm{\phi}_{L^{\infty}(\RR^d)}\sup_{y_1, y_0\in \RR^d} \int_{\RR^d}\frac{|H_1(z'-y_1)-H_1(z'-y_0)|}{|y_1-y_0|^{\alpha}}\, dz' \ ,
\end{align*}
where the final equality follows from changing variables $z'=z/s^{1/2}$ and $y_i=x_i/s^{1/2}$.
A similar argument shows that
\begin{align}\label{eq:nabla_u_c_theta}
\norm{\nabla u}_{\dot{C}^{\alpha}(\{s\}\times \RR^d)}\leq s^{-(1+\alpha)/2}\norm{\phi}_{L^{\infty}(\RR^d)}\sup_{x_1, x_0\in \RR^d} \int_{\RR^d}\frac{|\nabla H_1(z-x_1)-\nabla H_1(z-x_0)|}{|x_1-x_0|^{\alpha}}\, dz.
\end{align}
Therefore, \[
\norm{u}_{\dot{C}^{\frac{1-\theta}{2}}(\{t\}\times \RR^d)}\lesssim t^{-\frac{1-\theta}{4}}\norm{\phi}_{L^{\infty}(\RR^d)}, \ \  
\norm{\nabla u}_{L^{1}([0,t];\dot{C}^{\frac{1+\theta}{2}}(\RR^d) )}\lesssim \norm{\phi}_{L^{\infty}(\RR^d)}\int_0^t s^{-\frac{3+\theta}{4}}ds=\frac{4}{1-\theta}t^{\frac{1-\theta}{4}}\norm{\phi}_{L^{\infty}(\RR^d)}.
\]

Hence, if we choose $t=\left(\frac{\norm{h}_{L^1(\RR^d)}}{\norm{\nabla h}_{\dot{W}^{-\theta,1}(\RR^d)}}\right)^{\frac{2}{1-\theta}}$ in inequality (\ref{namename}), we obtain
\begin{align*}
\int_{\RR^d} \phi(x) (h(x)-h(x+y)) dx  & \leq \left(\norm{h}_{L^1(\RR^d)}t^{-\frac{1-\theta}{4}} +\norm{\nabla h}_{\dot{W}^{-\theta,1}(\RR^d)}\frac{4}{1-\theta}t^{\frac{1-\theta}{4}}\right)|y|^{(1-\theta)/2} \norm{\phi}_{L^{\infty}(\RR^d)} , \\
&  \leq \frac{8}{1-\theta} \left(\norm{h}_{L^1(\RR^d)}^{1/2}\norm{\nabla h}_{\dot{W}^{-\theta,1}(\RR^d)}^{1/2}\right)|y|^{(1-\theta)/2}\norm{\phi}_{L^{\infty}(\RR^d)}.
\end{align*}
Therefore, since $C^\infty_c(\Rd)$ is weakly-* dense in $L^\infty(\Rd)$,  there exists $C_{\theta}>0$ so that 

\begin{align*}
\int_{\RR^d} |h(x)-h(x+y)|\, dx&= \sup_{\{\phi\in C^{\infty}_c(\RR^d):\norm{\phi}_{L^{\infty}(\RR^d)}\leq 1\}} \int_{\RR^d} \phi(x)(h(x)-h(x+y))\, dx \\
&\leq   C_{ \theta}  \norm{\nabla h}_{\dot{W}^{-\theta,1}(\RR^d)}^{1/2}\norm{h}_{L^1(\RR^d)}^{1/2} |y|^{\frac{1-\theta}{2}} .
\end{align*}
The Besov norm bound follows from dividing both sides by $|y|^{\frac{1-\theta}{2}}$ and taking a supremum over $y$.

For the second result, we will argue in the same way.  We can compute
\begin{align}\label{eq:f_phi}
\int_{\R^d} \phi(x)h(x)\, dx&=\int_{\R^d} h(x) \left(u(t,x)-\int_0^t \Delta u(s,x)\, ds \right)\, dx\\ \nonumber
&\leq \norm{h}_{L^1(\R^d)}\norm{u}_{L^{\infty}(\{t\}\times\R^d)}+\norm{\nabla h}_{\dot{W}^{-\theta,1}(\RR^d)} \norm{\nabla u}_{L^1([0,t];\dot{C}^{\theta}(\R^d))}.
\end{align}
Using the semigroup property of the heat equation, we can write $u(s,x)=(H_{s/2}*u(s/2,\cdot))(x)$. Thus, using the bound in (\ref{eq:nabla_u_c_theta}) with $u(s/2,\cdot)$ playing the role of $\phi$, we have the estimate
\[
\norm{\nabla u}_{\dot{C}^{\theta}(\{s\}\times \RR^d)}\lesssim  (s/2)^{-(1+\theta)/2}\norm{u}_{L^{\infty}(\{s/2\}\times \RR^d)}.
\]
Then, for any $r\in \left(1, \frac{d}{d-1+\theta}\right)$, we can use Young's convolution inequality to compute
\[
\norm{u}_{L^{\infty}(\{s\}\times\RR^d)} \leq \norm{H_s}_{L^r(\RR^d)}\norm{\phi}_{L^{\frac{r}{r-1}}(\R^d)}\lesssim_{d, \theta}  s^{-\frac{d(r-1)}{2r}}\norm{\phi}_{L^{\frac{r}{r-1}}(\R^d)}.
\]
Hence, returning to (\ref{eq:f_phi}), we have
\begin{align*}
\int_{\R^d} \phi(x)h(x)\, dx&\lesssim_{d, \theta} \left(\norm{h}_{L^1(\R^d)}t^{-\frac{d(r-1)}{2r}}+\norm{\nabla h}_{\dot{W}^{-\theta,1}(\RR^d)} \int_0^t s^{-\frac{d(r-1)+r(1+\theta)}{2r}}\, ds\right)\norm{\phi}_{L^{\frac{r}{r-1}}(\R^d)}.
\end{align*}
 The time integral is finite whenever $d(r-1)+r(1+\theta)<2r$ which is equivalent to the condition $r<\frac{d}{d-1+\theta}$. Hence,   
 \begin{align*}
\int_{\R^d} \phi(x)h(x)\, dx&\lesssim_{d, \theta} \left(\norm{h}_{L^1(\R^d)}t^{-\frac{d(r-1)}{2r}}+\norm{\nabla h}_{\dot{W}^{-\theta,1}(\RR^d)} t^{1-\frac{d(r-1)+r(1+\theta)}{2r}}\, \right)\norm{\phi}_{L^{\frac{r}{r-1}}(\R^d)}.
\end{align*}
 Choosing $t=\left(\frac{\norm{h}_{L^1(\R^d)}}{\norm{\nabla h}_{\dot{W}^{-\theta,1}(\RR^d)}}\right)^{\frac{2}{1-\theta}}$
 we get 
 \begin{align*}
\int_{\R^d} \phi(x)h(x)\, dx&\lesssim_{d, \theta} \norm{h}_{L^1(\R^d)}^{1-\frac{d(r-1)}{r(1-\theta)}}\norm{\nabla h}_{\dot{W}^{-\theta,1}(\RR^d)}^{\frac{d(r-1)}{r(1-\theta)}} \, \norm{\phi}_{L^{\frac{r}{r-1}}(\R^d)}
\end{align*}
 and the result now follows.
\end{proof}

We conclude with the following lemma, which is a straightforward adaptation of certain well-known compensated compactness results, in which one separates out smoothness in space and time; see for instance
\cite[Lemma 5.1]{lions_cc}.
\begin{lem}[Compensated compactness]\label{lem:compensated_compactness}
 Let $\{h_k\}_{k\in \mathbb{N} }\subseteq L^1_{\loc}([0,+\infty);L^1( \RR^d))\cap L^{\infty}_{\loc}([0,+\infty);L^{\infty}( \RR^d))$, $\{g_k\}_{k\in \mathbb{N}}\subseteq L^1_{\loc}([0,+\infty);L^1( \RR^d))$ be sequences of nonnegative functions that converge weakly in $L^1_{\loc}([0,+\infty)\times\RR^d)$ to limits $h\in L^1_{\loc}([0,+\infty);L^1( \RR^d))\cap L^{\infty}_{\loc}([0,+\infty);L^{\infty}( \RR^d))$ and $g\in L^1_{\loc}([0,+\infty);L^1( \RR^d))$.  If there exists $\alpha>0$ such that
\begin{align} \label{timespaceregularityas}
\sup_{k \in \mathbb{N}}  \norm{h_k}_{L^{\infty}([0,T]\times\RR^d)}+ \norm{ h_k}_{L^1([0,T];B_{1,\infty}^{\alpha}(\RR^d))}+\norm{g_k}_{H^1([0,T];W^{-1,1}(\RR^d))}<+\infty,
\end{align}
for all $T>0$,
then for any nonnegative function $\psi\in L^{\infty}_{c}([0,+\infty)\times\RR^d)$
\[
{ \lim_{k\to\infty}} \int_{[0,+\infty)\times\RR^d} h_kg_k\psi\leq \int_{[0,+\infty)\times\RR^d} hg\psi.
\]

\end{lem}
\begin{proof}
The weak convergence of $g_k$ to $g$ in $L^1_{\loc}([0,+\infty)\times\RR^d)$ along with the uniform $L^{\infty}$ bound of the $h_k$ implies that the product $h_kg_k$ is uniformly integrable on compact subsets of $ [0,+\infty)\times\RR^d$, { hence convergent weakly in $L^1_\loc([0,+\infty \times \Rd)$}. Therefore, it suffices to prove the result for $\psi\in C^{\infty}_c([0,+\infty)\times\RR^d)$.

Fix some nonnegative $\psi\in C^{\infty}_c([0,+\infty)\times\RR^d)$ and choose $T\geq 0$ such that $\spt(\psi)\subset [0,T]\times\RR^d$.
Let $\eta:\RR^d\to[0,+\infty)$ be a smooth, even, compactly supported \emph{spatial} mollifier with $\supp \eta \subseteq B_1(0)$. Denote $\eta_\delta(x)=\delta^{-d}\eta(x/\delta)$  and fix some $\delta>0$.
  Assumption (\ref{timespaceregularityas}) implies that   $\{\eta_{\delta}*g_k\}_{k \in \mathbb{N}}$ is uniformly bounded in $H^1([0,T]; W^{1,\infty}(\RR^d))$, for fixed $\delta >0$.  By Morrey's inequality, $\{\eta_{\delta}*g_k\}_{k \in \mathbb{N}}$ is uniformly bounded in $C^{1/2}([0,T]\times\RR^d)$ and is thus an equicontinuous and pointwise bounded family. Arzel\'a-Ascoli then ensures that, for any compact set $B\subset [0,T]\times\R^d$, $\{\eta_{\delta}*g_k\}_{k\in\mathbb{N}}$ converges uniformly to a limit in $C(B)$. Now the weak convergence of  $g_{k}$ to $g$ and uniqueness of limits forces the full sequence   $\eta_{\delta}*g_k$ to converge strongly in $C(B)$ to  $\eta_{\delta}*g$.  It then follows that $h_k (\eta_{\delta}*g_k)$ converges weakly in $L^1_{\loc}([0,+\infty)\times\RR^d)$ to $h (\eta_{\delta}*g)$.
Thus,
\[
{ \lim_{k\to\infty}} \int_{[0,+\infty)\times\RR^d} h_kg_k\psi= \int_{[0,+\infty)\times\RR^d} h(\eta_{\delta}*g)\psi+ { \lim_{k\to\infty}} \int_{[0,+\infty)\times\RR^d} \psi h_k(g_k-\eta_{\delta}*g_k).
\]
Using the strong convergence of $\eta_{\delta}*g$  to $g$ in $L^1_{\loc}([0,+\infty)\times\R^d)$, we obtain
\[
{\lim_{k\to\infty}} \int_{[0,+\infty)\times\RR^d} h_kg_k\psi= \int_{[0,+\infty)\times\RR^d} hg\psi+\lim_{\delta\to 0} {\lim_{k\to\infty} }\int_{[0,+\infty)\times\RR^d}  h_k(g_k-\eta_{\delta}*g_k)\psi.
\]

It remains to show that
\[
\lim_{\delta\to 0} {\lim_{k\to\infty}} \int_{[0,+\infty)\times\RR^d} h_k(g_k-\eta_{\delta}*g_k)\psi\leq 0.
\]
Moving the mollifier off of $g_k$, the integral on the previous line is equivalent to 
\[
 \int_{ [0,T] \times\RR^d} g_k\left(\psi h_k-\eta_{\delta}* (\psi h_k)\right).
\]
Let $A\subseteq [0,+\infty)\times\RR^d$ be a compact set such that $\spt(\psi)\subseteq A$. 
Fix some $M>0$ and let 
\[ E_{k,M}=\{(t,x)\in A: g_k(t,x)>M\} . \]
{
Note that the fact that $\sup_k \|g_k\|_{L^1([0,+\infty)\times \Rd)}<+\infty$ implies
\[ \sup_k |E_{k,M}| \leq \sup_k \frac{1}{M} \int_{E_{k,M}} g_k \leq \frac{1}{M} \sup_k \|g_k\|_{L^1([0,+\infty) \times \Rd)} \xrightarrow{M \to +\infty} 0. \]
}

 We can then estimate  
\begin{align*}
 \int_{[0,T]\times\RR^d} g_k\left(\psi h_k-\eta_{\delta}* (\psi h_k)\right)&\leq 2\norm{h_k\psi}_{L^{\infty}([0,T]\times\RR^d)}\norm{g_k}_{L^1(E_{k,M})}+\int_{[0,T]\times\RR^d} M\big|h_k\psi-\eta_{\delta}* (h_k\psi)\big|\\
 & 
 \leq 2\norm{h_k\psi}_{L^{\infty}([0,T]\times\RR^d)}\norm{g_k}_{L^1(E_{k,M})}+MC_\eta\delta^{\alpha}\norm{h_k\psi}_{L^1([0,T];B_{1,\infty}^{\alpha}(\R^d))}.
 \end{align*}
The Besov norm of the product $h_k\psi$ can be controlled via the simple inequality  \[
\norm{h_k\psi}_{L^1([0,T];B_{1,\infty}^{\alpha}(\R^d))}\leq \big(\norm{h_k}_{L^1([0,T];B_{1,\infty}^{\alpha}(\R^d))}\norm{\psi}_{L^{\infty}([0,T]\times\RR^d)}+\norm{h_k}_{L^{\infty}([0,T]\times\RR^d)}\norm{\psi}_{L^1([0,T];B_{1,\infty}^{\alpha}(\R^d))} \big),
\]
so $\norm{h_k\psi}_{L^1([0,T];B_{1,\infty}^{\alpha}(\R^d))}$ is uniformly bounded in $k$.
Therefore, sending $\delta\to 0$, it follows that, for any $M>0$
\begin{align*}
\lim_{\delta\to 0} {\lim_{k\to\infty}} \int_{[0,T]\times\RR^d} h_k(g_k-\eta_{\delta}*g_k) \psi \leq C_\psi \sup_{k\in \mathbb{N}}\norm{h_k}_{L^{\infty}([0,T]\times\RR^d)}\norm{g_k}_{L^1(E_{k,M})}.
\end{align*}
The weak $L^1_{\loc}([0,T]\times\RR^d)$ convergence of the $g_k$ to $g$ implies that {$g_k$ is uniformly integrable. Hence,} \[
\lim_{M\to\infty}\sup_{k\in\mathbb{N}} \norm{g_k}_{L^1(E_{k,M})}=0.
\]
Combining this with the uniform control on $\norm{h_k}_{L^{\infty}([0,T]\times\RR^d)}$ we can conclude that
\[
\lim_{M\to\infty}\sup_{k\in \mathbb{N}}\norm{h_k}_{L^{\infty}([0,T]\times\RR^d)}\norm{g_k}_{L^1(E_{k,M})}=0,
\]
which completes the argument.
\end{proof}

\bibliographystyle{abbrv}
\bibliography{Blob.bib}

\end{document}